\documentclass[
  journal=bjps,
]{cup-journal}

\usepackage{booktabs,microtype,siunitx,tabularx,relsize,appendix}
\usepackage[T1]{fontenc}
\usepackage{lmodern}
\usepackage[style=authoryear,backend=biber]{biblatex}

\newcommand{\SECO}{\mathrm{SECO}}
\newcommand{\MECO}{\mathrm{MECO}}
\newcommand{\Cov}{\mathrm{Cov}}
\newcommand{\indep}{\perp \!\!\! \perp}
\bibliography{pick.bib}

\title{High-dimensional variable clustering based on maxima of a weakly dependent random process}

\author{Alexis Boulin}
\affiliation{Université Côte d’Azur, CNRS, LJAD, France}
\affiliation{Univ Montpellier, CNRS, Montpellier, France}
\alsoaffiliation{Inria, Lemon}
\email{aboulin@unice.fr}
\author{Elena Di Bernardino}
\affiliation{Université Côte d’Azur, CNRS, LJAD, France}
\author{Thomas Lalo\"{e}}
\affiliation{Université Côte d’Azur, CNRS, LJAD, France}
\author{Gwladys Toulemonde}
\affiliation{Univ Montpellier, CNRS, Montpellier, France}
\alsoaffiliation{Inria, Lemon}

\keywords{Asymptotic independence, Consistent estimation, Extreme value theory, High dimensional models, Variable clustering.}

\begin{document}

\begin{abstract}
We propose a new class of models for variable clustering called Asymptotic Independent block (AI-block) models, which defines population-level clusters based on the independence of the maxima of a multivariate stationary mixing random process among clusters. This class of models is identifiable, meaning that there exists a maximal element with a partial order between partitions, allowing for statistical inference. We also present an algorithm depending on a tuning parameter that recovers the clusters of variables without specifying the number of clusters \emph{a priori}. Our work provides some theoretical insights into the consistency of our algorithm, demonstrating that under certain conditions it can effectively identify clusters in the data with a computational complexity that is polynomial in the dimension. A data-driven selection method for the tuning parameter is also proposed. To further illustrate the significance of our work, we applied our method to neuroscience and environmental real-datasets. These applications highlight the potential and versatility of the proposed approach.
\end{abstract}

\section{Introduction}
\paragraph{Motivation} Multivariate extremes arise when two or more extreme events occur simultaneously. These events are of prime interest to assess natural hazard, stemming from heavy rainfall, wind storms and earthquakes since they are driven by joint extremes of several of meteorological variables. Results from multivariate extreme value theory show that the possible dependence structure of extremes satisfy certain constraints. Indeed, the dependence structure may be described in various equivalent ways (\cite{beirlant2004statistics}): by the exponent measure (\cite{balkema1977max}), by the Pickands dependence function (\cite{pickands1981multivariate}), by the stable tail dependence function (\cite{huang1992statistics}), by the madogram (\cite{naveau2009modelling}, \cite{boulin2021non}), and by the extreme value copula (\cite{gudendorf2009extremevalue}).

While the modeling of univariate and low-dimensional extreme events has been well-studied, it remains a challenge to model multivariate extremes, particularly when multiple rare events may occur simultaneously. Recent research in this area has focused on connecting the study of multivariate extremes to modern statistical and machine learning techniques. The general idea of the proposed methods is to identify groups of variables that may become large without affecting the others, also referred to as extreme direction. \cite{goix2016sparse} focus on identifying extreme directions, thus providing a sparse representation of the extremal dependence. \cite{chiapino2019identifying} proposed an incremental-type algorithm for scenarios with a high number of extreme directions. \cite{janssen2020k} identify extreme directions by adapting the spherical $K$-means (sKmeans) clustering algorithm to the extremal setting and construct a nonparametric estimator for the theoretical cluster centers. Lastly, \cite{meyer2021sparse, meyer2023multivariate} frame extreme directions within what they call sparse regular variation. Our work is aligned with these directions of research as we propose a clustering algorithm for learning the dependence structure of multivariate extremes and, withal, to bridge important ideas from modern statistics and machine learning to the framework of extreme-value theory.

It is possible to perform clustering on $\textbf{X}_1,\dots,\textbf{X}_n$, where $n$ is  the number of observations of a random vector $\textbf{X} \in \mathbb{R}^d$, through two different approaches: by partitioning the set of row indices $\{1,\dots,n\}$ or by partitioning the set of column indices $\{1,\dots,d\}$. The first problem is known as the data clustering problem, while the second is called the variable clustering problem, which is the focus of this paper. In data clustering, observations are drawn from a mixture distribution, and clusters correspond to different realizations of the mixing distribution, which is a distribution over all of $\mathbb{R}^d$.

The problem of variable clustering (see, e.g., \cite{10.1214/18-AOS1794}) involves grouping similar components of a random vector $\textbf{X} = (X^{(1)},\dots,X^{(d)})$ into clusters. The goal is to recover these clusters from observations $\textbf{X}_1, \dots, \textbf{X}_n$. Instead of clustering similar observations based on a dissimilarity measure, the focus is on defining cluster models that correspond to subsets of the components $X^{(j)}$ of $\textbf{X} \in \mathbb{R}^d$. The goal is to cluster similar variables such that variables within the same cluster are more similar to each other than they are to variables in other clusters. Variable clustering is of particular interest in the study of weather extremes, with examples in the literature on regionalization (\cite{bador2015spatial,bernard2013clustering, saunders2021regionalisation}), where spatial phenomena are observed at a limited number of sites. A specific case of interest is clustering these sites according to their extremal dependencies. This can be done using techniques such as $k$-means or hierarchical clustering with a dissimilarity measure designed for extremes. However, the statistical properties of these procedures have not been extensively studied, and it is not currently known which probabilistic models on $\textbf{X}$ can be estimated using these techniques. In this paper, we consider model-based clustering, where the population-level clusters are well-defined, offering interpretability and a benchmark to evaluate the performance of a specific clustering algorithm.
	
The assumption that data are realizations of independent and identically distributed (i.i.d.) random variables is a fundamental assumption in statistical theory and modeling. However, this assumption is often unrealistic for modern datasets or the study of time series. Developing methods and theory to handle departures from this assumption is an important area of research in statistics. One common approach is to assume that the data are drawn from a multivariate stationary and mixing random process, which implies that the dependence between observations weakens over the trajectory. This assumption is widely used in the study of non-i.i.d. processes.

Our contribution is twofold. First, we develop a probabilistic setting for Asymptotic Independent block (AI-block) models to address the problem of clustering extreme values of the target vector. These models are based on the assumption that clusters of components of a multivariate random process are independent relative to their extremes. This approach has the added benefit of being amenable to theoretical analysis, and we show that these models are identifiable (see Theorem \ref{thm:unicity}). Second, we motivate and derive an algorithm specifically designed for these models (see Algorithm \ref{alg:rec_pratic}). We analyze its performance in terms of exact cluster recovery for minimally separated clusters, using a cluster separation metric (see Theorem \ref{thm:grow_dim}). The issue is investigated in the context of nonparametric estimation  over block maxima of a multivariate stationary mixing random process, where the block length is a tuning parameter.

\paragraph{Notations} All bold letters $\textbf{x}$ correspond to vectors in $\mathbb{R}^d$. Let $O = \{O_g\}_{g= 1 ,\dots, G}$ be a partition of $\{1,\dots,d\}$ into $G$ groups and let $s : \{1,\dots,d\} \rightarrow \{1,\dots,G\}$ be a variable index assignment function, thus $O_g = \{a \in {1,\dots,d} : s(a) = g\} = \{i_{g,1}, \dots, i_{g,d_g}\}$ with $d_1+\dots+d_G = d$. Using these notations, the variable $X^{(i_{g,\ell})}$ should be read as the $\ell$th element from the $g$th cluster. By considering $B \subseteq \{1,\dots,d\}$, we denote the $|B|$-subvector of $\textbf{x}$ by $\textbf{x}^{(B)} = (x^{(j)}){j \in B}$. We denote by $\textbf{X} \in \mathbb{R}^d$ a random vector with cumulative distribution function $H$ and $\textbf{X}^{(B)}$ a random subvector of $\textbf{X}$ with marginal distribution $H^{(B)}$ whose domain is $\mathbb{R}^{|B|}$. Remark that when $B = \{1,\dots,d\}$, one has $H= H^{(B)}$. Classical inequalities of vectors such as $\textbf{x} > 0$ should be understood componentwise. The notation $\delta_x$ corresponds to the Dirac measure at $x$. Let $\textbf{X}^{(O_g)}$, $g \in \{1,\dots,G\}$ be random vectors with $\textbf{X} = (\textbf{X}^{(O_1)},\dots,\textbf{X}^{(O_G)})$, we recall that $\textbf{X}^{(O_1)},\dots, \textbf{X}^{(O_G)}$ are independent if and only if $H(\textbf{x}) = \Pi_{g=1}^G H^{(O_g)}\left(\textbf{x}^{(O_g)}\right), \textbf{x} \in \mathbb{R}^d$.
	
\paragraph{Structure of the chapter} In Section \ref{sec:variable_clust}, we provide background on extreme-value theory and describe the probabilistic framework of AI-block models. We show that these models are identifiable and provide a series of equivalent characterizations. In Section \ref{sec:estimation}, we develop a new clustering algorithm for AI-block models and prove that it can recover the target partition with high probability {under mixing conditions over the random process}. We provide a process that satisfies our probabilistic and statistical assumptions in Section \ref{sec:example}. We illustrate the finite sample performance of our approach on simulated datasets in Section \ref{sec:numerical_examples}. To exemplify further motivation for our research, we applied our method to real-data from neuroscience and environmental sciences, as discussed in Section \ref{sec:application}.

\section{A model for variable clustering}
\label{sec:variable_clust}

\subsection{Background setting} 
\label{subsec:framework}
	
Consider $\textbf{Z}_{t} = (Z_t^{(1)},\dots,Z_t^{(d)})$, where $t \in \mathbb{Z}$ is a strictly stationary multivariate random process. Let $\textbf{M}m = (M_m^{(1)}, \dots, M_m^{(d)})$ be the vector of component-wise maxima, where $M^{(j)}m = \max{i=1,\dots,m} Z_i^{(j)}$. Consider a random vector $\textbf{X} = (X^{(1)},\dots,X^{(d)})$ with cumulative distribution function $H$. A normalizing function $a$ on $\mathbb{R}$ is a non-decreasing, right continuous function that goes to $\pm \infty$ as $x \rightarrow \pm \infty$. In extreme value theory (see, for example, the monograph of \cite{beirlant2004statistics}), a fundamental problem is to characterize the limit distribution $H$ in the following limit:
\begin{equation}
\label{eq:domain_attraction}
\lim_{m \rightarrow \infty} \mathbb{P}\left\{ \textbf{M}_m \leq \textbf{a}_m(\textbf{x})\right\} = H(\textbf{x}),
\end{equation}
where $\textbf{a}_m = (a_m^{(1)},\dots,a_m^{(d)})$ with $a_m^{(j)}, 1 \leq j \leq d$ are normalizing functions and $H$ is a non-degenerate distribution. Typically, $H$ is an extreme value distribution, and $\textbf{X}$ is a max-stable random vector with generalized extreme value margins. In this case, we can write:
\begin{equation*}
\mathbb{P}\left\{ \textbf{X} \leq \textbf{x}\right\} = \exp \left\{ -\Lambda(E \setminus [0,\textbf{x}]) \right\},
\end{equation*}
where $\Lambda$ is a Radon measure on the punctured cone $E = [0,\infty)^d \setminus {\textbf{0}}$. When \eqref{eq:domain_attraction} holds with $H$ an extreme value distribution, the process $(\textbf{Z}_t, t \in \mathbb{Z})$ is said to be in the max-domain of attraction of the random vector $\textbf{X}$ with cumulative distribution function $H$, denoted as $\mathcal{L}((\textbf{Z}_t, t \in \mathbb{Z})) \in \mathcal{D}(H)$, where $\mathcal{L}((\textbf{Z}_t, t \in \mathbb{Z}))$ is the law of the stationary time series $(\textbf{Z}_t, t \in \mathbb{Z})$ on $(\mathbb{R}^{d})^{\mathbb{Z}}$. In our context of a dependent process $(\textbf{Z}_t, t \in \mathbb{Z})$, the limit in \eqref{eq:domain_attraction} will in general be different from a multivariate extreme value distribution, see, e.g., \cite[Section 4.1]{bucher2014extreme}, and further conditions over the regularity (or mixing conditions, see Appendix \ref{sec:details}) are thus needed to obtain a multivariate extreme value distribution. In particular, if the random process $(\textbf{Z}_t, t \in \mathbb{Z})$ is $\beta$-mixing, then \eqref{eq:domain_attraction} holds with $H$ a multivariate extreme value distribution.

The max-domain of attraction can be described in the language of copulae. Subsequently, we assume that the marginals of $Z_{1}^{(1)},\dots,Z_{1}^{(d)}$ are continuous and we denote by $C_m$ the unique copula associated with $\textbf{M}_m$. More precisely, the max-domain of attraction condition in Equation \eqref{eq:domain_attraction} is equivalent to a max-domain of attraction condition on the levels of copulae (see Condition \ref{ass:domain} below) and a max-domain of attraction for each margin.

\begin{Assumption}{$\mathcal{A}$}
		\label{ass:domain}
		There exists a copula $C_\infty$ such that
		\begin{equation*}
			\underset{m \rightarrow \infty}{\lim} C_m(\normalfont{\textbf{u}}) = C_\infty(\normalfont{\textbf{u}}), \quad \normalfont{\textbf{u}} \in [0,1]^d.
		\end{equation*}
	\end{Assumption}
 
Specifically, when Equation \eqref{eq:domain_attraction} holds, Condition \ref{ass:domain} is satisfied, and consequently, the copula associated with $H$ is $C_\infty$. Typically, the limit $C_\infty$ is an extreme value copula, that is, the copula $C_\infty$ is max-stable $C_\infty(\textbf{u}^{1/s})^s = C_\infty(\textbf{u})$, for all $s > 0$ and it can be expressed as follows for $\textbf{u} \in [0,1]^d$:
\begin{equation*}
C_\infty(\textbf{u}) = \exp\left\{-L\left(-\ln(u^{(1)}), \dots, -\ln(u^{(d)})\right)\right\},
\end{equation*}
where $L :[0,\infty]^d \rightarrow [0,\infty]$ {is the associated} stable tail dependence function  (see \cite{gudendorf2009extremevalue} for an overview of extreme value copulae).  {However, $C_\infty$ is in general different from the extreme value copula, denoted $C_\infty^{\textrm{iid}}$, obtained when the process $(\textbf{Z}_t, t \in \mathbb{Z})$ is serially independent (see, e.g., \cite[Section 4.1]{bucher2014extreme}).}

As $L$ is an homogeneous function of order $1$, i.e., $L(a\textbf{z}) = a L(\textbf{z})$ for all $a > 0$, we have, for all $\textbf{z} \in [0,\infty)^d$,
\begin{equation*}
	L(\textbf{z}) = (z^{(1)} + \dots + z^{(d)})A(\textbf{t}),
\end{equation*}
with $t^{(j)} = z^{(j)} / (z^{(1)} + \dots + z^{(d)})$ for $j \in \{2,\dots,d\}$, $t^{(1)} = 1-(t^{(2)}+\dots+t^{(d)})$, and $A$ is the restriction of $L$ into the $d$-dimensional unit simplex, viz.
\begin{equation*}
	\Delta_{d-1} = \{(v^{(1)}, \dots, v^{(d)}) \in [0,1]^d : v^{(1)} + \dots + v^{(d)} = 1\}.
\end{equation*}
The function $A$ is known as the Pickands dependence function and is often used to quantify the extremal dependence among the elements of $\textbf{X}$. Indeed, $A$ satisfies the constraints $1/d \leq \max(t^{(1)}, \dots, t^{(d)}) \leq A(\textbf{t}) \leq 1$ for all $\textbf{t} \in \Delta_{d-1}$, with lower and upper bounds corresponding to the complete dependence and independence among maxima. For the latter, it is commonly said that the stationary random process $(\textbf{Z}_t, t \in \mathbb{Z})$ exhibits asymptotic independence, i.e., the multivariate extreme value distribution $H$ in the max-domain of attraction is equal to the product of its marginal extreme value distributions.
	
\subsection{Proposed AI-block models} 
\label{subsec:ai_block_models}
	
In this paper, our main focus is to identify disjoint groups of variables that may simultaneously be large without affecting the other groups. We thus introduce a novel class of models called AI-block models for variable clustering. These models define population-level clusters as groups of variables that exhibit dependence within clusters but extremes are independent from variables in other clusters. Formally, these variables can be partitioned into an unknown number, denoted as $G$, of clusters represented by $O = \{O_1, \dots, O_G\}$. Within each cluster, the variables display dependence, while the clusters themselves are asymptotically independent. In this section, our primary focus is on the identifiability of the model, specifically addressing the existence of a unique maximal element according to a specific partial order on the partition. We provide an explicit construction of this maximal element, which represents the thinnest partition where the desired property holds. This maximal element serves as a target for statistical inference within our framework.

In a different framework, consider $\textbf{X}^{(O_1)}, \dots, \textbf{X}^{(O_G)}$ be arbitrary random subvectors with marginal copulae $C^{(O_1)},\dots,C^{(O_G)}$ respectively. Independence between random vectors holds if and only if the underlying copula of $\textbf{X} = (\textbf{X}^{(O_1)},\dots,\textbf{X}^{(O_g)})$ is the product of the marginal copulae. This statement also holds for marginal extreme value copulae $C_\infty^{(O_1)},\dots,C_\infty^{(O_G)}$ with the property that the copula of $\textbf{X}$ is again an extreme value copula.
\begin{proposition}
	\label{prop:Cevt}
	Let $\normalfont{\textbf{X}^{(O_1)}, \dots, \textbf{X}^{(O_G)}}$ be independent extreme value random vectors with extreme value copulae $C_\infty^{(O_1)}, \dots, C_\infty^{(O_G)}$. Then the function $C_\infty$ defined as
	\begin{align*}
    	\begin{array}{lrcl}
        C_\infty : & [0,1]^d & \longrightarrow & [0,1] \\
        & \normalfont{\textbf{u}} & \longmapsto & \Pi_{g=1}^G C_\infty^{(O_g)}(u^{(i_{g,1})}, \dots, u^{(i_{g,d_g})}), \end{array}
    \end{align*}
    is an extreme value copula associated to the random vector $\normalfont{\textbf{X} = (\textbf{X}^{(O_1)}, \dots, \textbf{X}^{(O_G)})}$.
\end{proposition}
As a result, a random vector $\textbf{X}$ that exhibits (asymptotic) independence between extreme-valued subvectors therefore inherits this extreme-valued property. Using the definitions and notations so far introduced in this work, we now present the definition of our model.

\begin{definition}[Asymptotic Independent-block model]
	\label{def:AI_block_models}
	Let $(\normalfont{\textbf{Z}_t}, t \in \mathbb{Z})$ be a $d$-variate stationary random process and $\textbf{X}$ a random vector with cumulative distribution function  $H$, a multivariate extreme value distribution with copula $C_\infty$. The random process $(\normalfont{\textbf{Z}_t}, t \in \mathbb{Z})$ is said to follow an AI-block model if $\mathcal{L}((\normalfont{\textbf{Z}_t}, t \in \mathbb{Z})) \in D(H)$ and there exists a partition $O = \{O_1,\dots,O_G\}$ of $\{1,\dots,d\}$ with $\normalfont{C_\infty(\textbf{u}) = \Pi_{g=1}^G C_\infty^{(O_g)}(\textbf{u}^{(O_g)})}$.
\end{definition}
 
Notice that, when $G=1$, the definition of AI-block models thus reduces to $\mathcal{L}((\textbf{Z}_t, t \in \mathbb{Z})) \in D(H)$.

Following \cite{10.1214/18-AOS1794}, we introduce the following notation in our framework. We say that $(\textbf{Z}_t, t \in \mathbb{Z})$ follows an AI-block model with a partition $O$, denoted $\mathcal{L}((\textbf{Z}_t, t \in \mathbb{Z})) \sim O$. We define the set $\mathcal{O}((\textbf{Z}t, t \in \mathbb{Z})) = \{O: O \textrm{ is a partition of } \{1,\dots,d\} \textrm{ and } \mathcal{L}((\textbf{Z}t, t \in \mathbb{Z})) \sim O\}$, which is nonempty and finite, and therefore has maximal elements. We introduce a partial order on partitions as follows: let $O = \{O_g\}_g$ and $\{S_{g'}\}_{g'}$ be two partitions of $\{1,\dots,d\}$. We say that $S$ is a sub-partition of $O$ if, for each $g'$, there exists $g$ such that $S_{g'} \subseteq O_g$. We define the partial order $\leq$ between two partitions $O$ and $S$ of $\{1,\dots,d\}$ as follows:
\begin{equation}
	\label{eq:partial_order_partition}
	O \leq S, \textrm{ if $S$ is a sub-partition of $O$.}
\end{equation}
For any partition $O = \{O_g\}_{1 \leq g \leq G}$, we write $a \overset{O}{\sim} b$ where $a,b \in \{1,\dots,d\}$ if there exists $g \in \{1,\dots,G\}$ such that $a,b \in O_g$.
	
\begin{definition}
    \label{def:ocaps}
	For any two partitions $O, S$ of $\{1,\dots,d\}$, we define $O \cap S$ as the partition induced by the equivalence relation $a \overset{O \cap S}{\sim} b$ if and only if $a \overset{O}{\sim} b$ and $a \overset{S}{\sim} b$.
\end{definition}
Checking that $a \overset{O \cap S}{\sim} b$ is an equivalence relation is straightforward. With this definition, we have the following interesting properties that lead to the desired result, the identifiability of AI-block models.
\begin{theorem}
	\label{thm:unicity}
	Let $(\normalfont{\textbf{Z}_t}, t \in \mathbb{Z})$ be a stationary random process. The following properties hold:
	\begin{enumerate}[label=\textcolor{frenchblue}{\bf(\roman*)}]
		\item Consider $O \leq S$. Then $\mathcal{L}((\normalfont{\textbf{Z}}_t, t \in \mathbb{Z})) \sim S$ implies $\mathcal{L}((\normalfont{\textbf{Z}}_t, t \in \mathbb{Z})) \sim O$, \label{thm_(i)}
		\item $O \leq O\cap S$ and $S \leq O \cap S$, \label{thm_(ii)}
		\item $\mathcal{L}((\normalfont{\textbf{Z}}_t, t \in \mathbb{Z}))\sim O$ and $\mathcal{L}((\normalfont{\textbf{Z}}_t, t \in \mathbb{Z})) \sim S$ is equivalent to $\mathcal{L}((\normalfont{\textbf{Z}}_t, t \in \mathbb{Z})) \sim O \cap S$, \label{thm_(iii)}
		\item The set $\mathcal{O}((\textbf{Z}_t, t \in \mathbb{Z}))$ has a unique maximum $\bar{O}$, with respect to the partition partial order $\leq$ in \eqref{eq:partial_order_partition}. \label{thm_(iv)}
	\end{enumerate}
\end{theorem}
	
The proof demonstrates that, for any partition such that $\mathcal{L}((\textbf{Z}_t, t \in \mathbb{Z}))$ follows an AI-block model, there exists a maximal partition, denoted by $\bar{O}$, and its structure is intrinsic to the definition of the extreme random vector $\textbf{X}$. This partition, which represents the thinnest partition where $\mathcal{L}((\textbf{Z}_t, t \in \mathbb{Z}))$ is asymptotically independent per block, matches our expectations for a reasonable clustering target in these models. Also, a careful reading of the proof shows that this statement can also hold for the setting of mutually independent random vectors.

\subsection{Extremal dependence structure for AI-block models}
\label{subsec:struc_learn}

In extreme value theory, independence between the components $X^{(1)}, \dots, X^{(d)}$ of an random vector with extreme value distribution $H$ can be characterized in a useful way: according to \cite[Theorem 2.2]{takahashi1994asymptotic}, total independence of $\textbf{X}$ is equivalent to the existence of a vector $\textbf{p}=(p^{(1)}, \ldots, p^{(d)}) \in \mathbb{R}^d$ such that $ H(\textbf{p}) = H^{(1)}(p^{(1)}) \dots H^{(d)}(p^{(d)})$. This characterization were extended for the independence of a multivariate extreme value distribution to its multivariate marginals from \cite[Proposition 2.1]{FERREIRA2011586}, i.e., it holds that $H(\textbf{x}) = \Pi_{g=1}^G H^{(O_g)}(\textbf{x}^{(O_g)})$ for every $\textbf{x} \in \mathbb{R}^d$ if and only if there exists $\textbf{p} \in \mathbb{R}^d$ such that $0 < H^{(O_g)}(\textbf{p}^{(O_g)}) < 1$ for every $g \in \{1,\dots,G\}$ and $H(\textbf{p}) = \Pi_{g=1}^G H^{(O_g)}(\textbf{p}^{(O_g)})$. An alternative proof of this result, which involves the spectral measure, along with additional characterizations of extremal dependence structures in AI-block models, is presented in Appendix \ref{sec:aux_results_section_variable_clust}. One direct application of this result in AI-block models is that $\textbf{X}^{(O_1)}, \dots, \textbf{X}^{(O_G)}$ are independent if and only if 
	$
		L\left(1, \dots, 1\right) = \sum_{g=1}^G L^{(O_g)}\left( \textbf{1}^{(O_g)} \right).
    $ 
 
    \begin{definition}[Sum of Extremal COefficients (SECO)]
        The extremal coefficient of a random vector $\normalfont{\textbf{X}}$ with copula $C_\infty$ is defined as (see \cite{Smith1990}):
        \begin{equation}
        \label{eq:ext_coeff_chap2}
            \theta: =\theta^{(\{1,\dots,d\})} =L(1,\dots,1),
        \end{equation}
        where $L$ is the stable tail dependence function. For a partition $O = \{O_1, \dots, O_G\}$ of $\{1,\dots,d\}$, we define $\theta^{(O_g)} = \normalfont{L^{(O_g)}(\textbf{1}^{(O_g)})}$, as the extremal coefficient of the subvectors $\normalfont{\textbf{X}^{(O_g)}}$ where $d_g = |O_g|$ is the size of the set $O_g$ and $L^{(O_g)}$ is the stable tail dependence function associated to $C_\infty^{(O_g)}$. Using these coefficients, we define the following quantity $\SECO$ as
	    \begin{equation}		
            \label{eq:SECO}
		      \SECO(O) = \sum_{g=1}^G \theta^{(O_g)} - \theta. 
	    \end{equation}
    \end{definition}
 
 The extremal coefficient satisfies $1 \leq \theta \leq d$ where the lower and upper bounds correspond to the complete dependence and independence among maxima, respectively. The Sum of Extremal Coefficient (SECO) serves as a quantitative measure that assesses how much the sum of extremal coefficients for subvectors $\textbf{X}^{(O_g)}$ deviates from the extremal coefficient of the full vector $\textbf{X}$. When the SECO equals 0, it signifies that the subvectors $\textbf{X}^{(O_1)}, \dots, \textbf{X}^{(O_G)}$ form an independent partition (see \cite[Proposition 2.1]{FERREIRA2011586}). In other words, these subvectors exhibit asymptotic independence, irrespective of any underlying distributional assumptions. Therefore, the SECO, as defined in Equation \eqref{eq:SECO}, is a valuable tool for capturing the asymptotic independent block structure of the random vector $\textbf{X},$ and it offers the dual advantages of computational feasibility and being free from parametric assumptions, as discussed in Section \ref{subsec:data_driv}.
	
	Additionally, we establish a condition based on the extremal dependence of each cluster, which allows us to introduce a straightforward yet robust algorithm. This algorithm facilitates the comparison of pairwise extreme dependence between vector components, enabling us to draw informed conclusions about the dependence structures using only pairwise comparisons. It provides a practical means of assessing and quantifying the relationships among the various components of the vector, aiding in the analysis of complex high-dimensional data.
\begin{Assumption}{$\mathcal{B}$}\label{ass:A}
    For every $g \in \{1,\dots, G\}$, the extreme value random subvector $\normalfont{\textbf{X}}^{(\bar{O}_g)}$ of $\normalfont{\textbf{X}}$ where the latter is given in Definition \ref{def:AI_block_models} and $\bar{O}_g$ is the maximal element of $\mathcal{O}((\normalfont{\textbf{Z}}_t, t\in \mathbb{Z}))$ in Theorem \ref{thm:unicity}, exhibits dependence between all of its components.
\end{Assumption}
One sufficient condition to satisfy Condition \ref{ass:A} is to suppose that the exponent measure of the random subvector $\textbf{X}^{(\bar{O}_g)}$ has nonnegative Lebesgue densities on the nonnegative orthant $[0,\infty)^{d_g} \setminus \{\textbf{0}^{(\bar{O}_g)}\}$, for every $g \in \{1, \dots, G\}$ (see, e.g., \cite{engelke2020graphical} and the associated discussions). This condition implies that components within a cluster are simultaneously large. Various classes of tractable extreme value distributions satisfy Condition \ref{ass:A}. These popular models, commonly used for statistical inference,  include the asymmetric logistic model (\cite{tawn1990}), the asymmetric Dirichlet model (\cite{10.2307/2345748}), the pairwise Beta model (\cite{COOLEY20102103}) or the Hüsler Reiss model (\cite{husler1989maxima}).

\section{Consistent estimation of minimaly separated clusters}
    \label{sec:estimation}
    \subsection{Multivariate tail coefficient}
    \label{subsec:mv_tail_coef}

	Throughout this section, assume that we observe one  excerpt $\textbf{Z}_1 \dots, \textbf{Z}_n$ from a $d$-dimensional stationary random process $(\textbf{Z}_t, t \in \mathbb{Z})$  that satisfies Definition \ref{def:AI_block_models}. The sample of size $n$ of $(\textbf{Z}_t, t \in \mathbb{Z})$ is divided into $k$ blocks of length $m$, so that $k = \lfloor n/m \rfloor$, the integer part of $n/m$ and there may be a remaining block of length $n-km$. For the $i$-th block, the maximum value in the $j$-th component is denoted by
	\begin{equation*}
		M_{m,i}^{(j)} = \max \left\{ Z_t^{(j)} \, : \, t \in (im-m,im] \cap \mathbb{Z} \right\}.
	\end{equation*}
	Let us denote by $\textbf{M}_{m,i} = (M_{m,i}^{(1)}, \dots, M_{m,i}^{(d)})$ the vector of the componentwise maxima in the $i$-th block. For a fixed block length $m$, the sequence of block maxima $(\textbf{M}_{m,i})_i$ forms a stationary process that exhibits the same regularity of the process $(\textbf{Z}_t, t \in \mathbb{Z})$. The distribution functions of block maxima are denoted by
	\begin{equation*}
		F_m(\textbf{x}) = \mathbb{P}\left\{ \textbf{M}_{m,1} \leq \textbf{x} \right\}, \quad F_{m}^{(j)} (X^{(j)}) = \mathbb{P}\left\{ M_{m,1}^{(j)} \leq X^{(j)} \right\},
	\end{equation*}
	with $\textbf{x} \in \mathbb{R}^d$ and $j \in \{1,\dots,d\}$. Denote by $U_{m,1}^{(j)} = F_{m}^{(j)}(M_{m,1}^{(j)})$ the unobservable uniform margin of $M_{m,1}^{(j)}$ with $j \in \{1,\dots,d\}$. Let $C_m$ be the unique (as the margins of $\textbf{M}_{m,1}$ are continuous) copula of $F_m$.  {Then,  from  Condition \ref{ass:domain}, $C_m$ is in the domain-of-attraction of a copula $C_\infty$. By \cite[Theorem 4.2]{hsing1989extreme}, $C_\infty$ is an extreme value copula if the time series $(\textbf{Z}_t, t \in \mathbb{Z})$ is $\beta$-mixing.}
	
	One way to measure tail dependence for a $d$-dimensional extreme value random vector is through the use of the extremal coefficient, as defined in Equation \eqref{eq:ext_coeff_chap2}. According to \cite{schlather2002inequalities}, the coefficient $\theta$ can be interpreted as the number of independent variables that are involved in the given random vector. Let $x \in \mathbb{R}$ and $\theta_m(x)$ be the extremal coefficient for the vector of maxima $\textbf{M}_{m,1}$, which is defined by the following relation:
	
	\begin{equation*}
	    \mathbb{P}\left\{\bigvee_{j=1}^d U_{m,1}^{(j)} \leq x\right\} = \mathbb{P}\{U_{m,1}^{(1)} \leq x \}^{\theta_m(x)}.
	\end{equation*}
    Under Condition \ref{ass:domain}, the coefficient $\theta_m(x)$ of the componentwise maxima $\textbf{M}_{m,1}$ converges to the extremal coefficient $\theta$ of the random vector $\textbf{X}$, that is:
	\begin{equation*}
	    \theta_m(x) \underset{m \rightarrow \infty}{\longrightarrow} \theta, \quad \forall x \in \mathbb{R}.
	\end{equation*}
    It is worth noting that $\theta$ is a constant since $\textbf{X}$ is a multivariate extreme value distribution. To generalize the bivariate madogram for the random vectors $\textbf{M}_{m,1}$ we follow the same approach as in \cite{boulin2021non} and define:
	\begin{equation}
		\label{eq:theoretical_mado}
	    \nu_m = \mathbb{E}\left[ \bigvee_{j=1}^d U_{m,1}^{(j)} - \frac{1}{d} \sum_{j=1}^d U_{m,1}^{(j)} \right], \quad \nu = \mathbb{E}\left[ \bigvee_{j=1}^d H^{(j)}(X^{(j)}) - \frac{1}{d} \sum_{j=1}^d H^{(j)}(X^{(j)}) \right].
	\end{equation}
	Condition \ref{ass:domain} implies that the distribution of $\textbf{M}_{m,1}$ converges to a multivariate extreme distribution with copula $C_{\infty}$. A common approach for estimating the extremal coefficient in this scenario consists of supposing that the sample follows exactly the extreme value distribution and to consider $\theta_m(x) := \theta_m$ where the latter quantity is defined as the \textit{pre-asymptotic} extremal coefficient (see, for example, \cite{engelke2022structure} for a similar terminology) which is constant for every $x$. Thus, we have
	\begin{equation*}
	    \theta_m = \frac{1/2 + \nu_m}{1/2 - \nu_m}, \quad 1 \leq \theta_m \leq d.
	\end{equation*}
    One issue with the \emph{pre-asymptotic} extremal coefficient is that it is misspecified, as extreme value distributions only arise in the limit as the block size $m$ tends to infinity, while in practice we must use a finite sample size. We study this  misspecification  error in Section \ref{subsec:grow_dim}. A plug-in estimation process can be obtained using:
	\begin{equation}
		\label{eq:est_theta}
	    \hat{\theta}_{n,m} = \frac{1/2 + \hat{\nu}_{n,m}}{1/2 - \hat{\nu}_{n,m}},
	\end{equation}
	where $\hat{\nu}_{n,m}$ is an estimate of $\nu_m$ obtained using:
	\begin{equation}
		\label{eq:est_mado}
	    \hat{\nu}_{n,m} = \frac{1}{k} \sum_{i=1}^k \left[ \bigvee_{j=1}^d \hat{U}_{n,m,i}^{(j)} - \frac{1}{d}\sum_{j=1}^d \hat{U}_{n,m,i}^{(j)} \right],
	\end{equation}
	and $(\hat{U}_{n,m,1}^{(j)}, \dots, \hat{U}_{n,m,k}^{(j)})$ are the empirical counterparts of $(U_{m,1}^{(j)},\dots,U_{m,k}^{(j)})$ or, equivalently, scaled ranks of the sample. A data-driven method for selection the block size $m$ is still lacking in the literature. To the best of our knowledge, only \cite{10.1214/20-AOS1957} propose a method in the multivariate time series setting for selecting $m$ through bias correction using sliding-block maxima, which is out of the scope of the paper. In the following, we provide non-asymptotic bounds for the error $|\hat{\nu}_{n,m} - \nu_m|$.
	
	\begin{proposition}
	\label{prop:concentration_inequality}
		Let $(\normalfont{\textbf{Z}_t}, t \in \mathbb{Z})$ be a stationary process with algebraic $\varphi$-mixing distribution, $\varphi(n) \leq \lambda n^{-\zeta}$ where $\lambda > 0$, and $\zeta > 1$. Then the following concentration bound holds
		\begin{equation*}
			\mathbb{P}\left\{ |\hat{\nu}_{n,m} - \nu_m | \geq C_1 k^{-1/2} + C_2 k^{-1} + t \right\} \leq (d + 2\sqrt{e}) \exp \left\{ - \frac{t^2 k}{C_3} \right\},
		\end{equation*}
		where $k$ is the number of block maxima and $C_1$, $C_2$ and $C_3$ are constants depending only on $\zeta$ and $\lambda$.
	\end{proposition}
    The proof of Proposition \ref{prop:concentration_inequality}, along with all proofs of the mathematical results derived in Section \ref{sec:estimation} may be found in Appendix \ref{subsec:estimation_proofs} in the supplementary material. The non-asymptotic analysis in Proposition \ref{prop:concentration_inequality} is stringent and requires the use of $\varphi$-mixing in order to apply Hoeffding and McDiarmid inequalities in a setting where observations are not serially independent (see \cite[Section 2]{boucheron2013concentration}). However, tail bounds can also be established under $\beta$-mixing coefficients. One can also use Bernstein inequalities for $\alpha$-mixing sequences with a more stringent condition, namely exponentially decaying $\alpha$-mixing, using the main theorem in \cite{merlevede2009bernstein}.
	
	\subsection{Inference in AI-block models}
	\label{subsec:inference}
	
	In this section, we present an adapted version of the algorithm developed in \cite{10.1214/18-AOS1794} for clustering variables based on a metric on their covariances, named as CORD. Our adaptation involves the use of the extremal correlation as a measure of dependence between the extremes of two variables.  
	
	The $\SECO$ in Equation \eqref{eq:SECO} can be written in the bivariate setting as
	\begin{equation}
		\label{eq:ext_corr}
		\chi(a,b) := \SECO(\{a,b\}) = 2 - \theta(a,b),
	\end{equation}
	where for notational convenience, $\theta(a,b) := \theta^{(\{a,b\})}$ is the bivariate extremal coefficient between $X^{(a)}$ and $X^{(b)}$ as defined in Equation \eqref{eq:ext_coeff_chap2}. In fact, the bivariate $\SECO$ is exactly equal to the extremal correlation $\chi$ defined in \cite{coles1999dependence}. This metric has a range between $0$ and $1$, with the boundary cases representing asymptotic independence and  comonotonic extremal dependence, respectively. In an AI-block model, the statement
	\begin{equation*}
		\textbf{X}^{(O_g)} \indep \textbf{X}^{(O_h)}, \quad g \neq h,
	\end{equation*}
	is equivalent to
	\begin{equation}
		\label{eq:indep_ext_corr}
		\chi(a,b) = \chi(b,a) = 0, \quad \forall a \in O_g, \forall\, b \in O_h, \quad g \neq h.
	\end{equation}
	Thus using Condition \ref{ass:A} and Equation \eqref{eq:indep_ext_corr}, where the first condition can be equivalently stated using extremal correlation as:
	 \begin{equation*}
 {a \overset{\bar{O}}{\sim} b \implies \chi(a,s) > 0, \; \chi(b,s) > 0, \textrm{ where } s \in \{1,\dots,d\} \textrm{ such that } a \overset{\bar{O}}{\sim} s \textrm{ and } b \overset{\bar{O}}{\sim} s},
    \end{equation*}
    the extremal correlation is a sufficient statistic to recover clusters in an AI-block model. Indeed, Equation \eqref{eq:indep_ext_corr} reveals:  			
    \begin{equation*} 
    	a \overset{\bar{O}}{\not \sim} b \implies \chi(a,b) = 0. 
   	\end{equation*}
 Consequently, in an AI-block model, two variables $X^{(a)}$ and $X^{(b)}$ are considered part of the same cluster under Condition \ref{ass:A} if and only if $\chi(a,b) > 0$.
    For the estimation procedure, using tools introduced in the previous section, we give a sample version of the extremal correlation associated to $M_{m,1}^{(a)}$ and $M_{m,1}^{(b)}$ by
	\begin{equation*}
		\hat{\chi}_{n,m}(a,b) = 2 - \hat{\theta}_{n,m}(a,b), \quad  a,b \in \{1,\dots, d\},
	\end{equation*}
	where $\hat{\theta}_{n,m}(a,b)$ is the sampling version defined in \eqref{eq:est_theta} of $\theta(a,b)$. With some technical arguments, a concentration result estimate follows directly from Proposition \ref{prop:concentration_inequality}.  
 
    We can represent the matrix of all extremal correlations as $\mathcal{X} = [\chi(a,b)]_{a = 1,\dots,d, b=1,\dots,d}$. Additionally, we introduce its empirical counterpart, denoted as $\hat{\mathcal{X}}$. This version, $\hat{\mathcal{X}}$ incorporates elements $\hat{\chi}_{n, m}(a, b)$ for pairs $(a, b) \in \{1, \dots, d\}^2$. We present an algorithm, named ECO (Extremal COrrelation), which estimates the partition $\bar{O}$ using a dissimilarity metric based on the extremal correlation. This algorithm, outlined in Algorithm \ref{alg:rec_pratic}, does not require the specification of the number of groups $G$, as it is automatically estimated by the procedure. The algorithm complexity for computing the $k$ vectors $\hat{\textbf{U}}_{n,m,i}= (\hat{U}_{n,m,i}^{(1)}, \dots, \hat{U}_{n,m,i}^{(d)})$ for $i \in \{1,\dots,k\}$ is of order $O(d k \ln(k))$. Given the empirical ranks, computing $\hat{\mathcal{X}}$ and performing the algorithm require $O(d^2\vee d k \ln(k))$ and $O(d^3)$ computations, respectively. So the overall complexity of the estimation procedure is $O(d^2 (d \vee  k \ln (k))))$.
	\begin{algorithm}
    \renewcommand{\thealgorithm}{(ECO)}
	
	\caption{Clustering procedure for AI-block models}

\begin{algorithmic}[1]
\Procedure{ECO}{$S$, $\tau$, $\hat{\mathcal{X}}$}
    \State Initialize: $S = \{1,\dots,d\}$, $\hat{\chi}_{n,m}(a,b)$ for $a,b \in \{1,\dots,d\}$ and $l = 0$
    \While{$S \neq \emptyset$}
    		\State $l = l +1$
    		\If{ $|S| = 1$}
    			\State $\hat{O}_l = S$
    		\EndIf
    		\If{$|S| > 1 $}
    			\State $(a_l, b_l) = \arg \underset{a,b \in S}{\max} \, \hat{\chi}_{n,m}(a,b)$
    			\If{$\hat{\chi}_{n,m}(a_l,b_l) \leq \tau$}
    				\State $\hat{O}_l = \{a_l\}$
    			\EndIf
    			\If{$\hat{\chi}_{n,m}(a_l,b_l) > \tau$}
    				\State $\hat{O}_l = \{ s \in S : \hat{\chi}_{n,m}(a_l,s) \wedge \hat{\chi}_{n,m}(b_l,s) \geq \tau \}$
    			\EndIf
    		\EndIf
    		\State $S = S \setminus \hat{O}_l$
    	\EndWhile
    \State return $\hat{O} = (\hat{O}_l)_l$
\EndProcedure
\end{algorithmic}
\label{alg:rec_pratic}
\end{algorithm}

	In \ref{subsec:asympto}, we provide conditions under the regularity of the process ensuring that our algorithm is asymptotically consistent. These conditions involve $\beta$-mixing coefficients which are less stringent than $\varphi$-mixing used in the next section. Unlike in asymptotic analysis where the choice of the threshold becomes trivial, in a non-asymptotic framework, the algorithm's performance is influenced by the parameter $\tau$. In a non-asymptotic framework, when $\tau \approx 0$, the algorithm is prone to identifying the sole cluster as $\{1,\dots,d\}$, while a value of $\tau \approx 1$ suggests that the algorithm is likely to return the largest partition $\{\{1\},\dots,\{d\}\}$. Thus, the parameter $\tau$ serves as a threshold that determines the algorithm's tolerance to differentiate between the noise in the inference and the signal indicating asymptotic dependence. This discriminatory capability depends on factors such as the sample size $n$, the dimension $d$, and the proximity between the sub-asymptotic framework and the maximum domain of attraction. Consequently, selecting an appropriate threshold $\tau$ becomes a critical consideration. However, this challenge can be addressed through a non-asymptotic analysis of the algorithm, which we will discuss in the following section.
 
	\subsection{Estimation in growing dimensions}
	\label{subsec:grow_dim}
	
	We provide consistency results for our algorithm, allowing estimation in the case of growing dimensions, by adding non asymptotic bounds on the probability of consistently estimating the maximal element $\bar{O}$ of an AI-block model. Furthermore, this result provides an answer for how to leverage $\tau$ in Algorithm \ref{alg:rec_pratic}. The difficulty of clustering in AI-block models can be assessed via the size of the Minimal Extremal COrrelation ($\MECO$) separation between two variables in a same cluster:
	\begin{equation*}
		\label{eq:MECO}
		\MECO(\mathcal{X}) := \underset{a \overset{\bar{O}}{\sim} b}{\min} \, \chi(a,b).
	\end{equation*}
	In AI-block models, with Condition \ref{ass:A}, we always have $\MECO(\mathcal{X}) > \eta$ with $\eta = 0$. However, a large value of $\eta$ will be needed for retrieving consistently the partition $\bar{O}$ stationary observations. We are now ready to state the main result of this section.
	  
\begin{theorem} \label{thm:grow_dim}
    We consider $(\normalfont{\textbf{Z}_t}, t \in \mathbb{Z})$ be a $d$-multivariate stationary process following a AI-block model given in Definition \ref{def:AI_block_models} satisfying Condition \ref{ass:A} and algebraic $\varphi$-mixing distribution, $\varphi(n) \leq \lambda n^{-\zeta}$ where $\lambda > 0$ and $\zeta > 1$ Define
        
    \begin{equation*}
            d_m = \underset{a \neq b}{\max} \left| \chi_{m}(a,b) - \chi(a,b) \right|.
        \end{equation*}
        
        Let $(\tau, \eta)$ be parameters fulfilling
		
        \begin{align*}
		&\tau \geq d_m + C_1 k^{-1/2} + C_2 k^{-1} + C_3 \sqrt{\frac{(1+\gamma) \ln(d)}{k}}, \\ &\eta \geq d_m + C_1 k^{-1/2} + C_2 k^{-1} + C_3 \sqrt{\frac{(1+\gamma) \ln(d)}{k}} + \tau,
		\end{align*}
        
        where $C_1, C_2, C_3$ are universal constants depending only on $\lambda$ and $\zeta$, $k$ is the number of block maxima, and $\gamma > 0$. For a given $\mathcal{X}$ and its corresponding estimator $\hat{\mathcal{X}}$, if $\MECO(\mathcal{X}) > \eta$, then the output of Algorithm \ref{alg:rec_pratic} is consistent, i.e.,
		
        \begin{equation*}
			\mathbb{P}\left\{ \hat{O} = \bar{O} \right\} \geq 1-2(1+\sqrt{e})d^{-2\gamma}.
		\end{equation*}
    \end{theorem}
	  
	The analysis of Algorithm \ref{alg:rec_pratic} can be separated into two distinct components: an analytic part that provides conditions ensuring $\hat{O} = \bar{O}$, as detailed in Lemma \ref{lem:exact_recovery}, and a stochastic part that deals with concentration results for $\hat{\chi}_{n,m}$ in Proposition \ref{prop:concentration_inequality}, directly stated in the proof of Theorem \ref{thm:grow_dim}. In Section \ref{sec:example}, we provide an example of a mixing process that satisfies all the conditions stated in Theorem \ref{thm:grow_dim}. As Theorem \ref{thm:grow_dim} is not concerned with asymptotics, we did not actually assume Condition \ref{ass:domain}. A link between $\textbf{M}_m$ and $\textbf{X}$ is implicitly  provided through the bias term $d_m$ which measures the distance between $\chi_m(a,b)$ and $\chi(a,b)$. This quantity vanishes when Condition \ref{ass:domain} holds as $m \rightarrow \infty$.
	 
	Some comments on the implications of Theorem \ref{thm:grow_dim} are in order. On a high level, larger dimension $d$ and {bias} $d_m$ lead to a higher threshold $\tau$. The effects of the dimension $d$ and the {bias} $d_m$ are intuitive: larger dimension or more {bias} make the partition recovery problem more difficult. It is clear that the partition recovery problem becomes more difficult as the dimension or bias  increases. This is reflected in the bound of the $\MECO$ value below which distinguish between noise and asymptotic independence is impossible by our algorithm. Thus, whereas the dimension $d$ increases, the dependence between each component should be stronger in order to distinguish between the two. In other words, for alternatives that are sufficiently separated from the asymptotic independence case, the algorithm will be able to distinguish between asymptotic independence and noise at the $\sqrt{\ln(d)k^{-1}}$ scale. For a more quantitative discussion, our algorithm is able to recover clusters when the data dimension scales at a polynomial rate, i.e., $d = o(n^{p}),$ with $p > 0$ as $\eta$ in Theorem \ref{thm:grow_dim} decreases with increasing $n$.

 {The order of the threshold $\tau$ involves known quantity such as $d$ and $k$ and a unknown parameter $d_m$. For the latter, there is no simple manner to choose optimally this parameter, as there is no simple way to determine how fast is the convergence to the asymptotic extreme behavior, or how far into the tail the asymptotic block dependence structure appears. In particular, Condition \ref{ass:domain} does not contain any information about the rate of convergence of $C_m$ to $C_\infty$. More precise statements about this rate can be made with second order conditions. Let a regularly varying function $\Psi :\mathbb{N} \rightarrow (0,\infty)$ with coefficient of regular variation $\rho_{\Psi} < 0$ and a continuous non-zero function $S$ on $[0,1]^d$ such that
    \begin{equation}
        \label{eq:second_order}
        C_m(\textbf{u}) - C_\infty(\textbf{u}) = \Psi(m) S(\textbf{u}) + o(\Psi(m)), \quad  \mbox{ for } m \rightarrow \infty,
    \end{equation}
    uniformly in $\textbf{u} \in [0,1]^d$ (see, e.g., \cite{bucher2019second, 10.1214/20-AOS1957} for a proper introduction to this condition). In this case, we can show that $d_m = O(\Psi(m))$. In the typical case $\Psi(m) = c\,t^{\rho_{\Psi}}$ with $c > 0$, choosing $m$ proportional to $n^{1/(1-\rho_\Psi)}$ leads to the optimal convergence rate $n^{\rho_\Psi / (1-2\rho_\Psi)}$ (see \cite{drees1998best}). However, there is no simple way to know in advance or infer the value of $\rho_\Psi$ and, in practice, it is advisable to use a data-driven procedure to select the threshold.
    }
	
\subsection{Data-driven selection of the threshold parameter}
\label{subsec:data_driv}

The performance of Algorithm \ref{alg:rec_pratic} depends crucially on the value of the threshold parameter $\tau$. This  threshold involves known quantities such as $d$ and $k$ and a unknown parameter $d_m$ (see  Theorem \ref{thm:grow_dim}). For the latter, there is no simple manner to choose optimally this parameter, as there is no simple way to determine how fast is the convergence to the asymptotic extreme behavior, or how far into the tail the asymptotic block dependence structure appears. Second order conditions, which are commonly used in the literature to ensure convergence to the stable tail dependence function at a certain rate, are theoretically relevant (see \cite{dombry2019maximum, einmahl2012mestimator, fougeres2015bias}for examples). However, finding the optimal value for the block length parameter remains a challenging task.

In practice, it is advisable to use a data-driven procedure to select the threshold in Algorithm \ref{alg:rec_pratic}. The idea is to use the $\SECO$ criteria presented in Equation \eqref{eq:SECO}. Let $\mathcal{L}((\textbf{Z}_t, t \in \mathbb{Z})) \sim O$, given a partition $\hat{O} = \{ \hat{O}_g \}_g$, we know from \cite{FERREIRA2011586} that the $\SECO$ similarity given by \eqref{eq:SECO} is equal to $0$ if and only if $\hat{O} \leq \bar{O}$. We thus construct a loss function given by the $\SECO$ where we evaluate its value over a grid of the $\tau$ values. The value of $\tau$ for which the $\SECO$ similarity has minimum values is also the value of $\tau$ for which we have consistent recovery of our clusters. The based estimator of the $\SECO$ in \eqref{eq:SECO} is thus defined as
\begin{equation}
    \label{eq:seco_cv}
    \widehat{\SECO}_{n,m}(\hat{O}) = \sum_{g} \hat{\theta}^{(\hat{O}_g)}_{n,m} - \hat{\theta}_{n,m}.
\end{equation}
Let $\widehat{\mathcal{O}}$ be a collection of partitions computed with Algorithm \ref{alg:rec_pratic}, by varying $\tau$ around its theoretical optimal value, of order $(d_m + \sqrt{\ln(d) k^{-1}})$, on a fine grid. For any $\hat{O} \in \widehat{\mathcal{O}}$, we evaluate $\widehat{\SECO}_{n,m}$ in \eqref{eq:seco_cv}. In practice, the $\widehat{\SECO}(\hat{O})$ could be minimal for several values of $\tau$. For example, if we incorrectly group all the components of the random vector into a single cluster. Therefore, we recommend retaining the partition obtained for the minimal value of $\widehat{\SECO}(\hat{O})$ associated with the largest parameter $\tau$, which results in the thinnest partition of the variables of the random vector. Proposition \ref{prop:cv} offers theoretical support for this procedure.
\begin{proposition}
\label{prop:cv}  
 We consider $(\normalfont{\textbf{Z}_t}, t \in \mathbb{Z})$ to be a $d$-multivariate stationary process following an AI-block model given in Definition \ref{def:AI_block_models} with algebraic $\varphi$-mixing distribution, $\varphi(n) \leq \lambda n^{-\zeta}$ where $\lambda > 0$ and $\zeta > 1$. Let $\bar{O} = \{\bar{O}_1,\dots,\bar{O}_G\}$ be the thinnest partition given by Theorem \ref{thm:unicity} with corresponding sizes $d_1,\dots,d_G$. Let $\hat{O} = \{ \hat{O}_1,\dots,\hat{O}_I\}$ be any partition of $\{1,\dots,d\}$ with corresponding sizes $d_1,\dots,d_I$. Define
\begin{equation*}
    D_m = \max \left\{ \left|\sum_{g=1}^G \theta_m^{(\bar{O}_g)} - \sum_{g=1}^G \theta^{(\bar{O}_g)} \right|, \left|\sum_{i=1}^I \theta_m^{(\hat{O}_i)} - \sum_{i=1}^I \theta^{(\hat{O}_i)} \right|\right\},
\end{equation*}
Then, there exists a constant $c > 0$, such that, if $\hat{O} \not\leq \bar{O}$ and
\begin{equation}
    \label{eq:bound_cv}
    \SECO(\hat{O}) > 2\left(D_m + c\sqrt{\frac{\ln(d)}{k}} \max(G,I) \max(\vee_{g=1}^G d_g^2, \vee_{i=1}^I d_i^2) \right), 
\end{equation}
it holds that
\begin{equation*}
    \mathbb{E}[ \widehat{\SECO}_{n,m}(\bar{O})] < \mathbb{E}[\widehat{\SECO}_{n,m}(\hat{O})].
\end{equation*}
\end{proposition}
However, the bound presented in Equation \eqref{eq:bound_cv} is overly pessimistic since it exhibits polynomial growth with respect to cluster sizes. Nevertheless, when we consider the scenario where $n \rightarrow \infty$ with $d$ fixed, then under Condition \ref{ass:domain}, this condition simplifies to $\SECO(\hat{O}) > 0$, which holds true for every $\hat{O} \not\leq \bar{O}$ (see Appendix \ref{prop:ineq} in the supplementary material). Therefore, despite the pessimistic nature of this bound, the asymptotic relevance of choosing the threshold parameter based on data-driven approaches remains intact. Additionally, numerical studies provide support for the effectiveness of $\SECO$ as an appropriate criterion for determining the threshold parameter for a suitable number of data and for important cluster sizes (see Section \ref{sec:numerical_examples}). Furthermore, we establish the weak convergence of an estimator for $\SECO(O)$ when $\mathcal{L}((\textbf{Z}_t, t \in \mathbb{Z})) \sim O$ (we refer to Appendix \ref{subsec:functional_central_limit_theorem} for detailed information).

\section{Hypotheses discussion for a multivariate random persistent process}
\label{sec:example}
  
A trivial example of an AI-block model is given by a partition $O$ such that $\mathcal{L}((\textbf{Z}_t^{(O_g)}, t \in \mathbb{Z})) \in D(H^{(O_g)})$ for $g \in \{1,\dots,G\}$ and $\mathcal{L}((\textbf{Z}_t^{(O_1)}, t \in \mathbb{Z})),\dots,\mathcal{L}((\textbf{Z}_t^{(O_G)}, t \in \mathbb{Z}))$ are independent. In this simple model, the peculiar dependence structure under study is not inherent of large values of the stationary law of the process.  
 
More interestingly, in this section we will focus on a process where the dependence between clusters disappears in the distribution tails. To this aim, we recall here a $\varphi$-algebraically mixing process. The interested reader is referred for instance to \cite{bucher2014extreme}. We show that Conditions \ref{ass:domain} and \ref{ass:A} hold with a bit more work.
 
 Let $D$ denote a copula and consider i.i.d $d$-dimensional random vectors $\boldsymbol{Z}_0, \boldsymbol{\xi}_1, \boldsymbol{\xi}_2, \dots$ from $D$ and  independent  Bernoulli  random variables $I_1, I_2, \dots$ i.i.d. with $\mathbb{P}\{I_t = 1\} = p \in (0,1]$. For $t = 1,2, \dots$, define the stationary random process $(\textbf{Z}_t, t\in \mathbb{Z})$ by
\begin{equation}
	\label{eq:random_repetition}
    \textbf{Z}_t = \boldsymbol{\xi}_t \delta_{1}(I_t) + \boldsymbol{Z}_{t-1} \delta_{0}(I_t),
\end{equation}
where we suppose without loss of generality that the process is defined for all $t \in \mathbb{Z}$ using stationarity. The persistence of the process $(\textbf{Z}_t, t \in \mathbb{Z})$ arises from repeatable values in \eqref{eq:random_repetition}. From this persistence, $(\textbf{Z}_t, t \in \mathbb{Z})$ is $\varphi$-mixing with coefficient of order $O((1-p)^n)$ \cite[Lemma B.1]{bucher2014extreme}, hence algebraically mixing. 

 Assuming that the copula $D$ belongs to the (i.i.d.) copula domain of attraction of an extreme value copula $D_\infty^{(iid)}$, denoted as
\begin{equation*}
    D_m(\textbf{u})= \{D(\textbf{u}^{1/m})\}^m \longrightarrow D_\infty^{(iid)}(\textbf{u}), \quad (m \rightarrow \infty).
\end{equation*}
 {Here, $D_m$ represents the copula of the componentwise block maximum of size $m$ based on the serially independent sequence $(\boldsymbol{\xi}_t, t \in \mathbb{N})$.}

 According to \cite[Proposition 4.1]{bucher2014extreme}, if $C_m$ denotes the copula of the componentwise block maximum of size $m$ based on the sequence $(\textbf{Z}_t, t \in \mathbb{N})$, then
\begin{equation*}
    C_m(\textbf{u}) \underset{m \rightarrow \infty}{\longrightarrow} D_\infty^{(iid)}(\textbf{u}), \quad \textbf{u} \in [0,1]^d.
\end{equation*}
This implies that Condition \ref{ass:domain} is satisfied.

 Consider the multivariate outer power transform of a Clayton copula with parameters $\theta > 0$ and $\beta \geq 1$, defined as:
\begin{equation*}
    D(\textbf{u}; \theta, \beta) = \left[ 1 + \left\{ \sum_{j=1}^d (\{u^{(j)}\}^{-\theta} -1)^\beta \right\}^{1/\beta} \right]^{-1/\theta}, \quad \textbf{u} \in [0,1]^d.
\end{equation*}
 The copula of multivariate componentwise maxima of an i.i.d. sample of size $m$ from a continuous distribution with copula $D(\boldsymbol{\cdot};\theta,\beta)$ is given by:

\begin{equation}
    \label{eq:clayton_powerm}
\left\{ D\left(\{u^{(1)}\}^{1/m}, \dots, \{u^{(d)}\}^{1/m} ; \theta, \beta\right) \right\}^m = D\left(u^{(1)},\dots,u^{(d)}; \theta/m, \beta\right),
\end{equation}

 As $m \rightarrow \infty$, this copula converges to the Logistic copula with shape parameter $\beta \geq 1$:
\begin{equation*}
    D_\infty^{(iid)}(\textbf{u}) = D(\textbf{u} ; \beta) = \underset{m \rightarrow \infty}{\lim} \, D\left(u^{(1)},\dots,u^{(d)} ; \theta/m, \beta\right) = \exp \left[ - \left\{ \sum_{j=1}^d (- \ln u^{(j)})^\beta \right\}^{1/\beta} \right],
\end{equation*}
 uniformly in $\textbf{u} \in [0,1]^d$. This result, originally stated in \cite[Proposition 4.3]{bucher2014extreme} for the bivariate case, can be extended to an arbitrary dimension without further arguments. Now, consider the following nested Archimedean copula given by:
\begin{equation}
    \label{eq:NAC}
    D\left(D^{(O_1)}(\textbf{u}^{(O_1)} ; \theta, \beta_1), \dots, D^{(O_G)}(\textbf{u}^{(O_g)} ; \theta, \beta_G) ; \theta,\beta_0\right).
\end{equation}
We aim to show that this copula is in the domain of attraction of an AI-block model. That is the purpose of the proposition stated below.
\begin{proposition}
    \label{prop:domain_attraction_repetition_model}
    Consider $1 \leq \beta_0 \leq \min\{\beta_1 ,\dots, \beta_G\}$, then the nested Archimedean copula given in \eqref{eq:NAC} is in the copula domain of attraction of an extreme value copula given by
    \begin{equation*}
    	D\left(D^{(O_1)}(\normalfont{\textbf{u}^{(O_1)}} ; \beta_1), \dots, D^{(O_G)}(\normalfont{\textbf{u}^{(O_G)}}; \beta_G) ; \beta_0\right).
	\end{equation*}
	In particular, taking $\beta_0 = 1$ gives an AI-block model where extreme value random vectors $\normalfont{\textbf{X}^{(O_g)}}$ correspond to a Logistic copula with parameter shape $\beta_g$.
\end{proposition}

From the last conclusion of Proposition \ref{prop:domain_attraction_repetition_model}, we obtain Condition \ref{ass:domain}, that is $(\textbf{Z}_t, t \in \mathbb{Z})$ in \eqref{eq:random_repetition} is in max-domain of attraction of an AI-block model. Noticing that the exponent measure of each cluster is absolutely continuous with respect to the Lebesgue measure, Condition \ref{ass:A} is thus valid.

\begin{remark}\label{dmiid}
 {Notice that, using results from \cite{bucher2014extreme, 10.1214/20-AOS1957}, in the i.i.d. case, i.e. $p=1$, there exists an auxiliary function $\Psi_D$ for $D_m$ with $\Psi_D(m) = O(m^{-1})$. By using considerations after Equation \eqref{eq:second_order}, we thus obtain $d_m = O(m^{-1})$.}
\end{remark}

\section{Numerical examples}
\label{sec:numerical_examples}

\subsection{Numerical results}
\label{sec:num}
In this section, we investigate the finite-sample performance of our algorithm to retrieve clusters in AI-block models. The results in this section can be reproduced using the code made available at \href{https://github.com/Aleboul/ai_block_model}{https://github.com/Aleboul/ai\_block\_model}. We consider a number of AI-block models of increasing complexity. We design three resulting partitions in the limit model {$C_\infty$:}
\begin{enumerate}[label = \textcolor{frenchblue}{E\arabic*}, leftmargin=*]
	\setcounter{enumi}{0}
	\item  {$C_\infty$} is composed of two blocks $O_1$ and $O_2$, of equal lengths where  {$C^{(O_1)}_\infty$} and  {$C^{(O_2)}_\infty$} are Logistic extreme value copulae with parameters set to $\beta_1 = \beta_2 =10/7$. \label{exp:E1}
	\item {$C_\infty$} is composed of $G=5$ blocks of random sample sizes $d_1,\dots,d_5$ from a multinomial distribution with  parameter $q_g = 0.5^g$ for $g \in \{1,\dots,4\}$ and $q_5 =1- \sum_{g=1}^{4}q_g$. Each random vector is distributed according to a Logistic distribution where parameters $\beta_g = 10/7$ for $g \in \{1,\dots,5\}$. \label{exp:E2}
	\item We consider the same model as \ref{exp:E2} where we add $5$ singletons. Then we have $10$ resulting  clusters. Model with singletons are known to be the hardest model to recover in the clustering literature.\label{exp:E3}
\end{enumerate}

{We consider here observations from the model described in Equation \eqref{eq:random_repetition} n Section \ref{sec:example}. Here, the copula $D$ is derived from a nested Archimedean copula, as indicated in Equation \eqref{eq:NAC}. Specifically, the outer Power Clayton copula with a parameter $\beta_0 = 1$ serves as the ``mother'' copula, while the outer Power Clayton copula with parameters $\beta_1 = \dots = \beta_G = 10/7$ act as the ``child'' copulae. It is worth noting that the copula $D_m$ does not fall under the category of an extreme value copula. This can be observed by considering two observations, $u^{(i)}$ and $u^{(j)}$, belonging to the same cluster $O_1$. In this case, the nested Archimedean copula presented in Equation \eqref{eq:NAC} takes the following form:}

\begin{equation*}
    D^{(O_1)}(\textbf{1},u^{(i)},u^{(j)},\textbf{1} ; \theta,\beta_1),
\end{equation*}
 {where the margins for the indices outside of $i$ and $j$ are considered as $1$. Consequently, the dependence is determined by an outer Power Clayton copula that does not exhibit max-stability. Similarly, when $i$ and $j$ belong to different clusters, the nested Archimedean copula in Equation \eqref{eq:NAC} follows the expression:}
\begin{equation*}
    D(\textbf{1},u^{(i)},u^{(j)},\textbf{1} ; \theta,1),
\end{equation*}
 representing a Clayton copula. It is worth noting that indices in different clusters exhibit dependence when the max-domain of attraction is not yet reached. This framework is particularly relevant as it allows us to evaluate the effectiveness of the proposed method in estimating the extremal dependence structure. We set $\theta = 1$ for every copula, as it does not alter the domain of attraction. Based on Proposition \ref{prop:domain_attraction_repetition_model} and Proposition 4.1 of \cite{bucher2014extreme}, we know that $C_m$ falls within the max domain of attraction of the corresponding copula $C_\infty$ defined in Experiments \ref{exp:E1}-\ref{exp:E3}. In other words, it represents an AI-block model with a Logistic dependence structure for the marginals. We simulate them using the method proposed by the copula \texttt{R} package (\cite{nacopula}). The goal of our algorithm is to cluster $d$ variables in $\mathbb{R}^n$. Several simulation frameworks are considered and detailed in the following. 
\begin{enumerate}[label = \textcolor{frenchblue}{F\arabic*}, leftmargin=*]
    \setcounter{enumi}{0}
	\item We first investigate the choice of the intermediate sequence $m$ of the block length used for estimation. We let $m \in \{3,6,\dots,30\}$ with a fixed sample size $n = 10 000$ and $k = \lfloor n/m \rfloor$. \label{fra:F1}
	\item We compute the performance of the structure learning method for varying sample size $n$. Since the value of $m$ which is required for consistent  estimation  is unknown in practice we choose $m = 20$. \label{fra:F2}
    \item We show the relationship between the average $\SECO$ and exact recovery rate of the method presented in Section \ref{subsec:data_driv}. We use the case $n = 16 000$, $k = 800$ and $d = 1600$ to study the ``large $k$, large $d$'' of our approach. \label{fra:F3}
\end{enumerate}

In the simulation study, we use the fixed threshold $\alpha = 2 \times (1/m + \sqrt{\ln(d)/k})$ for \ref{fra:F1} and  \ref{fra:F2} since our theoretical results given in Theorem \ref{thm:grow_dim}  suggest the usage of a threshold proportional $d_m + \sqrt{\ln(d)/k}$ and we can show, in the i.i.d. settings (where $p = 1$) that $d_m = O(1/m)$ (see details in Section \ref{subsec:estimation_proofs}). For {Framework}  \ref{fra:F3}, we vary $\alpha$ around its theoretical optimal value, on a fine grid. The specific parameter setting we employ involves setting $p = 0.9$, which is further detailed below and illustrated in Figure \ref{fig:results_0.9p}.
 
\begin{figure}[!ht]
    \centering
    \begin{subfigure}[b]{1\textwidth}
        \includegraphics[width=1\linewidth, height = 0.25\textheight]{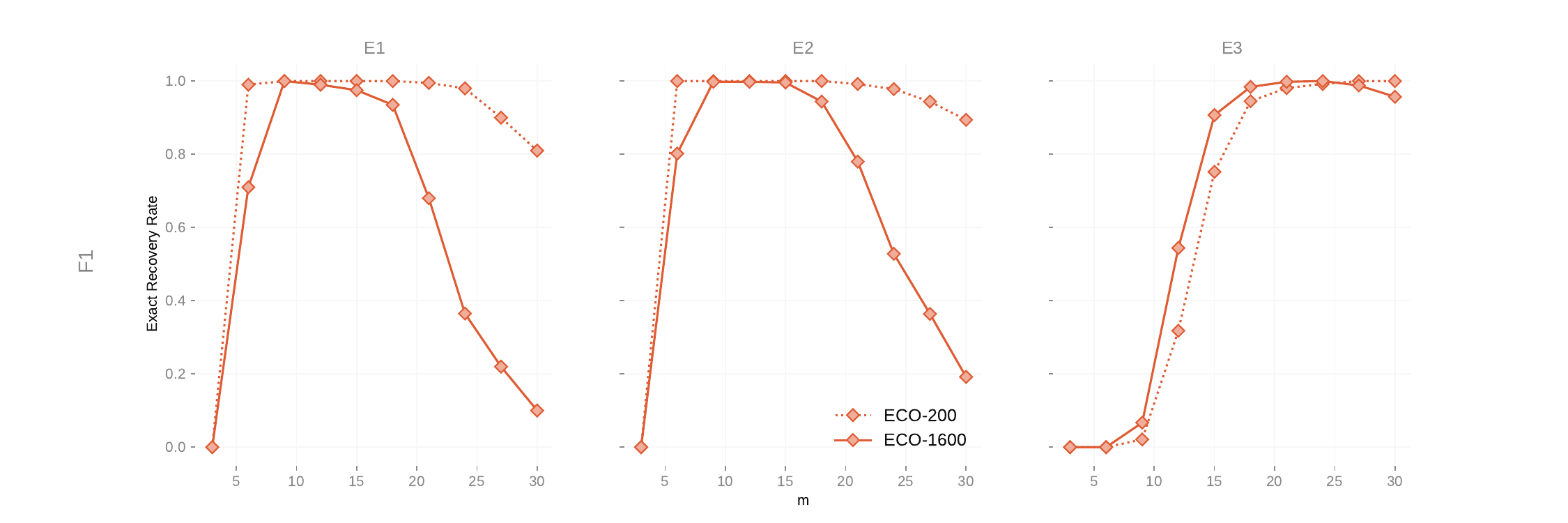}
    \end{subfigure}
    \begin{subfigure}[b]{1\textwidth}
        \includegraphics[width=1\linewidth, height = 0.25\textheight]{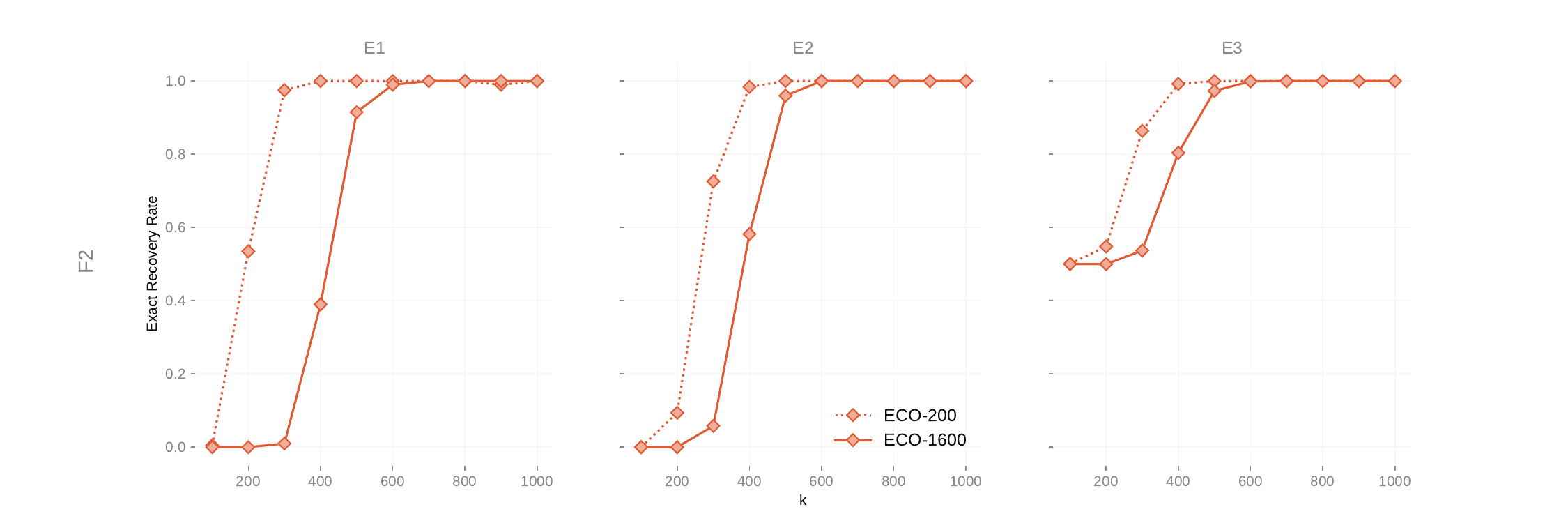}
    \end{subfigure}
    \begin{subfigure}[b]{1\textwidth}
        \includegraphics[width=1\linewidth, height = 0.25\textheight]{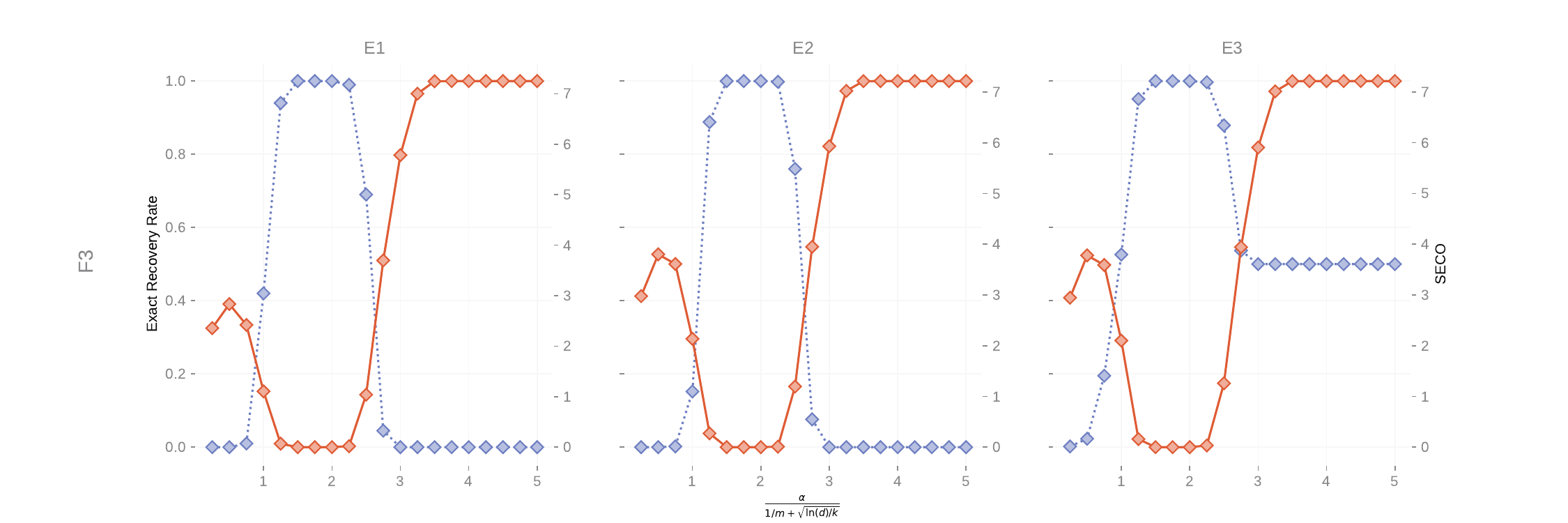}
    \end{subfigure}
    \caption{Simulation results with $p=0.9$. From top to bottom: Framework \ref{fra:F1}, Framework \ref{fra:F2}, Framework \ref{fra:F3}. From left to right: Experiment \ref{exp:E1}, Experiment \ref{exp:E2}, Experiment \ref{exp:E3}. Exact recovery rate for our algorithm (red, diamond points) for Frameworks \ref{fra:F1} and \ref{fra:F2} across 100 runs. Dotted lines correspond to $d = 200$, solid lines to $d = 1600$. The threshold $\tau$ is taken as $2 \times (1/m + \sqrt{\ln(d)/k})$.  For Framework 
 \ref{fra:F3}, average $\SECO$ losses (red solid lines, diamonds points) and exact recovery percentages (blue dotted lines, diamond points) across 100 simulations. For better illustration, the $\SECO$ losses are standardized first by  subtracting  the minimal $\SECO$ loss in each figure, and the standardized $\SECO$ losses plus 1 are then plotted on the logarithmic scale.}
    \label{fig:results_0.9p}
\end{figure}
\paragraph{Results.} Figure \ref{fig:results_0.9p} states all the results we obtain from each experiment and framework considered in this numerical section. We plot the exact recovery rate for Algorithm \ref{alg:rec_pratic} with dimensions $d=200$ and $d=1600$. Each experiment is performed using $p = 0.9$. As expected, the performance of our algorithm in Framework \ref{fra:F1} (see Figure \ref{fig:results_0.9p}, first row) is initially increasing in $m$, reaches a peak, and then decreases. This phenomenon depicts a trade-off between bias and the accuracy of inference. Indeed, a large block's length $m$ induces a lesser bias as we reach the domain of attraction. However, the number of blocks $k$ is consequently decreasing and implies a high variance for the inference process. These joint phenomena explain the parabolic form of the exact recovery rate for our algorithms for $d \in \{200,1600\}$. Considering the Framework \ref{fra:F2}  the performance of our algorithm is better as the number of block-maxima increases (see Figure \ref{fig:results_0.9p}, second row).

A classical pitfall for learning algorithms is high dimensional settings. Here, when the dimension increases from $200$ to $1600$, our algorithm consistently reports the maximal element $\bar{O}$ with a reasonable number of blocks.  This is in accordance with our theoretical findings, as the difficulty of clustering in AI-block models, as quantified by $\eta$ in Theorem \ref{thm:grow_dim}, scales at a rate of $\sqrt{\ln(d)k^{-1}}$. This rate has a moderate impact on the dimension $d$. In  Framework \ref{fra:F3}, the numerical studies in Figure \ref{fig:results_0.9p} (third row)  show that the optimal ranges of $\tau$ value, for high exact recovery percentages, are also associated with low average $\SECO$ losses. This supports our data-driven choice of $\tau$ provided  in Section~\ref{subsec:data_driv}.

\subsection{Comparison with competitors}

In this section, we examine the performance of approximate recovery of clusters of \ref{alg:rec_pratic} compared to DAMEX (\cite{goix2016sparse}), CLEF (\cite{chiapino2019identifying}), sKmeans (\cite{janssen2020k}), MUSCLE (\cite{meyer2023multivariate}) in terms of the Adjusted Rand Index (ARI). The ARI is a continuous metric ranging from -1 to 1 used to compare two partitions of a set. An ARI value of $1$ indicates identical partitions, while random partitions typically yield a value close to zero. Negative values occur for adversarial partitions, indicating that two elements that should be together fall into different groups more often than expected at random. The results in this section can be reproduced using the code made available at \href{https://github.com/Aleboul/ai_block_model}{https://github.com/Aleboul/ai\_block\_model}.

\paragraph{The setup} We consider the discrete-time $d$-variate moving maxima process $(\textbf{Y}_t, t \in \mathbb{Z})$ of order $p \in \mathbb{N}$ given by
\begin{equation}
	\label{eq:mm_process}
	Y_t^{(a)} = \bigvee_{\ell=0}^p \rho^\ell \epsilon_{t+ \ell}^{(a)}, (t \in \mathbb{Z}, \, a = 1,\dots,K), \quad \rho \in (0,1).
\end{equation}
Here $(\boldsymbol{\epsilon}_t, t\in \mathbb{Z})$ is an i.i.d. sequence of $K$-dimensional random vectors having a Clayton copula dependence function with parameter equal to unity and standard Pareto margins. Let us consider $(\textbf{Z}_t, t \in \mathbb{Z})$ as $\textbf{Z}_t = A \textbf{Y}_t + \textbf{E}_t$, where $A = (A_{ja})_{j=1,\dots,d, a = 1,\dots,K} \in [0,1]^{d \times K}$ be a coefficient matrix with rows sums to $\sum_{a=1}^K A_{ja} = 1$ for all $j = 1,\dots,d$ and $\textbf{E}_t$ serves as a vector of noise, independent of $\textbf{Y}_t$ with a tail that is lighter than $\textbf{Y}_t$, for any $t\in \mathbb{Z}$. Specifically, taking $\textbf{E}_t$ to be a multivariate Gaussian vector with the identity as its covariance matrix verifies this tail condition and is considered in this section. Then, the considered process $(\textbf{Z}_t, t \in \mathbb{Z})$ is in the max-domain of attraction of a max-linear model. The extreme directions of this model are the sets $J_a = \{ j \in \{1,\dots,d\}, A_{ja} > 0\}$, for $a = 1,\dots,K$. Moreover, this model can also be linked to AI-block models through the matrix $A$ by considering $\mathcal{L} = \{ L_1,\dots L_G\}$ a partition of $\{1,\dots,K\}$, then the clusters
\begin{equation*}
	O_g = \{ j \in \{1,\dots,d\}, \exists ! g \in \{1,\dots,G\}, A_{ja} \neq 0, a \in L_g \},
\end{equation*}
constitute an asymptotic independent partition of $\{1,\dots,d\}$, hence an AI-block model. Moreover, we specifically have in this setting $\bigcup_{a \in L_g} J_a = O_g$. This equation also supports a merging step for procedures that learn extreme directions to achieve clustering in AI-block models. In the experiments, we specifically merge two extreme directions    if they share a common variable. We design the extremal dependence using the matrix $A$ in two Experiments  {\ref{exp:E4} and \ref{exp:E5}}. In each of these experiments, we consider two different frameworks {\ref{fra:F4} and \ref{fra:F5}}. They are described below:
\begin{enumerate}[label = \textcolor{frenchblue}{E\arabic*}, leftmargin=*]
	\setcounter{enumi}{3}
	\item  \textit{Few large clusters}:  We set $K=100$, with $5$ clusters associated with groups of columns $L_g = \{20\times g+1,\dots, 20 \times (g+1)\}$ where $g=0,\dots,4$. These groups contain respectively $(6,5,4,3,2) \times C$ entities, where $C$ is a positive integer. \label{exp:E4}
	\item \textit{Many small clusters}: We set $K=100 \times C$, with $5 \times C$ clusters corresponding to the group of columns $L_g = \{20 \times k \times c+1,\dots,20 \times(k+1)\times c\}$ where $ k \in \{0,1,2,3,4\}$, $c \in \{1,\dots,C\}$, $g=(k+1)\times C$ so that $G$ equals $5 \times C$ with $C$ is a positive integer.\label{exp:E5}
\end{enumerate}
\begin{enumerate}[label = \textcolor{frenchblue}{F\arabic*}, leftmargin=*]
	\setcounter{enumi}{3}
	\item We consider a framework where \textit{Condition \ref{ass:A} holds}: rows of $A$, denoted as $A_{j \cdot}$, with $j \in O_g$, are sampled uniformly over the unit simplex $\mathbb{R}_+^{(L_g)}$. We investigate the performance of the algorithms with varying $d$ and $n$. We let $C$ range over $\{1,2,4,8,16,32\}$, resulting in $d \in \{20,40,80,160,320\}$ and using $n \in \{2000,3000,\dots,10000\}$. \label{fra:F4}
	\item In this scenario, we explore a framework where \textit{Condition \ref{ass:A} fails}. Let $s \in \{3,\dots,20\}$ represents the sparsity index. Then, for $j \in O_g$, the rows of the matrix $A$ are uniformly sampled from a random subset of $L_g$ of size $s$ over the unit simplex in $\mathbb{R}_+^s$. In this setup, we enforce clusters to be asymptotically dependent by ensuring that at least one association is shared between any pairs of variables, not necessarily the same association, so that Condition \ref{ass:A} fails. We let $C$ range over $\{1,2,4,8,16,32\}$, resulting in $d \in \{20,40,80,160,320\}$ and using $s \in \{3,4,\dots,20\}$ with a fixed $n=5000$.\label{fra:F5}
\end{enumerate}
We present and provide commentary on the results for specific values of $d$ and $n$; results for other values are available upon request.

\begin{figure}[ht] % <---
   \begin{subfigure}{0.32\textwidth}
       \includegraphics[width=\linewidth, height = 0.25\textheight]{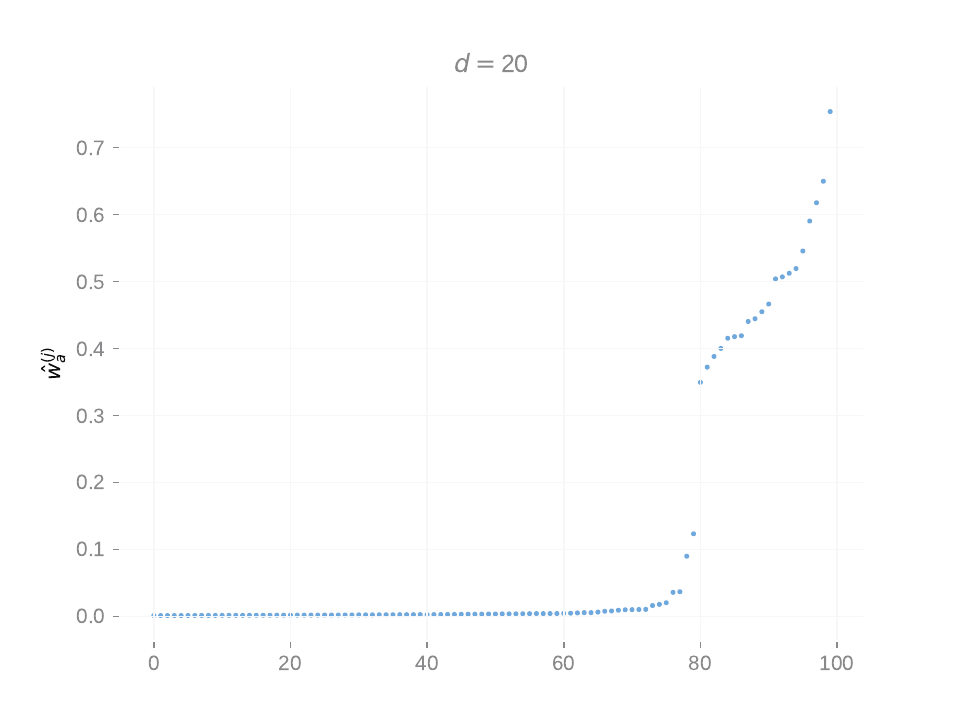}
       \caption{}
       \label{fig:diag_plot_sk1}
   \end{subfigure}
\hfill % <--- 
   \begin{subfigure}{0.32\textwidth}
       \includegraphics[width=\linewidth, height = 0.25\textheight]{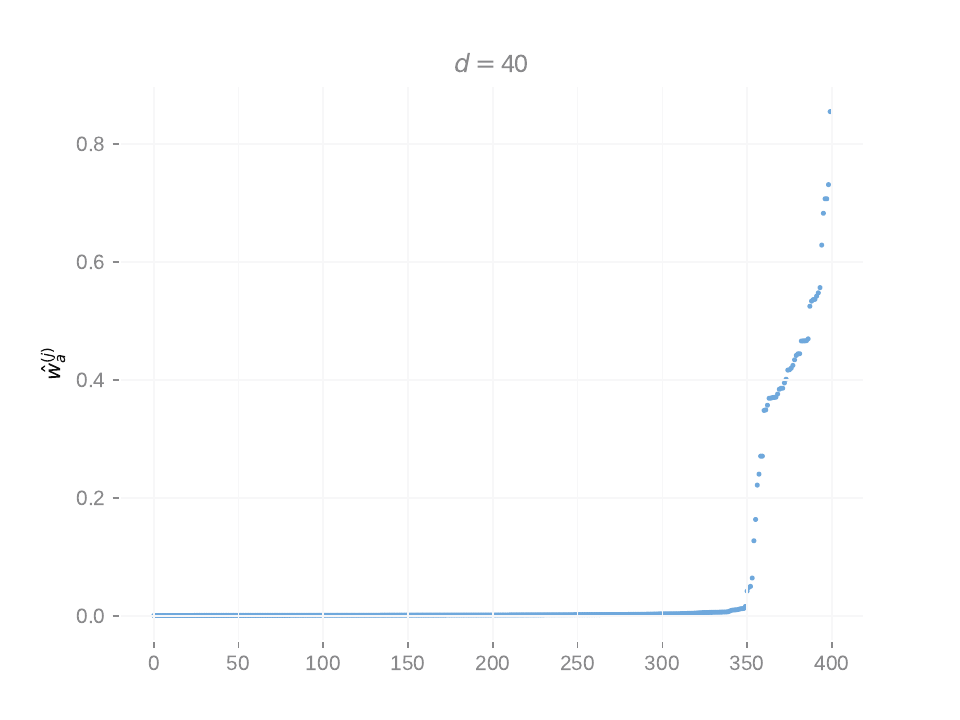}
       \caption{}
       \label{fig:diag_plot_sk2}
   \end{subfigure}
\hfill % <---
   \begin{subfigure}{0.32\textwidth}
       \includegraphics[width=\linewidth, height = 0.25\textheight]{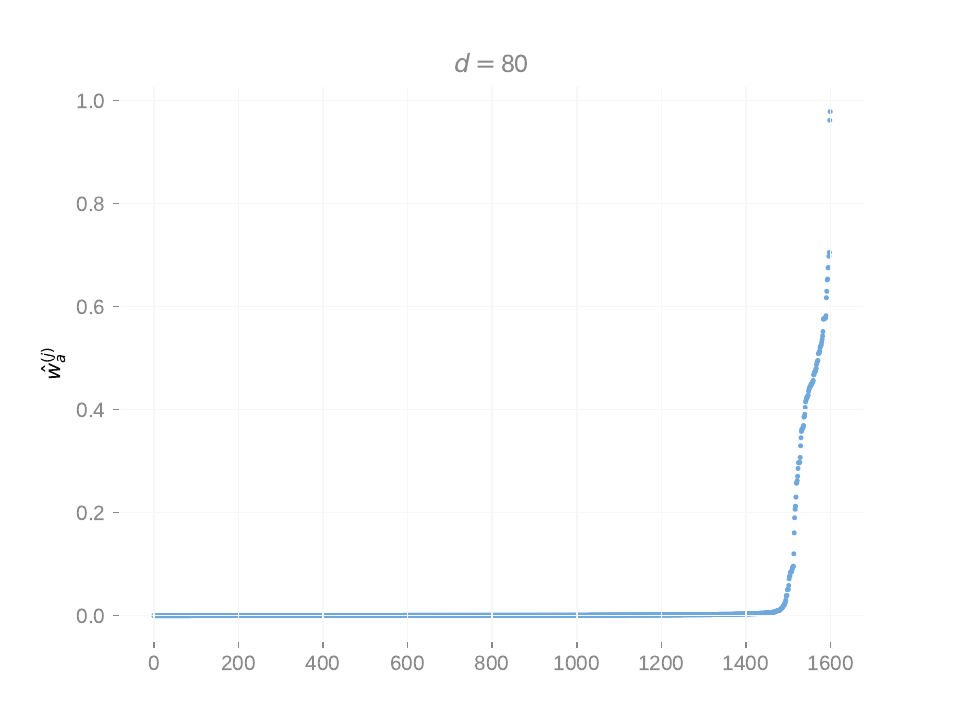}
       \caption{}
        \label{fig:diag_plot_sk3}
   \end{subfigure}
   \caption{Sorted centroïds $\hat{w}_a^{(j)}$ in Experiment \ref{exp:E5} with $j=1,\dots,d$ for $d \in \{20,40,80\}$, $a = 1,\dots,G$ with $G \in \{5,10,20\}$ and $n=10000$.}
   \label{fig:diag_plot_skmeans}
\end{figure} 

\paragraph{Calibrating parameters.} The tuning parameter $\tau$ of \ref{alg:rec_pratic} is selected by the data-driven approach described in Section \ref{subsec:data_driv} where the block size is taken to be $m=20$. In CLEF and DAMEX, the threshold was chosen by trial and error using the associated Adjusted Rand Index (ARI) with respect to the ground truth (which is unknown in practice) in the interval (0, 1). Thus, $\epsilon = 0.3$ and $\kappa = 0.2$ were selected for CLEF and DAMEX, respectively. The selected number of extremes is the one used by the authors, i.e., $k = \lfloor \sqrt{n} \rfloor$. The MUSCLE algorithm is fully adaptative and does not require specifying any parameters. We exclude the first extreme direction from the merging step because it is always associated with the trivial direction $\{1,\dots,d\}$, a phenomenon previously observed in \cite{meyer2023multivariate} (Appendix 2). 

Since sKmeans does not directly perform variable clustering, we gather the estimated centroids $\hat{\textbf{w}}_a \in \mathbb{R}^d$, $a = 1,\dots,G$. We then threshold them by $\tau$. Variables that remain positive represent groups of variables that are extremes together. Since this threshold parameter changes with the structure of A, several values of $\tau$ must be chosen. Specifically, $\tau$ was selected from $\{0.15,0.1,0.05,0.05,0.04,0.025,0.02\}$ for Experiment \ref{exp:E4} and set to $\tau = 0.15$ for Experiment \ref{exp:E5} where, for each, we set the true number of clusters (unknown in practice) to $G = 5$ in Experiment \ref{exp:E4} and $G = 5 \times C$ in Experiment \ref{exp:E5}.

Figure \ref{fig:diag_plot_skmeans} provides a diagnostic plot to set the threshold $\tau$ for the sorted estimated centroids $\hat{w}_a^{(j)}$  with $j=1,\dots,d$ in Experiment \ref{exp:E5}. In cases where $d=20$, the gap between components that are extreme together and those that are not is clear, but it narrows as the dimension increases.
\paragraph{Results and discussion} 
Figure \ref{fig:diag_F4} illustrates the numerical results on the approximate recovery of clusters using ARI in Framework \ref{fra:F4}, considering Experiments \ref{exp:E4} and \ref{exp:E5}. We were able to run the CLEF algorithm for small values of $d$ in Experiment \ref{exp:E4}, specifically for $d \in \{20,40\}$, before encountering memory limitations for larger dimensions. As sKmeans cannot be performed when there are fewer extreme observations than the desired number of clusters, some data are missing in Experiment \ref{exp:E5}.
\begin{figure}[ht] % <---
   \begin{subfigure}{0.48\textwidth}
       \includegraphics[width=\linewidth, height = 0.3\textheight]{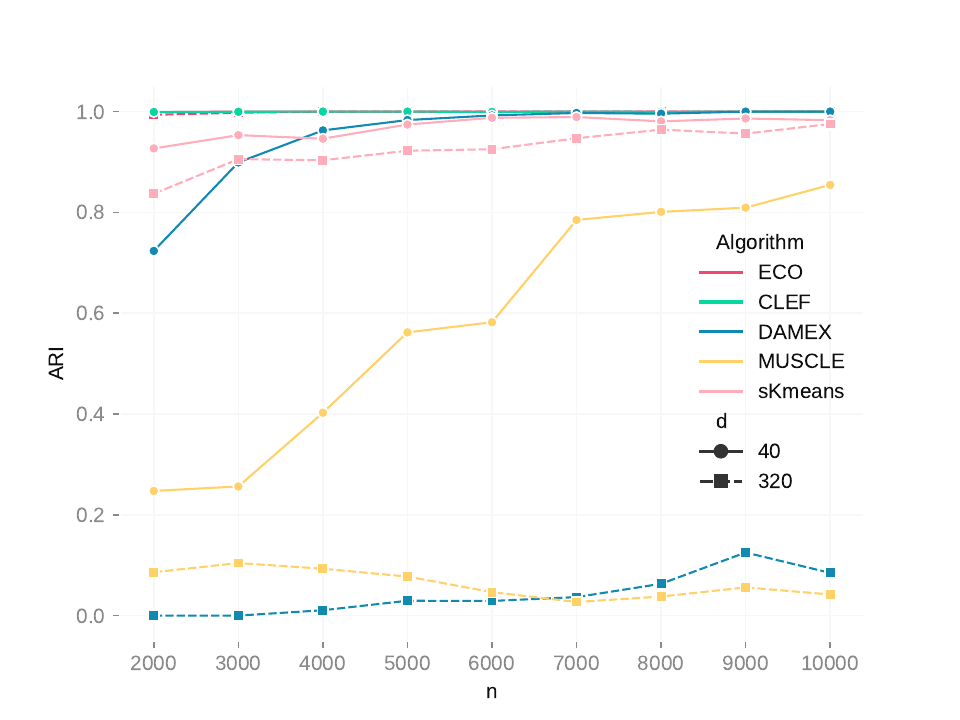}
       \caption{Experiment \ref{exp:E4} with Framework \ref{fra:F4}}
       \label{fig:diag_E4_F4}
   \end{subfigure}
\hfill % <--- 
   \begin{subfigure}{0.48\textwidth}
       \includegraphics[width=\linewidth, height = 0.3\textheight]{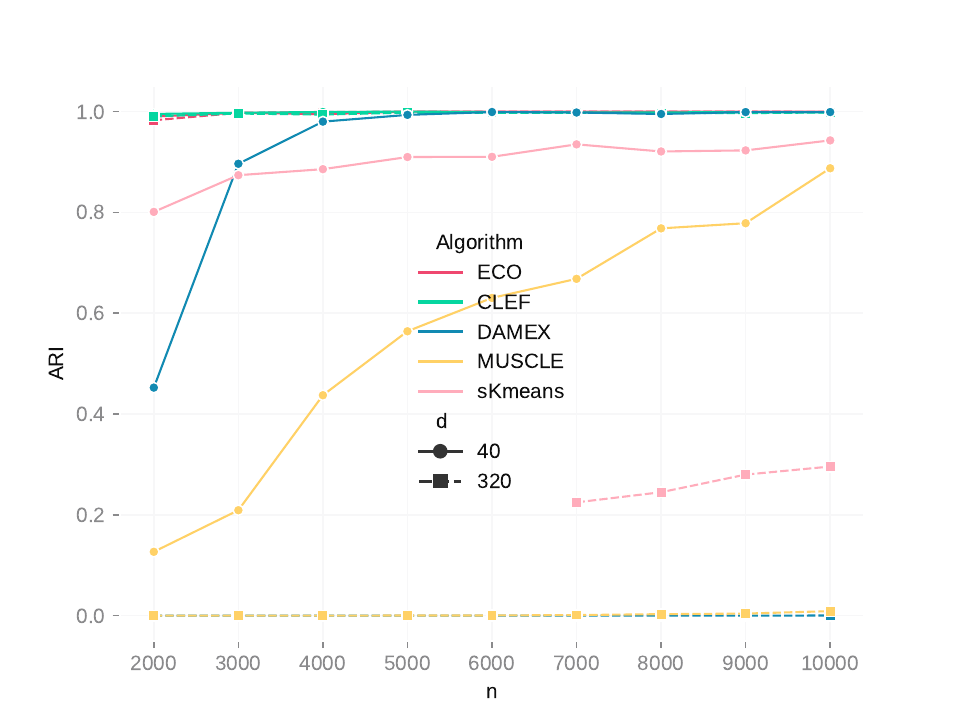}
       \caption{Experiment \ref{exp:E5} with Framework \ref{fra:F4}}
       \label{fig:diag_E5_F4}
   \end{subfigure}

   \caption{Panel \subref{fig:diag_E4_F4} (resp. Panel \subref{fig:diag_E5_F4}) depicts numerical results for Experiment \ref{exp:E4} (resp. \ref{exp:E5}) coupled with Framework \ref{fra:F4} for $n \in \{2000,3000,\dots,10000\}$ and $d\in \{40,320\}$.}
   \label{fig:diag_F4}
\end{figure}

All algorithms demonstrate an increase in performance as the number $n$ of observations increases. However, with increasing dimensionality, we observe decreasing performance for DAMEX, MUSCLE, and sKmeans, indicating difficulties in recovering   extreme directions in higher dimensions. As expected, Algorithm \ref{alg:rec_pratic} remains robust to the rise in dimensionality, even for smaller values of $n$. Since the CLEF algorithm constructs asymptotically dependent pairs, triplets, quadruplets, and so on, it is anticipated that in Experiment \ref{exp:E5} the procedure operates without memory limitations, given that the maximum cluster size is $6$. Figure \ref{fig:condition_b_fails} presents the numerical results on approximate recovery of clusters using Adjusted Rand Index (ARI) in Framework \ref{fra:F5}, considering both Experiments \ref{exp:E4} and \ref{exp:E5}. The selected threshold for sKmeans is directly linked to the structure of the matrix A. Due to the complexity of determining this threshold within this context, this procedure is excluded in Framework \ref{fra:F5}. Additionally, the CLEF algorithm requires a large amount of memory, and the procedure fails to run for a sparsity index greater than $11$ when $d = 160$ in Experiment \ref{exp:E4}, which explains missing  points in panel \subref{fig:diag_E4_F5} of Figure \ref{fig:condition_b_fails} .
\begin{figure}[ht] % <---
   \begin{subfigure}{0.48\textwidth}
       \includegraphics[width=\linewidth, height = 0.3\textheight]{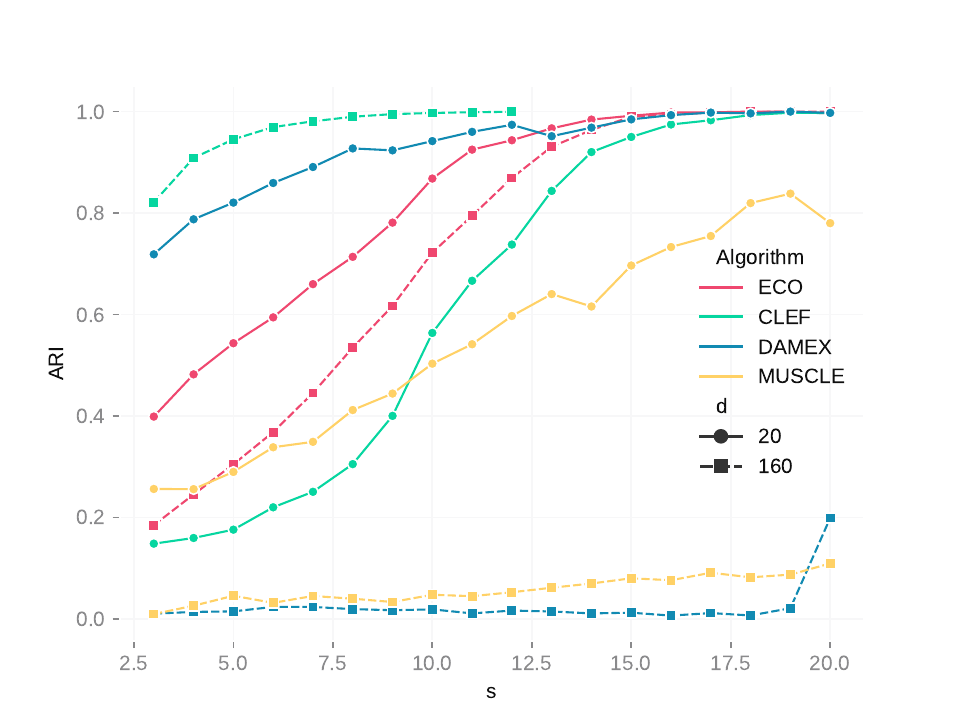}
       \caption{Experiment \ref{exp:E4} with Framework \ref{fra:F5}}
       \label{fig:diag_E4_F5}
   \end{subfigure}
\hfill % <--- 
   \begin{subfigure}{0.48\textwidth}
       \includegraphics[width=\linewidth, height = 0.3\textheight]{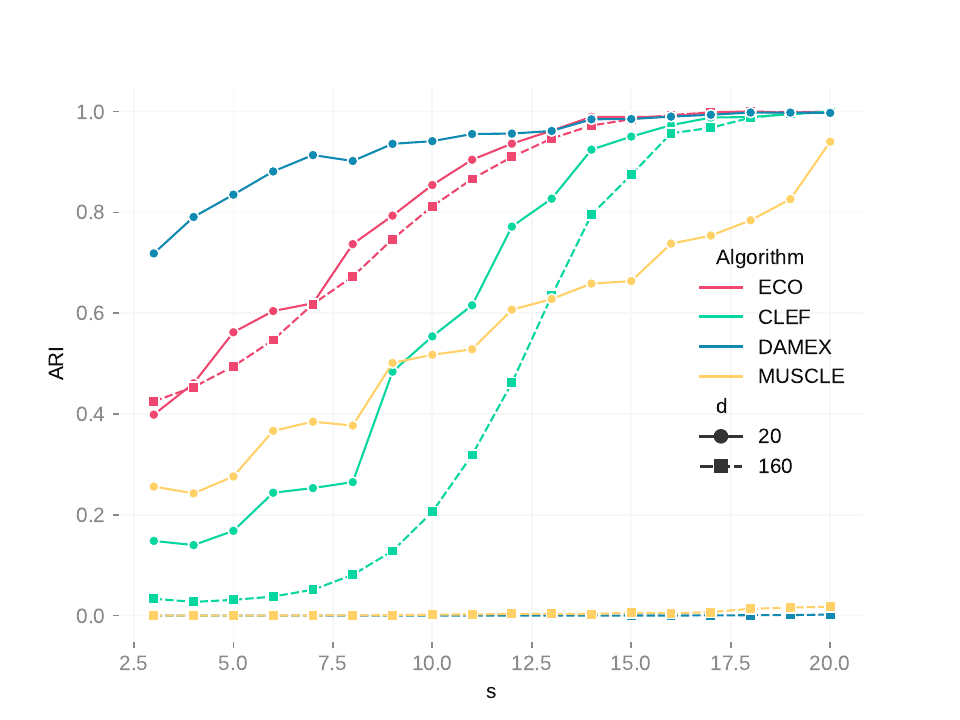}
       \caption{Experiment \ref{exp:E5} with Framework \ref{fra:F5}}
       \label{fig:diag_E5_F5}
   \end{subfigure}

   \caption{Panel \subref{fig:diag_E4_F5} (resp. Panel \subref{fig:diag_E5_F5}) depicts numerical results for Experiment \ref{exp:E4} (resp. \ref{exp:E5}) coupled with Framework \ref{fra:F5} for $n = 5000$ and $d\in \{20,160\}$, with the sparsity index $s \in \{3,4,\dots,20\}$.}
   \label{fig:condition_b_fails}
\end{figure}
As anticipated, our procedure demonstrates decreasing performance as the sparsity index decreases, given its heavy reliance on Condition \ref{ass:A}. When this conditions fails, our procedure recovers clusters that are too sparse. Surprisingly, other algorithms also exhibit sensitivity to the sparsity index $s$ and display a similar declining trend. Notably, the \ref{alg:rec_pratic} algorithm remains the most robust procedure to increasing dimensions in both experiments, while both DAMEX and MUSCLE show declining performance. We now provide a more nuanced discussion of the CLEF algorithm. In Experiment \ref{exp:E4}, both in Framework \ref{fra:F4} and \ref{fra:F5}, the CLEF algorithm demonstrates better performance in higher dimensions. This phenomenon can be explained by considering that one cluster may contain many variables that exhibit asymptotic dependence. Consequently, by construction, the CLEF algorithm is more likely to identify ``good candidates'' of pairs, triplets, quadruplets, and so on, that are indeed asymptotically dependent. Thus, the merging step we introduce to construct the cluster is more likely to yield the desired outcome. This explanation is coherent with Experiment \ref{exp:E4}, where clusters have a constant size. In this case, we observe that CLEF shows a decreasing performance in higher dimensions.

\section{Real-data applications}
\label{sec:application}

\subsection{Clustering brain extreme from EEG channel data}
\label{subsec:brain_extreme}

 Epilepsy, a significant neurological disorder, manifests as recurring unprovoked seizures. These seizures represent uncontrolled and abnormal electricity activity in the brain, posing a negative impact on one's quality of life and potentially triggering comorbid conditions like depression and anxiety. During a seizure episode, the patient may experience a loss of muscle control, which can result in accidents and injuries (see \cite{strzelczyk2023impact}).

 One essential tool used in the diagnosis of epilepsy is the electroencephalogram (EEGs). EEGs are utilized to measure the electrical activity of the brain by employing a uniform array of electrodes. Each EEG channel is formed by calculating the potential difference between two electrodes and captures the combined potential of millions of neurons. The EEG plays a crucial role in capturing the intricate brain activity, especially during epileptic seizures, and requires analysis using statistical models. Currently, most analysis methods rely on Gaussian models that focus on the central tendencies of the data distribution (see, for example, \cite{embleton2020multiscale,ombao2005slex}). However, a significant limitation of these approaches is their disregard for the fact that neuronal oscillations exhibit non-Gaussian probability distributions with heavy tails. To address this limitation, we employ AI-block models as a comprehensive framework to overcome the limitations of light-tailed Gaussian models and investigate the extreme neural behavior during an epileptic seizure.

 The dataset used to evaluate our method comprises of $916$ hours of continuous scalp EEG data sampled at a rate of $256$ Hz. This dataset were recorded from a total of $23$ pediatric patients at Children's Hospital Boston, see, e.g., \cite{shoeb2009application}. We focus the analysis on the Patient number $5$ which is the first patient where $40$ hours of continuous scalp EEG were sampled without interruption. Throughout the recordings, the patient experienced a total of five events that were identified as clinical seizures by medical professionals. The pediatric EEG data used in this paper is contained within the CHB-MIT database, which can be downloaded from: \href{https://physionet.org/content/chbmit/1.0.0/}{https://physionet.org/content/chbmit/1.0.0/}.

 For each non-seizure and seizure events, we follow the same specific processing pipeline. First, we calculate the block maxima, then calibrate the threshold using the $\SECO$ metric, as is suggested in Section \ref{subsec:data_driv}. Finally, we perform the clustering task (see Algorithm \ref{alg:rec_pratic}) using this adjusted threshold.

 In the case of non-seizure records, we compute the block maxima using a block duration of $4$ minutes. Figure \ref{fig:subim1} illustrates the relationship between the $\SECO$ and the threshold $\tau$. Two notable local minima are observed at $\tau = 0.24$ and $\tau = 0.4$. We execute the algorithm for both values and present the results for $\tau = 0.4$ since these results are better suited to AI-block models. Indeed, we obtain three clusters that demonstrate extreme dependence within the clusters while displaying weak extreme dependence in the block's off-diagonal (refer to Figure \ref{fig:subim2}). The spatial organisation of channel clusters is depicted in Figure \ref{fig:subim3}.

 Regarding seizure events, as the time series spans only $558$ seconds, we compute block maxima with a length of $5$ seconds. Considering the heavy-tailed nature of oscillations during a seizure, we believe that the limited length of the block used would not introduce a significant bias with respect to the domain of attraction. Figure \ref{fig:subim4} shows that the $\SECO$ is monotonically increasing. Thus, the optimal selected threshold is the lowest value (in this case, $\tau = 0.1$), which results in the minimal cluster $\{1,\dots,d\}$. This phenomenon is also reflected, in the extremal correlation matrix, where each channel exhibits strong pairwise extremal dependence with other channels. Consequently, the neurological disorder of the studied Patient $5$ manifests simultaneous extremes across all channels, indicating generalized seizures with inter-channel communication.

\begin{figure}[ht] % <---
   \begin{subfigure}{0.32\textwidth}
       \includegraphics[width=\linewidth]{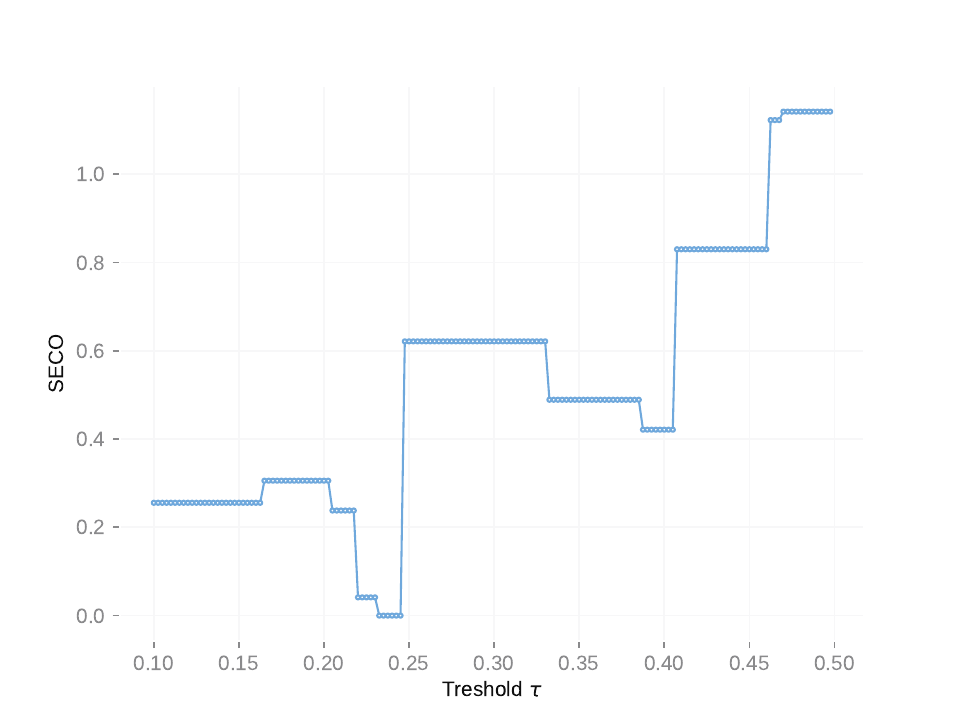}
       \caption{}
       \label{fig:subim1}
   \end{subfigure}
\hfill % <--- 
   \begin{subfigure}{0.32\textwidth}
       \includegraphics[width=\linewidth]{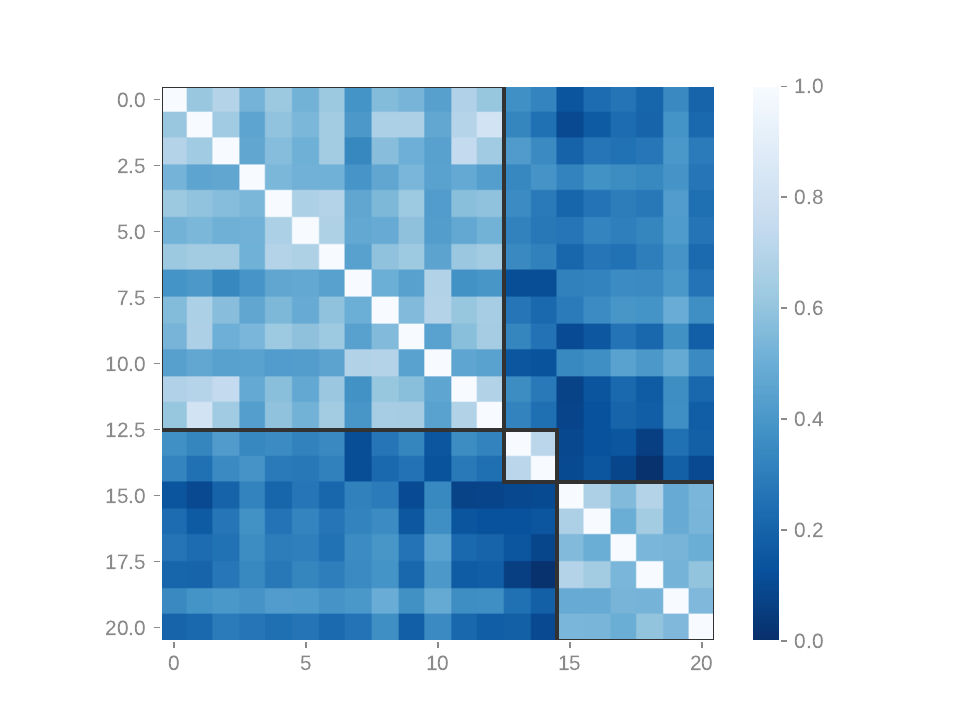}
       \caption{}
       \label{fig:subim2}
   \end{subfigure}
\hfill % <---
   \begin{subfigure}{0.32\textwidth}
       \includegraphics[width=\linewidth]{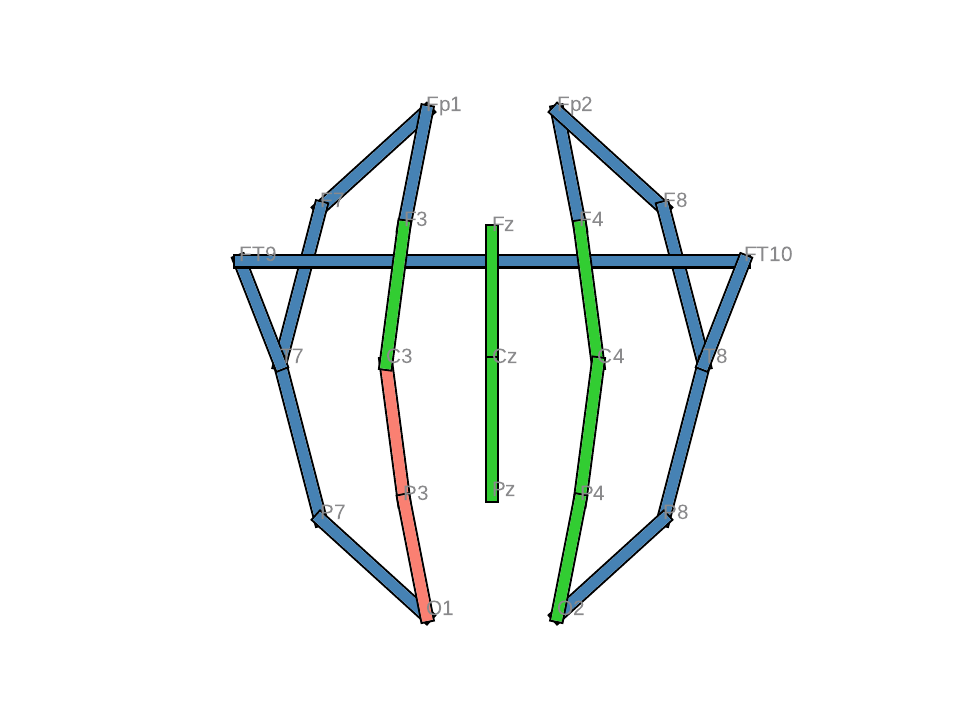}
       \caption{}
       \label{fig:subim3}
   \end{subfigure}

   \begin{subfigure}{0.32\textwidth}
       \includegraphics[width=\linewidth]{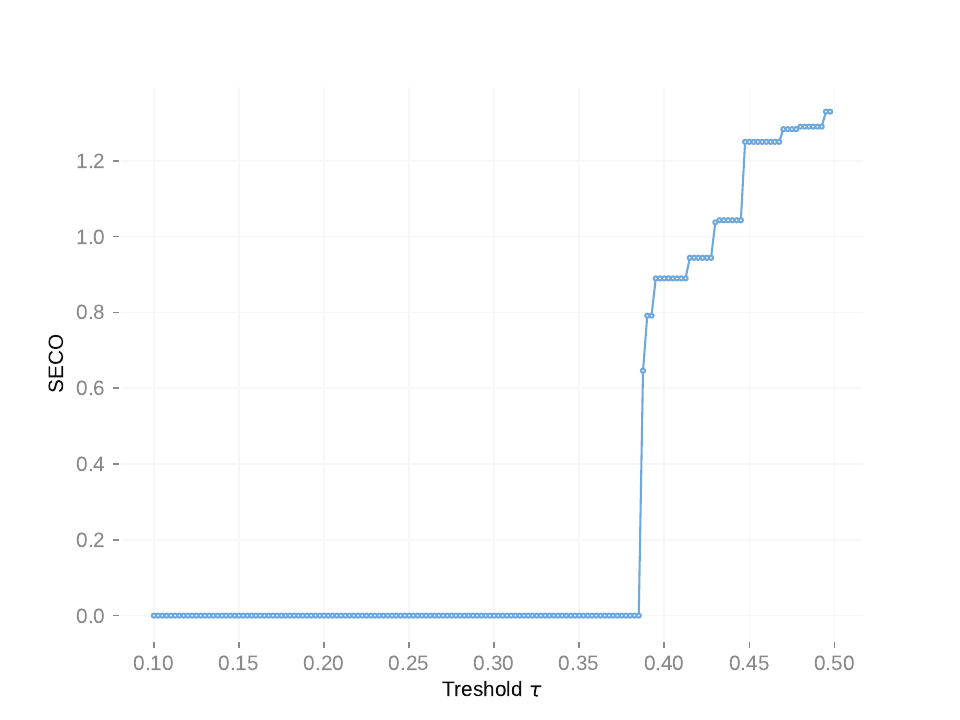}
       \caption{}
       \label{fig:subim4}
   \end{subfigure}
\hfill % <--- 
   \begin{subfigure}{0.32\textwidth}
       \includegraphics[width=\linewidth]{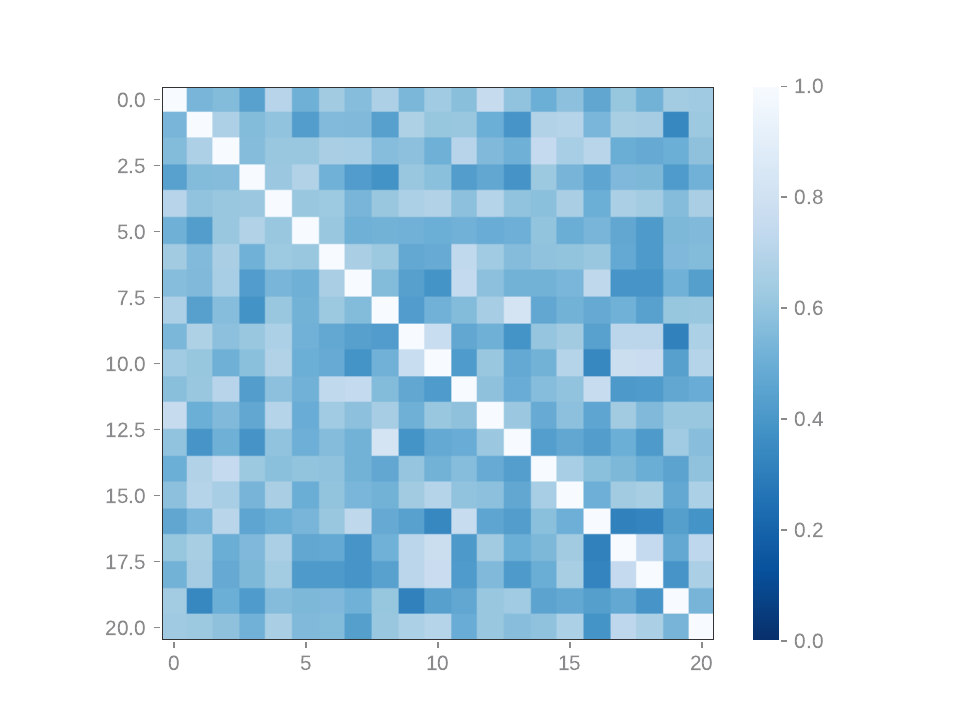}
       \caption{}
       \label{fig:subim5}
   \end{subfigure}
\hfill % <---
   \begin{subfigure}{0.32\textwidth}
       \includegraphics[width=\linewidth]{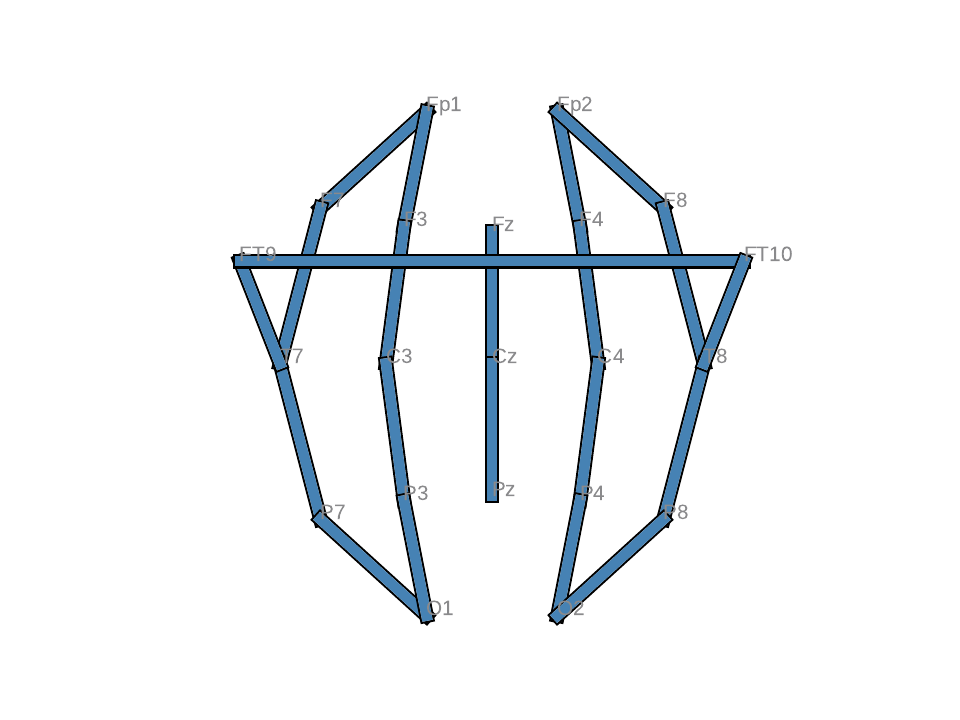}
       \caption{}
       \label{fig:subim6}
   \end{subfigure}

   \caption{{Clustering analysis on extreme brain activity derived from EEG channel data. The results are presented in the first and second rows, representing non-seizure and seizure events, respectively. The first column illustrates the behavior of the SECO metric as it relates to the threshold level, $\tau$. The second column showcases the resulting clustering performed on the extremal correlation matrix using the optimal value of $\tau$. Finally, the third column provides a spatial organisation of the clustered channels.}}
   \label{fig:image2}
\end{figure}

\subsection{Extremes on river network}

 {To demonstrate the novel regionalization method described in this paper, we employed biweekly maximum river discharge data, specifically, records collected over 14-day intervals, measured in ($m^3 /s$). This dataset were sourced from a network of 1123 gauging stations strategically positioned across European rivers. The European Flood Awareness System (EFAS) provided these data, and they are accessible free of charge via the following website \href{https://cds.climate.copernicus.eu/}{https://cds.climate.copernicus.eu/}. EFAS primarily relies on a distributed hydrological model that operates on a grid-based system, focusing on extreme river basins. The model integrates various medium-range weather forecasts, including comprehensive sets from the Ensemble Prediction System (EPS). The dataset was generated by inputting gridded observational precipitation data, with a resolution of $5 \times 5$ km, into the LISFLOOD hydrological model across the EFAS domain. The temporal resolution utilized was a $24$-hour time step, covering a span over $50$ years.}

 {For the calibration of the LISFLOOD within the EFAS framework, a total of $1137$ stations from $215$ different catchments across the Pan-European EFAS domain were used. From this list of stations with available coordinates, we extracted time-series data from the nearest cell where EFAS data were accessible. However, in this pre-processing step, stations from Albania had to be excluded as the extracted time series were identical for those stations. Additionally, calibration stations from Iceland and Israel were removed since they were located far outside the domain. As a result, we were left with $1123$ gauging stations, covering $10898$ observed days of river discharge between $1991$ and $2020$. The biweekly block maxima approach yielded $783$ observations.}

 Following the pipeline described in Section \ref{subsec:brain_extreme}, in Figure \ref{fig:river1}, the $\SECO$ is depicted as it evolves in relation to the threshold $\tau$. The minimum value is attained at $\tau = 0.25$. Using this data-driven threshold, the Algorithm \ref{alg:rec_pratic} is applied, resulting in $17$ clusters, with $11$ clusters comprising fewer than $20$ stations. Figure \ref{fig:river2} presents the resulting extremal correlation matrix, with clusters visually highlighted by squares. Within the clusters, there is evidence of asymptotic dependence, while moderate asymptotic dependence is observed in the off block-diagonal. Figure \ref{fig:river3} provides a spatial representation of three main clusters. Notably, the clusters exhibit spatial concentration, despite the algorithm being unaware of their spatial dispersion. Overall, distinct clusters representing western, central, and northern Europe can be identified. It is crucial to emphasize that the northern Europe cluster includes stations situated in the Alps and the Pyrenees, which are geographically distant from the Scandinavian peninsula. Despite the geographical separation, these regions share mountainous terrain, and the simultaneous occurrence of extreme river discharges may be attributed to snow melting.

\begin{figure}[th]
    \begin{subfigure}[b]{0.32\textwidth}
      \begin{subfigure}[t]{\textwidth}
        \includegraphics[width=\textwidth]{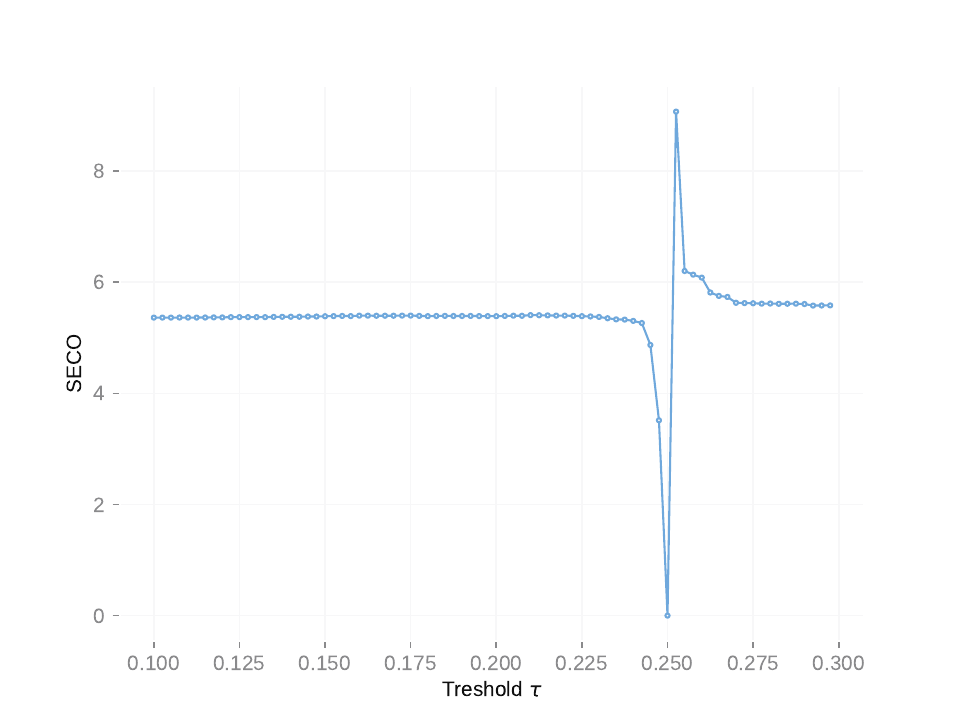}
        \caption{}
        \label{fig:river1}
      \end{subfigure}
      % NOTE2: two empty lines = 1 linebreak
      \begin{subfigure}[b]{\textwidth}
        \includegraphics[width=\textwidth]{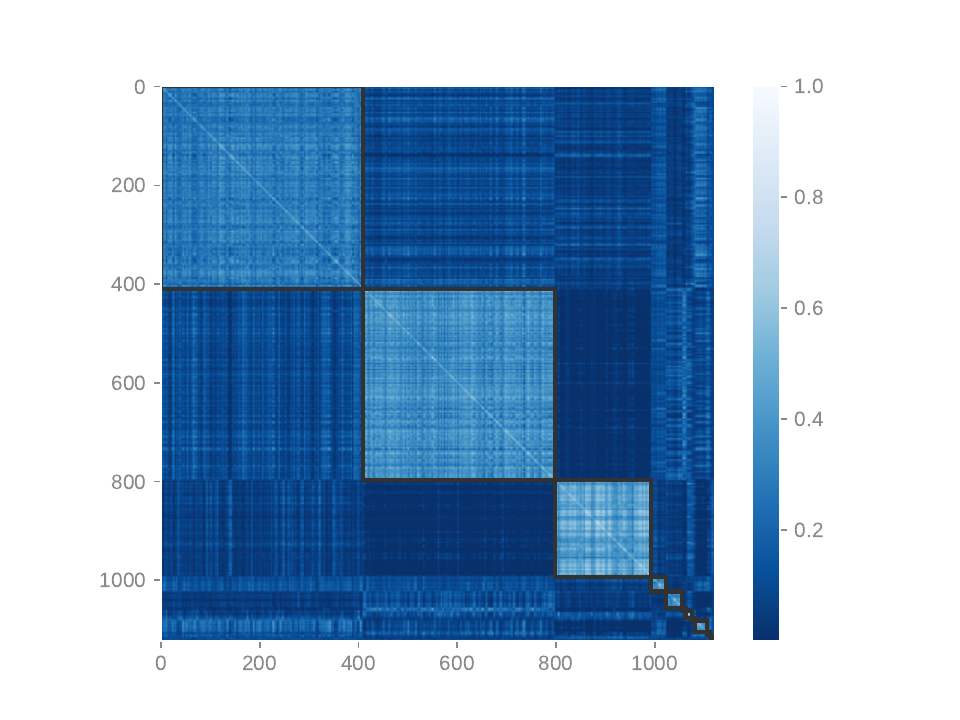}
        \caption{}
        \label{fig:river2}
      \end{subfigure}
    \end{subfigure}
    \hfill  % NOTE1: hfill moves horizontally stacked objects as far apart as it can
    \begin{subfigure}[b]{0.65\textwidth}
      \includegraphics[width=\textwidth]{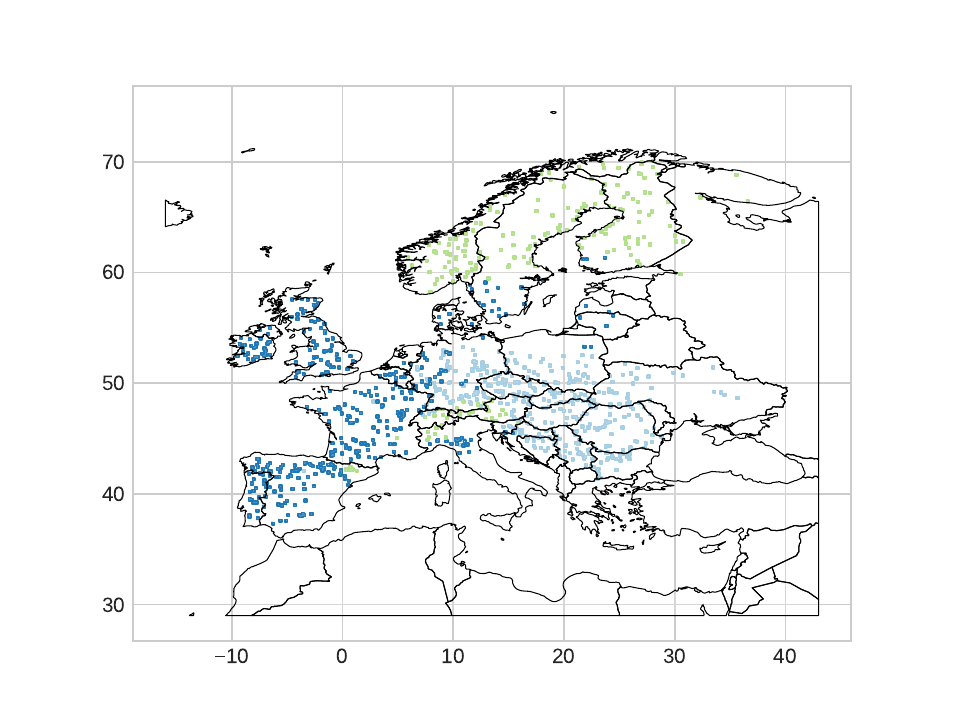}
      \caption{}
      \label{fig:river3}
    \end{subfigure}

    \caption{{Clustering analysis on extreme river discharges on EFAS data. The first panel illustrates the behavior of the SECO metric as it relates to the threshold level, $\tau$. The second panel showcases the resulting clustering performed on the extremal correlation matrix using the optimal value of $\tau$. Finally, the third one provides a spatial representation of the clustered stations.}}
  \end{figure}

\section{Conclusions}
\label{sec:conclusion}
Our main focus in this work was to develop and analyze an algorithm for recovering clusters in AI-block models, and to understand how the dependence structure of maxima impacts the difficulty of clustering in these models. This is particularly challenging when we are dealing with high-dimensional data and weakly dependent observations that are sub-asymptotically distributed. In order to better understand these phenomena, we ask stronger assumptions about the extremal dependence structure in our theoretical analysis. Specifically, we  assume the asymptotic independence between blocks, which is the central assumption of AI-block models. This assumption enables us to examine the impact of the dependence structure and develop an efficient algorithm for recovering clusters in AI-block models. By employing this procedure, we can recover the clusters with high probability by employing a threshold that scales logarithmically with the dimension $d$. However, it remains important to explore the optimal achievable rate for recovering AI-block models.

In this paper, we find a bound for the minimal extremal correlation separation $\eta > 0$.  A further goal is to find the minimum value $\eta^*$ below which it is impossible, with high probability, to exactly recover $\bar{O}$ by any method. This question can be formally expressed using Le Cam's theory as follows:
\begin{equation*}
    \underset{\hat{O}}{\inf} \underset{\mathcal{X} \in \mathbb{X}(\eta)}{\sup} \mathbb{P}_{\mathcal{X}}(\hat{O} \neq \bar{O}) \geq \mathrm{constant} > 0, \quad \forall \,  \eta < \eta^*,  
\end{equation*}
with $\mathbb{X}(\eta) = \left\{ \mathcal{X}, \MECO(\mathcal{X}) > \eta \right\}$ and the infimum is taken over all possible estimators. One possible direction to obtain such a result is to follow methods introduced by \cite{drees2001minimax} for risk bounds of extreme value index. An interesting consequence of this result is to determine whether our procedure is optimal (in a minimax sense), i.e., whether the order of $\eta^*$ and the one found in Theorem \ref{thm:grow_dim} are the same.

\printbibliography

\appendix

\section{Details on mixing coefficients}
\label{sec:details}
Consider $\textbf{Z}_{t} = (Z_t^{(1)},\dots,Z_t^{(d)})$, where $t \in \mathbb{Z}$ be a a strictly stationary multivariate random process. Let
\begin{equation*}
	\mathcal{F}_k = \sigma(\textbf{Z}_t, t \leq k ), \quad \textrm {  and  } \quad \mathcal{G}_k = \sigma(\textbf{Z}_t, t \geq k),
\end{equation*}
be respectively the natural filtration and ``reverse'' filtration of $(\textbf{Z}_t, t \in \mathbb{Z})$. Many types of mixing conditions exist in the literature. The weakest among those most commonly used is called strong or $\alpha$-mixing. Specifically, for two $\sigma$-fields $\mathcal{A}_1$ and $\mathcal{A}_2$ of a probability space $(\Omega, \mathcal{A}, \mathbb{P})$ the $\alpha$-mixing coefficient of a multivariate random process is defined for $\ell \geq 1$
\begin{equation}
	\label{eq:alpha_mixing}
	\alpha(\ell) = \underset{t \in \mathbb{Z}}{\sup} \; \alpha\left(\mathcal{F}_t, \mathcal{G}_{t+\ell}\right),
\end{equation}
where
\begin{equation*}
	\alpha\left(\mathcal{A}_1,\mathcal{A}_2\right) = \underset{A_1 \in \mathcal{A}_1, A_2 \in \mathcal{A}_2}{\sup} \left| \mathbb{P}(A_1 \cap A_2) - \mathbb{P}(A_1)\mathbb{P}(A_2) \right|.
\end{equation*}
For any process $(\textbf{Z}_t, t\in \mathbb{Z})$, let
\begin{equation*}
	\beta(\mathcal{A}_1, \mathcal{A}_2) = \sup \frac{1}{2} \sum_{i,j \in I \times J} \left| \mathbb{P}(A_i \cap B_j) - \mathbb{P}(A_i) \mathbb{P}(B_j) \right|,
\end{equation*}
where the sup is taken over all finite partitions $(A_i)_{i \in I}$ and $(B_j)_{j \in J}$ of $\Omega$ with the sets $A_i$ in $\mathcal{A}_1$ and the sets $B_j$ in $\mathcal{A}_2$. The $\beta$-mixing (or completely regular) coefficient is defined as follows
\begin{equation}
	\label{eq:beta_mixing}
	\beta(\ell) = \underset{t \in \mathbb{Z}}{\sup} \, \beta(\mathcal{F}_t, \mathcal{G}_{t+\ell}).
\end{equation}
By considering 
\begin{equation*}
	\varphi(\mathcal{A}_1, \mathcal{A}_2) = \underset{A_1, A_2 \in \mathcal{A}_1 \times \mathcal{A}_2, \mathbb{P}(A_1) \neq 0}{\sup} \left| \mathbb{P}(A_2|A_1) - \mathbb{P}(A_1) \right|,
\end{equation*}
the $\varphi$-mixing coefficient is defined by
\begin{equation}
    \label{eq:varphi_mixing}
    \varphi(\ell) = \underset{t \in \mathbb{Z}}{\sup} \, \varphi(\mathcal{F}_t, \mathcal{G}_{t + \ell})
\end{equation}
It should be noted that if the original process $(\textbf{Z}_t, t\in \mathbb{Z})$ satisfies an $\alpha$- or $\beta$- or $\varphi$-mixing condition, then the stationary process $(f(\textbf{Z}_t), t\in \mathbb{Z})$ for a measurable function $f$ also satisfies the same mixing condition. The $\alpha$-mixing rate, $\beta$-mixing rate, and $\varphi$-mixing rate of the stationary process are all bounded by the corresponding rate of the original process. In terms of their order, the three mixing coefficients are related as follows:
\begin{equation}
\label{eq:order_mixing_coeff}
\alpha(\ell) \leq \beta(\ell) \leq \varphi(\ell).
\end{equation}
This means that the $\alpha$-mixing coefficient is the weakest, followed by the $\beta$-mixing coefficient, and finally the $\varphi$-mixing coefficient is the strongest.

\section{Proofs of main results}
In the subsequent section of our materials, we employ the notation $(\boldsymbol{1},\textbf{x}^{(B)}, \boldsymbol{1})$ having its $j$th component equal to $x^{(j)} \mathds{1}_{\{j \in B\}} + \mathds{1}_{\{ j \notin B\}}$. In a similar way, we note $(\textbf{0}, \textbf{x}^{(B)}, \textbf{0})$ the vector in $\mathbb{R}^d$ which equals $x^{(j)}$ if $j \in B$ and $ 0 $ otherwise. 
\label{sec:main_results}

In the subsequent section of our materials, we employ the notation $(\boldsymbol{1},\textbf{x}^{(B)}, \boldsymbol{1})$ having its $j$th component equal to $x^{(j)} \mathds{1}_{\{j \in B\}} + \mathds{1}_{\{ j \notin B\}}$. In a similar way, we note $(\textbf{0}, \textbf{x}^{(B)}, \textbf{0})$ the vector in $\mathbb{R}^d$ which equals $x^{(j)}$ if $j \in B$ and $ 0 $ otherwise. 
\subsection{Proofs of Section \ref{sec:variable_clust}}
\label{subsec:variable_clust_proofs}

In Proposition \ref{prop:Cevt}, we prove that the function introduced in Section \ref{subsec:ai_block_models} is an extreme value copula. We do this by showing that its margins are distributed uniformly on the unit interval [0,1] and that it is max-stable, which is a defining characteristic of extreme value copulae.

\begin{proof}[Proof of Proposition \ref{prop:Cevt}]
	We first show that $C_\infty$ is a copula function. It is clear that $C_\infty(\textbf{u}) \in [0,1]$ for every $\textbf{u} \in [0,1]^d$. We check that its univariate margins are uniformly distributed on $[0,1]$. Without loss of generality, take $u^{(i_{1,1})} \in [0,1]$ and let us compute 
	\begin{equation*}
		C_\infty(1, \dots, u^{(i_{1,1})}, \dots, 1) = C_\infty^{(O_1)}(u^{(i_{1,1})}, 1, \dots, 1) = u^{(i_{1,1})}.
	\end{equation*}
	So $C_\infty$ is a copula function. We now have to prove that $C_\infty$ is an extreme value copula. We recall that $C_\infty$ is an extreme value copula if and only if $C_\infty$ is max-stable, that is for every $m \geq 1$
	\begin{equation*}
		C_\infty(u^{(1)},\dots, u^{(d)}) = C_\infty(\{u^{(1)}\}^{1/m}, \dots, \{u^{(d)}\}^{1/m})^m.
	\end{equation*}
	By definition, we have
	\begin{equation*}
		C_\infty(\{u^{(1)}\}^{1/m}, \dots, \{u^{(d)}\}^{1/m})^m = \Pi_{g=1}^G  \left\{ C_\infty^{(O_g)}\left(\{u^{(i_{g,1})}\}^{1/m}, \dots, \{u^{(i_{g,d_g})}\}^{1/m}\right) \right\}^m.
	\end{equation*}
	Using that $C_\infty^{(O_1)}, \dots, C_\infty^{(O_G)}$ are extreme value copulae, thus max stable, we obtain
	\begin{equation*}
		C_\infty(\{u^{(1)}\}^{1/m}, \dots, \{u^{(d)}\}^{1/m})^m = \Pi_{g=1}^G C_\infty^{(O_g)}\left(u^{(i_{g,1})}, \dots, u^{(i_{g,d_g})}\right) = C_\infty(u^{(1)}, \dots, u^{(d)}).
	\end{equation*}
	Thus $C_\infty$ is an extreme value copula. Finally, we prove that $C_\infty$ is the copula of the random vector $\textbf{X} = (\textbf{X}^{(O_1)}, \dots, \textbf{X}^{(O_G)})$, that is
	\begin{equation*}
		\mathbb{P} \left\{ \textbf{X} \leq \textbf{x} \right\} = C_\infty(H^{(1)}(x^{(1)}), \dots, H^{(d)}(x^{(d)})), \quad \textbf{x} \in \mathbb{R}^d.
	\end{equation*}
	Using mutual independence between random vectors, we have
	\begin{align*}
		\mathbb{P}\left\{ \textbf{X} \leq \textbf{x} \right\} &= \Pi_{g=1}^G \mathbb{P}\left\{ X^{(i_{g,1})} \leq x^{(i_{g,1})}, \dots, X^{(i_{g,d_g})} \leq x^{(i_{g,d_g})} \right\} \\ &= \Pi_{g=1}^G C_\infty^{(O_g)}\left(H^{(i_{g,1})}(x^{(i_{g,1})}), \dots, H^{(i_{g,d_g})}(x^{(i_{g,d_g})}) \right) \\ &= C_\infty(H^{(1)}(x^{(1)}), \dots, H^{(d)}(x^{(d)})).
	\end{align*}
	Hence the result.
\end{proof} 

Theorem \ref{thm:unicity}, proved below, establishes several fundamental properties of the set $\mathcal{O}((\textbf{Z}_t, t \in \mathbb{Z}))$, including the fact that subpartitions of an element $O \in \mathcal{O}((\textbf{Z}_t, t \in \mathbb{Z}))$ also belong to $\mathcal{O}((\textbf{Z}_t, t \in \mathbb{Z}))$ (item \ref{thm_(i)}), the ordering of partitions and their intersections (item \ref{thm_(ii)}) and the stability of the intersection of two elements $O, S \in \mathcal{O}((\textbf{Z}_t, t \in \mathbb{Z}))$ (item \ref{thm_(iii)}). Using these results, the theorem also provides an explicit construction of the unique maximal element $\bar{O}$ of $\mathcal{O}((\textbf{Z}_t, t \in \mathbb{Z}))$ (see item \ref{thm_(iv)}).
\begin{proof}[Proof of Theorem \ref{thm:unicity}]
    For \ref{thm_(i)}, if $\mathcal{L}((\textbf{Z}_t, t\in \mathbb{Z})) \sim S$, then there exist a random vector $\textbf{X}$ with extreme value distribution $H$ such that $\mathcal{L}((\textbf{Z}_t, t\in \mathbb{Z})) \in D(H)$ and a partition $S=\{S_1, \dots,S_G\}$ of $\{1,\dots,d\}$ which induces  mutually independent random vectors $\textbf{X}^{(S_1)}, \dots, \textbf{X}^{(S_G)}$. As $S$ is a sub-partition of $O$, it also generates a partition where vectors are mutually independent.
    
    Now let us prove \ref{thm_(ii)}, take $g \in \{1,\dots,G\}$ and $a,b \in (O \cap S)_g$, in particular $a \overset{O}{\sim} b$, thus there exists $g' \in \{1,\dots, G'\}$ such that $a,b \in O_{g'}$. The following inclusion $(O \cap S)_g \subseteq O_{g'}$ is hence obtained and the second statement follows.
    
    The third result \ref{thm_(iii)} comes down from  the definition for the direct sense and by \ref{thm_(i)} and \ref{thm_(ii)} for the reverse one. We now go to the last item of the theorem, i.e. item \ref{thm_(iv)}. The set $\mathcal{O}((\textbf{Z}_t, t \in \mathbb{Z}))$ is non-empty since the trivial partition $O = \{1,\dots,d\}$ belongs to $\mathcal{O}((\textbf{Z}_t, t \in \mathbb{Z}))$. It is also a finite set, and we can enumerate it $\mathcal{O}((\textbf{Z}_t, t \in \mathbb{Z})) = \{O_1, \dots, O_M\}$. Define the sequence $O'_1, \dots, O'_M$ recursively according to
    \begin{itemize}
    	\item $O'_1 = O_1$,
    	\item $O'_g = O_g \cap O'_{g-1}$ for $g = 2,\dots, M$.
    \end{itemize}
    According to \ref{thm_(iii)}, we have that by induction $O'_1, \dots O'_M \in \mathcal{O}((\textbf{Z}_t, t \in \mathbb{Z}))$. In addition, we have both $O'_{g-1} \leq O'_g$ and $O_g \leq O'_g$, so by induction $O_1, \dots, O_g \leq O'_g$. Hence the partition $\bar{O} := O'_M = O_1 \cap \dots  \cap O_{M-1}$ is the maximum of $\mathcal{O}((\textbf{Z}_t, t \in \mathbb{Z}))$.
\end{proof}
\begin{remark}
    The examination of the proof of Theorem \ref{thm:unicity} reveals that many arguments may also apply to the scenario of mutually independent random vectors.
\end{remark}

\subsection{Proofs of Section \ref{sec:estimation}}
\label{subsec:estimation_proofs}

Denote by $C_{n,m}^o$ the empirical estimator of the copula $C_m$ based on the (unobservable) sample $(U_{m,1}^{(j)},\dots,U_{m,k}^{(j)})$ for $j \in \{1,\dots,d\}$. In Proposition \ref{prop:concentration_inequality} we state a concentration inequality for the madogram estimator. This inequality is obtained through two main steps, that are using classical concentration inequalities, such as Hoeffding and McDiarmid inequalites and chaining arguments in our specific framework of multivariate mixing random process. In the following, $C_1$, $C_2$ and $C_3$ denote universal constants whose values could change from line to line of the proof.

\begin{proof}[Proof of Proposition \ref{prop:concentration_inequality}]
Let us define the following quantity
  \begin{equation}
  	\label{eq:mado_oracle}
		\hat{\nu}_{n,m}^o = \frac{1}{k} \sum_{i=1}^k \left[ \bigvee_{j=1}^d U_{m,i}^{(j)} - \frac{1}{d} \sum_{j=1}^d U_{m,i}^{(j)} \right],
	\end{equation}
	that is the madogram estimated through the sample $\textbf{U}_{m,1},\dots, \textbf{U}_{m,k}$. Then, the following bound is given:
	\begin{equation*}
		\left| \hat{\nu}_{n,m} - \nu_m \right| \leq \left| \hat{\nu}_{n,m} - \hat{\nu}_{n,m}^o \right| + \left| \hat{\nu}_{n,m}^o - \nu_m \right|.
	\end{equation*}
 For the second term, using the triangle inequality, we obtain 
	\begin{align*}
		\left| \hat{\nu}_{n,m}^o - \nu_m \right| &\leq  \left|  \frac{1}{k} \sum_{i=1}^k \left\{ \bigvee_{j=1}^d U_{m,i}^{(j)} - \mathbb{E}\left[ \bigvee_{j=1}^d U_{m,i}^{(j)}  \right] \right\}\right| + \left| \frac{1}{k} \sum_{i=1}^k \left\{ \frac{1}{d} \sum_{j=1}^d U_{m,i}^{(j)} - \mathbb{E}\left[ \frac{1}{d} \sum_{j=1}^d U_{m,i}^{(j)} \right] \right\} \right|  \\
		& \triangleq E_1 + E_2,
	\end{align*}
	and  for the first term, 
	\begin{equation*}
		\left| \hat{\nu}_{n,m} - \hat{\nu}_{n,m}^o \right| \leq 2 \underset{j \in \{1,\dots,d\}}{\sup} \underset{x \in \mathbb{R}}{\sup} \left| \hat{F}_{n,m}^{(j)}(x) - F_{m}^{(j)}(x) \right| \triangleq E_3.
	\end{equation*}
 The rest of this proof is devoted to  control each term: $E_1$, $E_2$ and $E_3$.  Notice that the sequences $(\bigvee_{j=1}^d U_{n,m,i}^{(j)})_{i=1}^k$, $(d^{-1} \sum_{j=1}^d U_{n,m,i}^{(j)})_{i=1}^k$ and $(\mathds{1}_{\{M_{n,m,i}^{(j)} \leq x\}})_{i=1}^k$ share the same mixing regularity as $(\textbf{Z}_t)_{t \in \mathbb{Z}}$ as measurable transformation of this process. Thus, they are in particular algebraically $\varphi$-mixing.
	
	\paragraph{Control of the term $E_1$.}
	For every $i \in \{1,\dots,k\}$, we have that $|| \bigvee_{j=1}^d U_{n,m,i}^{(j)} ||_\infty \leq 1$, by applying the Hoeffding's inequality for  algebraically   $\varphi$-mixing sequences (see \cite[Corollary 2.1]{rio2017asymptotic}) we can control the following event, for $t > 0$,
	\begin{equation*}
		\mathbb{P}\left\{ E_1 \geq t \right\} \leq \sqrt{e} \exp \left\{ - \frac{t^2 k }{2( 1 + 4 \sum_{i=1}^{k-1}\varphi(i))}\right\}.
	\end{equation*}
	The term in the numerator can be bounded as
	\begin{equation*}
		1 + 4 \sum_{i=1}^k \varphi(k) \leq 1 + 4 \sum_{i=1}^k \lambda i^{-\zeta} \leq 1 + 4\lambda \left(1+\int_{1}^k x^{-\zeta}dx \right) = 1 + 4\lambda \left(1 + \frac{k^{1-\zeta}-1}{1-\zeta}\right).
	\end{equation*}
	Using the assumption $\zeta > 1$, we can upper bound $k^{1-\zeta}$ by $1$ and obtain
	\begin{equation*}
		1+4\lambda \left( 1 + \frac{k^{1-\zeta}-1}{1-\zeta}\right) \leq 1 + 4\lambda \left(1 + \frac{1}{\zeta-1} \right) = 1 + \frac{4  \lambda\zeta}{\zeta-1}.
	\end{equation*}
	We thus obtain
	\begin{equation*}
		\mathbb{P} \left\{ E_1 \geq \frac{t}{3} \right\} \leq \sqrt{e} \exp \left\{ - \frac{t^2 k}{C_3} \right\},
	\end{equation*}
	where $C_3 > 0$ is a constant depending on $\zeta$ and $\lambda$.
	\paragraph{Control of the term $E_2$.}
	This control is obtained with the same arguments used for $E_1$. Thus, we obtain, for $t > 0$, 
	\begin{equation*}
		\mathbb{P} \left\{ E_2 \geq \frac{t}{3} \right\} \leq \sqrt{e} \exp \left\{ - \frac{t^2 k}{C_3} \right\}.
	\end{equation*}
 
	\paragraph{Control of the term $E_3$.}
	This bound is more technical. Before proceeding, we introduce some notations. For every $j \in \{1,\dots,d\}$, we define 
	\begin{equation*}
		\alpha_{n,m}^{(j)} = \left(\mathbb{P}_{n,m}^{(j)} - \mathbb{P}_{m}^{(j)} \right), \quad \beta_{n,m}^{(j)}(x) = \alpha_{n,m}^{(j)} (]-\infty,x]), \quad x \in \mathbb{R},
	\end{equation*}
	where $\mathbb{P}_{n,m}^{(j)}$ corresponds to the empirical measure for the sample $(M_{m,1}^{(j)},\dots, M_{m,k}^{(j)})$ and $\mathbb{P}_{m}^{(j)}$ is the law of the random variable $M_{m}^{(j)}$. To control the term $E_3$, we introduce chaining arguments as used in the proof of Proposition 7.1 of \cite{rio2017asymptotic}. Let be $j \in \{1,\dots,d\}$ fixed and $N$ be some positive integer to be chosen later. For any real $x$ such that $F_{m}^{(j)}(x) \neq 0$ and $F_{m}^{(j)}(x) \neq 1$, let us write $F_m^{(j)}(x)$ in base $2$ :
	\begin{equation*}
		F_{m}^{(j)}(x) = \sum_{l=1}^N b_l(x) 2^{-l} + r_N(x), \textrm{ with } r_N(x) \in [0, 2^{-N}[
	\end{equation*}
	where $b_l = 0$ or $b_l = 1$. For any $L$ in $[1,\dots,N]$, set 
	\begin{equation*}
		\Pi_L(x) = \sum_{l=1}^L b_l(x) 2^{-l} \textrm{ and } i_L = \Pi_L(x)2^L. 
	\end{equation*}
	Let the reals $(x_L)_L$ be chosen in such a way that $F_{m}^{(j)}(x_L) = \Pi_L(x)$. With these notations
	\begin{align*}
		\beta_{n,m}^{(j)}(x) =& \beta_{n,m}^{(j)}(\Pi_1(x)) + \beta_{n,m}^{(j)}(x) - \beta_{n,m}^{(j)}(\Pi_N(x)) \\ &+ \sum_{L=2}^N \left[\beta_{n,m}^{(j)}(\Pi_L(x)) - \beta_{n,m}^{(j)}(\Pi_{L-1}(x)) \right] .
	\end{align*}
	Let the reals $x_{L,i}$ be defined by $F_{m}^{(j)}(x_{L,i}) = i2^{-L}$. Using the above equality, we get that
	\begin{equation*}
		\underset{x \in \mathbb{R}}{\sup} \left| \beta_{n,m}^{(j)}(x) \right| \leq \sum_{L=1}^N \Delta_L + \Delta_N^*,
	\end{equation*}
	with
	\begin{equation*}
		\Delta_L = \underset{i \in [1,2^L]}{\sup} \left| \alpha_{n,m}^{(j)}(]x_{L,i-1}, x_{L,i}]) \right| \textrm{ and } \Delta_N^* = \underset{x \in \mathbb{R}}{\sup} \left| \alpha_{n,m}^{(j)}(]\Pi_N(x),x]) \right|.
	\end{equation*}
	From the inequalities
	\begin{equation*}
		-2^{-N} \leq \alpha_{n,m}^{(j)}(]\Pi_N(x),x]) \leq \alpha_{n,m}^{(j)}(]\Pi_N(x),\Pi_N(x) + 2^{-N}]) + 2^{-N},
	\end{equation*}
	we get that
	\begin{equation*}
		\Delta_N^* \leq \Delta_N + 2^{-N} \textrm{ and } \mathbb{E}\left[ \underset{x \in \mathbb{R}}{\sup} |\beta_{n,m}^{(j)}(x) |\right] \leq 2 \sum_{L=1}^N ||\Delta_L||_1 + 2^{-N},
	\end{equation*}
	where $||\Delta_L||_1$ is the $L^1$-norm of $\Delta_L$. Let $N$ be the natural number such that $2^{N-1} < k \leq 2^N$. For this choice of $N$, we obtain
	\begin{equation*}
	\label{eq:bound_mean}
		\mathbb{E}\left[ \underset{x \in \mathbb{R}}{\sup} |\beta_{n,m}^{(j)}(x) |\right] \leq 2 \sum_{L=1}^N ||\Delta_L||_1 + k^{-1}.
	\end{equation*}
	Hence, using \cite[Lemma 7.1]{rio2017asymptotic} (where we divide by $\sqrt{k}$ the considering inequality in the lemma), we obtain that
	\begin{align*}
		\mathbb{E}\left[ \underset{x \in \mathbb{R}}{\sup} |\beta_{n,m}^{(j)}(x) |\right] &\leq 2 \frac{C_0}{\sqrt{k}} \sum_{L=1}^N \left(2^{-\frac{(\zeta-1)^2}{(4\zeta)^2}}\right)^L + k^{-1} \\ &\leq \frac{2}{\sqrt{k}} \frac{C_0}{1-2^{-\frac{(\zeta-1)^2}{(4\zeta)^2}}} + k^{-1} \triangleq C_1 k^{-1/2} + k^{-1},
	\end{align*}
	where $C_0$ and $C_1$ are constants depending on $\zeta$ and $\lambda$.
	
	Now, fix $x \in \mathbb{R}$ and denote by $\Phi : \mathbb{R}^k \mapsto [0,1]$, the function defined by
	\begin{equation*}
		\Phi(x_1,\dots, x_k) = \underset{x \in \mathbb{R}}{\sup} \left|  \frac{1}{k} \sum_{i=1}^k \mathds{1}_{\{x_i \leq x\}} - F_{m}^{(j)}(x) \right|.
	\end{equation*}
	For $\textbf{x}, \textbf{y} \in \mathbb{R}^k$, we obtain with some calculations:
	\begin{equation*}
		\left| \Phi(\textbf{x}) - \Phi(\textbf{y}) \right| \leq \underset{x \in \mathbb{R}}{\sup} \frac{1}{k} \sum_{i=1}^k \left| \mathds{1}_{\{x_i \leq x\}} -\mathds{1}_{\{y_i \leq x\}} \right| \leq \frac{1}{k} \sum_{i=1}^k \mathds{1}_{\{ x_i \neq y_i\}}.
	\end{equation*}
	Thus, $\Phi$ is $k^{-1}$-Lipschitz with respect to the Hamming distance. Under algebraically $\varphi$-mixing process, we may apply \cite[Theorem 8]{mohri2010stability} with $(M_{m,1}^{(j)},\dots,M_{m,k}^{(j)})$, we obtain with probability at least $1-\exp\{-2t^2 k / ||\Delta_k||_{\infty}^2\}$ where $||\Delta_k||_{\infty}\leq 1 + 4 \sum_{i=1}^k \varphi(i)$
	\begin{align*}
		\underset{x \in \mathbb{R}}{\sup} \left| \hat{F}_{n,m}^{(j)}(x) - F_{m}^{(j)}(x) \right| & \leq \mathbb{E}\left[ \underset{x \in \mathbb{R}}{\sup} \; \left|\beta_{n,m}^{(j)}(x) \right| \right] + \frac{t}{3}   \leq C_1 k^{-1/2} + C_2 k^{-1} + \frac{t}{3}.
	\end{align*}
	Thus, for a sufficiently large $C_3$, with probability at most $\exp\{-t^2k / C_3\}$
	\begin{equation*}
	    	\underset{x \in \mathbb{R}}{\sup} \left| \hat{F}_{n,m}^{(j)}(x) - F_{m}^{(j)}(x) \right| \geq C_1 k^{-1/2} + k^{-1} + \frac{t}{3}.
	\end{equation*}
	Using Bonferroni inequality
	\begin{equation*}
		\mathbb{P}\left\{ E_3 \geq \frac{t}{3} \right\} \leq d \mathbb{P} \left\{ \underset{x \in \mathbb{R}}{\sup} \left| \hat{F}_{n,m}^{(j)}(x) - F_{m}^{(j)}(x) \right| \geq t \right\},
	\end{equation*}
	we thus obtain a control bound for $E_3$. Assembling all the controls obtained for $E_1$, $E_2$ and $E_3$, we obtain the desired result. 
%	\begin{equation*}
%		\left| \hat{\nu}_{n,m} - \nu_m \right| \geq C_1 k^{-1/2} + C_2 k^{-1} + t \quad \textrm{ with probability at most } (d+2\sqrt{e}) \exp\left\{-\frac{t^2 k}{C_3} \right\},
%	\end{equation*}
%	where $C_1$, $C_2$ and $C_2$ are constants depending only on $\zeta$ and $\lambda$.

\end{proof}

The proof of Theorem \ref{thm:grow_dim} needs the following results : (1) an upper bound over the quantity $|\hat{\theta}_{n,m}(a,b) - \theta_m(a,b)|$ with respect to $|\hat{\nu}_{n,m}(a,b) - \nu_m(a,b)|$ to use the concentration inequality introduced in Proposition \ref{prop:concentration_inequality}, (2) exhibit an event such that $\{\hat{O} = \bar{O}\}$.  Lemmas \ref{lem:upper_bounds_ext_coeff} and \ref{lem:exact_recovery}   below address these two questions. Then, taking benefits of these results, we show that the probability of the exhibited event such that $\{\hat{O} = \bar{O}\}$ holds with high probability, as stated in  Theorem \ref{thm:grow_dim}.
 
\begin{lemma}
    \label{lem:upper_bounds_ext_coeff}
    Consider a pair $(a,b) \in \{1,\dots,d\}^2$, the following inequality holds:
    \begin{equation*}
        |\hat{\theta}_{n,m}(a,b) - \theta_m(a,b)| \leq 9 |\hat{\nu}_{n,m}(a,b) - \nu_m(a,b)|.
    \end{equation*}
\end{lemma}
\begin{proof}[Proof of Lemma \ref{lem:upper_bounds_ext_coeff}]
    We may write the respective quantities as $\theta = f(\nu(a,b))$ and $\hat{\theta}_{n,m} = f(\hat{\nu}_{n,m}(a,b))$ where $f$ is a function defined as follows,  
    \begin{equation*}
        \begin{array}{ccccc}
f & : & [0,1/6] & \to & [1,2] \\
& & x & \mapsto & \frac{1/2 + x}{1/2 - x}, \\
        \end{array}
    \end{equation*}
    with $f(x) \in [1,2]$ by definition of the pre-asymptotic extremal coefficient $\theta_m$. The domain of this function is restricted to the interval $[0,1/6]$ because we have $f(x) \leq 2$, or
    \begin{equation*}
        x + \frac{1}{2} \leq 1 - 2 \, x,
    \end{equation*}
    which holds if $x \leq 1/6$. The inequality $f(x) \geq 1$ gives the positivity of the domain. In particular, $x < 1/2$ and thus $2^{-1} - x \geq 3^{-1} > 0$. Taking derivative of $f$, we find that
    \begin{equation*}
        |f'(x)| = \frac{1}{(1/2 - x)^2} \leq 3^2, \quad x \in \left[0, 1/6\right].
    \end{equation*}
    Therefore, $f$ is $9$-Lipschitz continuous and we have
    \begin{equation*}
        |\hat{\theta}_{n,m}(a,b) - \theta_m(a,b)| = \left| f(\hat{\nu}_{n,m}(a,b)) - f(\nu_m(a,b)) \right| \leq 9 |\hat{\nu}_{n,m}(a,b) - \nu_m(a,b)|.
    \end{equation*}
    This completes the proof.
\end{proof}

\begin{lemma} \label{lem:exact_recovery}
    Consider the AI-block model in Definition \ref{def:AI_block_models}. Define  
    \begin{equation*}
        \kappa = \underset{a,b \in \{1,\dots,d\}}{\sup} \, |\hat{\chi}_{n,m}(a,b) - \chi(a,b)|.
    \end{equation*}
    Consider parameters $(\tau,\eta)$ fulfilling
    \begin{equation}
        \label{eq:condition_consistent_recovery}
        \tau \geq \kappa, \quad \eta \geq \kappa + \tau.
    \end{equation}
    If $\MECO(\mathcal{X}) > \eta$,  then Algorithm \ref{alg:rec_pratic} yields $\hat{O} = \bar{O}$.
    
 \end{lemma}

\begin{proof}[Proof of Lemma \ref{lem:exact_recovery}]
    If $a \overset{\bar{O}}{\not \sim} b$, then $\chi(a,b) = 0$ and 
    \begin{equation*}
        \hat{\chi}_{n,m}(a,b) = \hat{\chi}_{n,m}(a,b) - \chi(a,b) \leq \kappa \leq \tau.
    \end{equation*}
    Now, if $a \overset{\bar{O}}{\sim} b$, if $\mathcal{X} \in \mathbb{X}(\eta)$ then $\chi(a,b) > \kappa + \tau$ and 
    \begin{equation*}
        \kappa + \tau < \chi(a,b) - \hat{\chi}_{n,m}(a,b) + \hat{\chi}_{n,m}(a,b),
    \end{equation*}
    and thus $\hat{\chi}_{n,m}(a,b) > \tau$. 
    In particular, under \eqref{eq:condition_consistent_recovery} and the separation condition $\MECO(\mathcal{X}) > \eta$, we have 
    \begin{equation}
        \label{eq:equivalence_exact_recovery}
    a \overset{\bar{O}}{ \sim } b \quad \Longleftrightarrow \quad \hat{\chi}_{n,m}(a, b) > \tau.
    \end{equation}
    Let us prove the   lemma  by induction on the algorithm step $l$. We consider the algorithm at some step $l-1$ and assume that the algorithm was consistent up to this step, i.e. $\hat{O}_j = \bar{O}_j$ for $j = 1,\dots,l-1$.
    
    If $\hat{\chi}_{n,m}(a_l,b_l) \leq \tau$, then according to \eqref{eq:equivalence_exact_recovery}, no $b \in S$ is in the same group of $a_l$. Since the algorithm has been consistent up to this step $l$, it means that $a_l$ is a singleton and $\hat{O}_l = \{a_l\}$.
    
    If $\hat{\chi}_{n,m}(a_l,b_l) > \tau$, then $a_l \overset{\bar{O}}{\sim} b$ according to \eqref{eq:equivalence_exact_recovery}. Furthermore, the equivalence implies that $\hat{O}_l = S \cap \bar{O}_l$. Since the algorithm has been consistent up to this step, we have $\hat{O}_l = \bar{O}_l$. To conclude, the algorithm remains consistent at the step $l$ and the  result  follows by induction.
    
\end{proof}
	
\begin{proof}[Proof of Theorem \ref{thm:grow_dim}]
    We have that for $t>0$ :
    \begin{equation*}
        \mathbb{P}\left\{\underset{a,b \in \{1,\dots,d\}}{\sup} \, |\hat{\theta}_{n,m}(a,b) - \theta_m(a,b)| \geq t \right\} \leq d^2 \mathbb{P}\left\{|\hat{\theta}_{n,m}(a,b) - \theta_m(a,b)| \geq t \right\}.
    \end{equation*}
    With probability at least $1 - 2(1+\sqrt{e})d^2 \exp \{-t^2 k /C_3\}$, and by using Proposition \ref{prop:concentration_inequality} and Lemma \ref{lem:upper_bounds_ext_coeff}, one has
    
    \begin{equation*}
        \underset{a,b \in \{1,\dots,d\}}{\sup} \, \left| \hat{\theta}_{n,m}(a,b) - \theta(a,b) \right| \leq d_m + C_1 k^{-1/2} + C_2 k^{-1} + t,
    \end{equation*}
    
    By considering $\delta \in ]0,1[$ and solve the following equation
    
    \begin{equation*}
        \frac{\delta}{d^2} = 2(1+\sqrt{e}) \exp\left\{ - \frac{k t^2}{C_3}\right\},
    \end{equation*}
    
    with respect to $t$ gives that the event
    
    \begin{equation*}
        \underset{a,b \in \{1,\dots,d\}}{\sup} \, \left| \hat{\theta}_{n,m}(a,b) - \theta(a,b) \right| \geq d_m +C_1k^{-1/2} + C_2 k^{-1} + C_3 \sqrt{\frac{1}{k} \ln\left( \frac{2(1+\sqrt{e}) d^2}{\delta} \right) },
    \end{equation*}
    
    is of probability at most $\delta$. Now, taking $\delta = 2(1+\sqrt{e})d^{-2\gamma}$, with $\gamma > 0$, we have
    
    \begin{equation*}
       \underset{a,b \in \{1,\dots,d\}}{\sup} \, \left| \hat{\theta}_{n,m}(a,b) - \theta(a,b) \right| \leq d_m +C_1k^{-1/2} + C_2 k^{-1} + C_3 \sqrt{\frac{(1+\gamma)\ln(d)}{k}},
    \end{equation*}
    
    with probability at least $1 - 2(1+\sqrt{e})d^{-2\gamma}$ for $C_3$ sufficiently large. The result then follows from Lemma \ref{lem:exact_recovery} along with Condition \ref{ass:A} and algebraically $\varphi$-mixing random process, since
    
    \begin{equation*}
        \mathbb{P}\left\{\kappa \leq d_m +C_1k^{-1/2} + C_2 k^{-1} + C_3 \sqrt{\frac{(1+\gamma)\ln(d)}{k}} \right\} \geq 1-2(1+\sqrt{e})d^{-2\gamma},
    \end{equation*}
    
    and $\MECO(\mathcal{X}) > \eta$ by assumption.
\end{proof}

{Therein, we prove the argument that were stated without proof in the paragraph next to Theorem \ref{thm:grow_dim}. A condition of order two were introduced and we have state that $d_m = O(\Psi_m)$ can be shown. We propose a proof of this statement below.}

\begin{proof}[Proof of $d_m = O(\Psi(m))$]
Take $a \neq b$ fixed, we have, using Lemma \ref{lem:upper_bounds_ext_coeff}
\begin{equation*}
    \left| \chi_m(a,b) - \chi(a,b) \right| = \left| \theta_m(a,b) - \theta(a,b) \right| \leq 9\left| \nu_m(a,b) - \nu(a,b) \right|,
\end{equation*}
where $\nu_m(a,b)$ (resp. $\nu(a,b)$) is the madogram computed between $M_{m}^{(a)}$ and $M_{m}^{(b)}$ (resp. between $X^{(a)}$ and $X^{(b)}$) and we use Lemma \ref{lem:upper_bounds_ext_coeff} to obtain the inequality. Using the results of Lemma 1 of \cite{marcon2017multivariate}, we have
\begin{align*}
    \nu_m(a,b) - \nu(a,b) &= \frac{1}{2}\left(\int_{[0,1]} (C_m-C_\infty)(\textbf{1},x^{(a)},\textbf{1}) dx^{(a)} + \int_{[0,1]} (C_m-C_\infty)(\textbf{1},x^{(b)},\textbf{1}) dx^{(b)}\right) \\ &- \int_{[0,1]} (C_m-C_\infty)(1,\dots,\underbrace{x}_{\textrm{$a$th index}},1,\dots,1,\underbrace{x}_{\textrm{$b$th index}},\dots,1)dx,
\end{align*}
where the integration is taken respectively for the $a$-th, $b$-th and $a$,$b$-th components. Hence
\begin{align*}
    \left| \nu_m(a,b) - \nu(a,b) \right| &\leq \frac{1}{2} \int_{[0,1]} |(C_m-C_\infty)(\textbf{1},x^{(a)},\textbf{1})| dx^{(a)} \\ &+ \frac{1}{2}\int_{[0,1]} |(C_m-C_\infty)(\textbf{1},x^{(b)},\textbf{1})| dx^{(b)} \\ &+ \int_{[0,1]} |(C_m-C_\infty)(1,\dots,\underbrace{x}_{\textrm{$a$th index}},1,\dots,1,\underbrace{x}_{\textrm{$b$th index}},\dots,1)|dx.
\end{align*}

Using the second order condition in Equation \eqref{eq:second_order} we obtain that $|C_m - C_\infty|(\textbf{u}) = O(\Psi_m)$, uniformly in $\textbf{u} \in [0,1]^d$. Hence the statement.
\end{proof}

Now, we prove the theoretical result giving support to our cross validation process.
\begin{proof}[Proof of Proposition \ref{prop:cv}]
{Using triangle inequality several times, we may obtain the following bound}
\begin{align*}
    \widehat{\SECO}_{n,m}(\bar{O}) - \widehat{\SECO}_{n,m}(\hat{O}) &\leq 2 D_m + |\widehat{\SECO}_{n,m}(\bar{O}) - \SECO_m(\bar{O})| \\ &+ |\widehat{\SECO}_{n,m}(\hat{O}) - \SECO_m(\hat{O})| + \SECO(\bar{O}) - \SECO(\hat{O}) \\ &=: 2 D_m + E_1 + E_2 + \SECO(\bar{O}) - \SECO(\hat{O}) .
\end{align*}
{Taking expectancy, we now have}
\begin{align*}
    \mathbb{E}[\widehat{\SECO}_{n,m}(\bar{O}) - \widehat{\SECO}_{n,m}(\hat{O})] &\leq 2 D_m + \mathbb{E}[E_1] + \mathbb{E}[E_2] + \SECO(\bar{O}) - \SECO(\hat{O}).
\end{align*}
{Using the same tool involved in the proof of Lemma \ref{lem:upper_bounds_ext_coeff}, we can show}
\begin{equation*}
    |\hat{\theta}_{n,m}^{(\bar{O}_g)} - \hat{\theta}_m^{(\bar{O}_g)}| \leq (d_g + 1)^2 |\hat{\nu}_{n,m}^{(\bar{O}_g)} - \hat{\nu}_m^{(\bar{O}_g)}|,
\end{equation*}
{Thus, using concentration bounds in Proposition \ref{prop:concentration_inequality}, there exists a universal constant $K_1 > 0$ independent of $n,k,m,t$ such that}
\begin{equation*}
    \mathbb{P}\left\{ |\hat{\theta}_{n,m}^{(\bar{O}_g)} - \hat{\theta}_m^{(\bar{O}_g)}| \geq t \right\} \leq d_g \exp\left\{  - \frac{t^2 k}{K_1 d_g^4}\right\}.
\end{equation*}
{Now,}
\begin{align*}
    \mathbb{P}\left\{ |\widehat{\SECO}_{n,m}(\bar{O}) - \SECO_m(\bar{O})| \geq t \right\} &\leq \sum_{g=1}^G\mathbb{P}\left\{ |\hat{\theta}_{n,m}^{(\bar{O}_g)} - \hat{\theta}_m^{(\bar{O}_g)}| \geq \frac{t}{G} \right\} \\ &\leq d \exp\left\{  - \frac{t^2 k }{K_1 G^2 \vee_{g=1}^G d_g^4}\right\}
\end{align*}
{Thus, for every $\delta > 0$, one obtains}
\begin{equation*}
    \mathbb{E}[E_1]^2 \leq \mathbb{E}[E_1^2] \leq \delta + \int_{\delta}^{\infty} \mathbb{P}\left\{ E_1 > t^{1/2} \right\} dt \leq \delta + d \int_{\delta}^{\infty} \exp\left\{  -\frac{t}{2\sigma^2}\right\}dt,
\end{equation*}
{where $\sigma^2 = \frac{K_1 G^2 \vee_{g=1}^G d_g^4}{2k}$. Set $\delta = 2 \sigma^2 \ln(d)$, we can obtain}

\begin{equation*}
    \mathbb{E}[E_1]^2 \leq \delta + 2\sigma^2 = c^2 \frac{\ln(d)G^2 \vee_{g=1}^G d_g^4}{k}
\end{equation*}
{with $c > 0$. Same results hold for $\mathbb{E}[E_2]$ with corresponding sizes, thus}
\begin{align*}
    \mathbb{E}[\widehat{\SECO}_{n,m}(\bar{O}) - \widehat{\SECO}_{n,m}(\hat{O})] &\leq 2\left(D_m + c\sqrt{\frac{\ln(d)}{k}} \max(G,I) \max(\vee_{g=1}^G d_g^2, \vee_{i=1}^I d_i^2) \right) \\ &+ \SECO(\bar{O}) - \SECO(\hat{O}),
\end{align*}
{which is strictly negative by assumption.}
\end{proof}

\subsection{Proofs of Section \ref{sec:example}}
\label{subsec:example_proofs}
In the following we prove that the  model introduced  in Section \ref{sec:example} is in the domain of attraction of an AI-block model. This comes down from some elementary algebra where the fundamental argument is given by \cite[Proposition 4.2]{bucher2014extreme}, from which the inspiration for the model was drawn thereof.
	
\begin{proof}[Proof of Proposition \ref{prop:domain_attraction_repetition_model}]
We aim to show that the following quantity
\begin{align*}
	\bigg| &D\left(D^{(O_1)}(\{\textbf{u}^{(O_1)}\}^{1/m}; \theta, \beta_1), \dots, D(\{\textbf{u}^{(O_G)}\}^{1/m}; \theta, \beta_G) ; \theta, \beta_0 \right)^m \\ &- D\left(D^{(O_1)}(\textbf{u}^{(O_1)}; \beta_1), \dots, D^{(O_G)}(\textbf{u}^{(O_G)}; \beta_G); \beta_0\right)\bigg|,
\end{align*}
converges to $0$ uniformly in $\textbf{u} \in [0,1]^d$. Using  Equation \eqref{eq:clayton_powerm} in the main article, the latter term is equal to
\begin{align*}
	E_{0,m} :=& \bigg| D\left(D^{(O_1)}(\textbf{u}^{(O_1)}; \theta/m,\beta_1)^{1/m}, \dots, D^{(O_G)}(\textbf{u}^{(O_G)}; \theta/m,\beta_G)^{1/m}; \theta,\beta_0\right)^m \\ &- D\left(D^{(O_1)}(\textbf{u}^{(O_1)}; \beta_1), \dots, D^{(O_G)}(\textbf{u}^{(O_G)}; \beta_G) ; \beta_0\right)\bigg|.
\end{align*}
Thus
\begin{align*}
	E_{0,m} \leq &\bigg| D\left(D^{(O_1)}(\textbf{u}^{(O_1)}; \theta/m,\beta_1)^{1/m}, \dots, D^{(O_G)}(\textbf{u}^{(O_G)}; \theta/m,\beta_G)^{1/m}; \theta,\beta_0\right)^m \\ &- D\left(D^{(O_1)}(\textbf{u}^{(O_1)}; \theta/m,\beta_1), \dots, D^{(O_G)}(\textbf{u}^{(O_G)}; \theta/m,\beta_G); \beta_0\right)\bigg|\\
	&+ \bigg|D\left(D^{(O_1)}(\textbf{u}^{(O_1)}; \theta/m,\beta_1), \dots, D^{(O_G)}(\textbf{u}^{(O_G)}; \theta/m,\beta_G); \beta_0\right) \\ &- D\left(D^{(O_1)}(\textbf{u}^{(O_1)}; \beta_1), \dots, D^{(O_G)}(\textbf{u}^{(O_G)}; \beta_G) ; \beta_0\right)\bigg|
	\\ &=: E_{1,m} + E_{2,m}.
\end{align*}

As $D(\cdot ;\theta/m, \beta_0)$ converges uniformly to $D(\cdot,\beta_0)$, then, uniformly in $\textbf{u} \in [0,1]^d$, 
$E_{1,m} \underset{m \rightarrow \infty}{\longrightarrow} 0.$
Now, using Lipschitz property of the copula function, one has
\begin{equation*}
	E_{2,m} \leq \sum_{g=1}^G \left| D^{(O_g)}(\textbf{u}^{(O_g)}; \theta/m, \beta_g) - D^{(O_g)}(\textbf{u}^{(O_g)}; \beta_g) \right|,
\end{equation*}
which converges almost surely to $0$ as $m \rightarrow \infty$. The limiting copula is an extreme value copula by $\beta_0 \leq \min\{\beta_1,\dots,\beta_G\}$, see Example 3.8 of \cite{hofert2018hierarchical}. Hence the result.
\end{proof}

\section{Additional results}
\label{sec:aux_results}

\subsection{Additional results of Section \ref{sec:variable_clust}}

\label{sec:aux_results_section_variable_clust}
% Debut

{Let $\textbf{Z} \geq 0$ be a random vector, and for simplicity, let's assume that it has heavy-tailed marginal distributions with a common tail-index $\alpha > 0$. There are two distinct yet closely related classical approaches for describing the extreme values of the multivariate distribution of $\textbf{Z}$.}

{The first approach focuses on scale-normalized componentwise maxima:}
\begin{equation*}
{c_n^{-1} \bigvee_{i=1}^n \textbf{Z}_i,}
\end{equation*}
{where $\textbf{Z}_i$ are independent copies of $\textbf{Z}$, and $c_n$ is a scaling sequence. The limiting results are typically derived under the assumption of independence for the sake of consistency. However, they hold under more general conditions, such as mixing conditions (see, e.g., \cite{hsing1989extreme}). The only possible limit laws for such maxima are max-stable distributions with the following distribution function:}
\begin{equation*}
{\underset{n \rightarrow \infty}{\lim} \mathbb{P} \left\{ \bigvee_{i=1}^n \textbf{Z}_i \leq c_n \textbf{u} \right\} = e^{-\Lambda([0,\textbf{u}]^c)}, \quad \textbf{u} \in \mathbb{R}^d+ \setminus {\boldsymbol{0}},}
\end{equation*}
{where the exponent measure $\Lambda$ is $(-\alpha)$-homogeneous.}

{The second approach examines the distribution of scale-normalized exceedances:}
\begin{equation*}
{u^{-1} \, \textbf{Z} \,| \bigvee_{j=1}^d Z^{(j)} > u,}
\end{equation*}
{which considers conditioning on the event that at least one component $Z^{(j)}$ exceeds a high threshold $u$. The only possible limits of these peak-over-thresholds as $u \rightarrow \infty$ are multivariate Pareto distributions (\cite{rootzen2006multivariate}). The probability laws of these distributions are induced by a homogeneous measure $\Lambda$ on the set $\mathcal{L} = E \setminus [0,1]^d$, where $E = [0,\infty)^d \setminus {\boldsymbol{0}}$. The probability measure takes the form:}
\begin{equation*}
{\mathbb{P}_{\mathcal{L}}(dy) = \frac{\Lambda(dy)}{\Lambda(\mathcal{L})}.}
\end{equation*}
The exponent measure serves as a clear connection between these two approaches, as it characterizes the distribution function for both cases. In fact, the connection arises from a fundamental limiting result that establishes a link between the two approaches through regular variation. This result has been elegantly presented in Theorem 2.1.6 and Equation (2.3.1) in \cite{kulik2020heavy}. As in the main text, let us denote by $\textbf{X}$ the random vector with extreme value distribution $H(\textbf{x}) = e^{-\Lambda(E \setminus [0,\textbf{x}])}$. The following proposition provides the form of the exponent measure when the random vectors $\textbf{X}^{(O_1)}, \dots, \textbf{X}^{(O_G)}$ are independent, and it establishes the connection between AI-block models for the two approaches.
\begin{proposition}
	\label{rm:rv_independent}
	Suppose $\normalfont{\textbf{X}}$ is a random vector having extreme value distribution H with exponent measure $\Lambda$ concentrating on $E \setminus [0,\normalfont{\textbf{x}}]$ where $E = [0,\infty)^d \setminus \{\boldsymbol{0}\}$ and $\normalfont{\textbf{x}} > \normalfont{\textbf{0}}$. The following properties are equivalent:
		\begin{enumerate}[label=\textcolor{frenchblue}{(\roman*)}]
		\item The vectors $\normalfont{\textbf{X}}^{(O_1)}, \dots, \normalfont{\textbf{X}}^{(O_G)}$ are independent. \label{prop_(i):mutual_independence}
		\item The vectors are blockwise independent: for every $1 \leq g < h \leq G$
		\begin{equation*}
			\normalfont{\textbf{X}}^{(O_g)} \textrm{ and } \normalfont{\textbf{X}}^{(O_h)}, \textrm{ are independent random vectors.} \label{prop_(ii):blockwise}
		\end{equation*}
		\item The exponent measure $\Lambda$ concentrates on
		\begin{equation}
			\bigcup_{g=1}^G \{ \boldsymbol{0} \}^{d_1} \times \dots \times ]0,\infty[^{d_g} \times \dots \times \{ \boldsymbol{0} \}^{d_G},
			\label{eq:support_measure}
		\end{equation}
		so that for $\normalfont{\textbf{x}} > \normalfont{\textbf{0}}$, 
		\begin{equation*}
            \Lambda \left( \bigcup_{1 \leq g < h \leq G} \left\{ \normalfont{\textbf{y}} \in E, \exists a \in O_g,  \exists b \in O_h, y^{(a)} > x^{(a)}, y^{(b)} > x^{(b)} \right\} \right) = 0.
        \end{equation*}
  \label{prop_(iii):measure}
		\end{enumerate}
    \end{proposition}
	These conditions generalize  straightforwardly those stated in Proposition 5.24 of \cite{resnick2008extreme} (see Exercise 5.5.1 of the book aforementioned or the Lemma in \cite{strokorb2020extremal}).
\begin{proof}[Proof of Proposition \ref{rm:rv_independent}]
We will establish the result proceeding as $\ref{prop_(iii):measure} \implies \ref{prop_(i):mutual_independence} \implies \ref{prop_(ii):blockwise}  \implies \ref{prop_(iii):measure}$ where we directly have $\ref{prop_(i):mutual_independence} \implies \ref{prop_(ii):blockwise}$. Now for $\ref{prop_(iii):measure} \implies \ref{prop_(i):mutual_independence}$, suppose $\Lambda$ concentrates on the set \eqref{eq:support_measure}. Then for $\textbf{x} > 0$, noting $A_g(\textbf{x}) = \{ \textbf{y} \in E, \exists a \in O_g, y^{(a)} > x^{(a)} \}$ for $g \in \{1,\dots,G\}$, we obtain
\begin{align*}
	-\ln H(\textbf{x}) &= \Lambda(E \setminus [0,\textbf{x}]) = \Lambda \left( \bigcup_{g=1}^G A_g(\textbf{x}) \right) \\ &= \sum_{g=1}^G \Lambda(A_g(\textbf{x})) + \sum_{g=2}^G (-1)^{g+1} \sum_{1 \leq i_1 < i_2 < \dots < i_l \leq G} \Lambda(A_{i_1}(\textbf{x}) \cap \dots \cap A_{i_l}(\textbf{x})),
\end{align*}
so that because of Equation \eqref{eq:support_measure}, 
\begin{equation*}
	-\ln H(\textbf{x}) = \sum_{g=1}^G \Lambda(A_g(\textbf{x})),
\end{equation*}
and we have $H(\textbf{x}) = \Pi_{g=1}^G \exp \left\{ - \Lambda\left( \{ \textbf{y} \in E, \exists a \in O_g, y^{(a)} > x^{(a)} \} \right) \right\} = \Pi_{g=1}^G H^{(O_g)}(\textbf{x}^{(O_g)}).$
 
Thus $H$ is a written as a product of the $G$ distributions corresponding to random vectors $\textbf{X}^{(O_1)}, \dots, \textbf{X}^{(O_G)}$, as desired.

It remains to show $\ref{prop_(ii):blockwise} \implies \ref{prop_(iii):measure}$. Set $Q^{(O_g)}(\textbf{x}^{(O_g)}) = - \ln \mathbb{P}\{ \textbf{X}^{(O_g)} \leq \textbf{x}^{(O_g)} \}$ for $g \in \{1,\dots,G\}$. We have for $\textbf{x} > 0$ that blockwise independence implies, with $ g\neq h$,
\begin{equation*}
	Q^{(O_g)}(\textbf{x}^{(O_g)}) + Q^{(O_h)}(\textbf{x}^{(O_h)}) = -\ln \mathbb{P}\{ \textbf{X}^{(O_g)} \leq \textbf{x}^{(O_g)}, \textbf{X}^{(O_h)} \leq \textbf{x}^{(O_h)} \}.
\end{equation*}
Since $H(\textbf{x}) = \exp \{ -\Lambda (E \setminus [\boldsymbol{0},\textbf{x}] ) \}$ for $\textbf{x} > 0$, we have

\begin{align*}
	Q^{(O_g)}(\textbf{x}^{(O_g)}) + Q^{(O_h)}(\textbf{x}^{(O_h)}) &= \Lambda(\{\textbf{y}, \exists a \in O_g, y^{(a)} > x^{(a)}\} \cup \{\textbf{y}, \exists b \in O_h, y^{(b)} > x^{(b)}\}) \\
	&= \Lambda(\{\textbf{y}, \exists a \in O_g, y^{(a)} > x^{(a)}\}) + \Lambda(\{\textbf{x}, \exists b \in O_h, y^{(b)} > x^{(b)}\}) \\ & - \Lambda(\{\textbf{y}, \exists a \in O_g, \exists b \in O_h, y^{(a)} > x^{(a)}, y^{(b)} > x^{(b)}\}) \\
	& = Q^{(O_g)}(\textbf{x}^{(O_g)}) + Q^{(O_h)}(\textbf{x}^{(O_h)}) \\ &- \Lambda(\{\textbf{y}, \exists a \in O_g, \exists b \in O_h, y^{(a)} > x^{(a)}, y^{(b)} > x^{(b)}\}),
\end{align*}

and thus
\begin{equation*}
\Lambda(\{\textbf{y}, \exists a \in O_g, \exists b \in O_h, y^{(a)} > x^{(a)}, y^{(b)} > x^{(b)}\}) = 0,
\end{equation*}
so that \ref{prop_(iii):measure} holds. This is equivalent to $\Lambda$ concentrates on the set in Equation \eqref{eq:support_measure}.
\end{proof}
If $\textbf{X}$ is a random vector with multivariate extreme value distribution $H$ then its extreme value copula, denoted as, $C_\infty$ is written as:
    \begin{equation*}
        C_\infty(\textbf{u}) = \exp\left\{-L\left(-\ln(u^{(1)}), \dots, -\ln(u^{(d)})\right)\right\},
    \end{equation*}
where $L$ is the stable tail dependence function. This function captures the tail dependence structure of the random vector and can be expressed as a specific integral with respect to the exponent measure (we refer to Section 8 of \cite{beirlant2004statistics}). In the context of AI-block models, the tail dependence function takes the following form:
\begin{equation}
\label{eq:mutual_indep_stdf}
{L\left(z^{(1)}, \dots, z^{(d)}\right) = \sum_{g=1}^G L^{(O_g)}\left( \textbf{z}^{(O_g)} \right), \quad \textbf{z} \in [0,\infty)^d,}
\end{equation}
where $L^{(O_1)}, \dots, L^{(O_G)}$ are the corresponding stable tail dependence functions with copulae $C_\infty^{(O_1)}, \dots, C_\infty^{(O_G)}$, respectively. This model is a specific form of the nested extreme value copula, as mentioned in the remark below and discussed in further detail in \cite{hofert2018hierarchical}.
\begin{remark}
Equation \eqref{eq:mutual_indep_stdf} can be rewritten as
	\begin{equation*}
L(\normalfont{\textbf{z}}) = L_{\Pi} \left(L^{(O_1)}\left(z^{(O_1)}\right), \dots, L^{(O_G)}\left(z^{(O_G)}\right) \right),
	\end{equation*}
where $L_{\Pi}(z^{(1)},\dots,z^{(G)}) = \sum_{g=1}^G z^{(g)}$ is a stable tail dependence function corresponding to asymptotic independence. According to Proposition \ref{prop:Cevt}, $C_\infty$ is an extreme value copula. Therefore, it follows that $C_\infty$, which has the representation
	\begin{equation*}
	C_\infty(\normalfont{\textbf{u}}) = C_\Pi \left(C_\infty^{(O_1)}(\normalfont{\textbf{u}}^{(O_1)}), \dots, C_\infty^{(O_G)}(\normalfont{\textbf{u}}^{(O_G)}) \right), \quad C_\Pi = \Pi_{g=1}^G u^{(g)},
	\end{equation*}
	is also a nested extreme value copula, as defined in \cite{hofert2018hierarchical}.
\end{remark}
Equation \eqref{eq:mutual_indep_stdf} can be restricted to the simplex, allowing us to express the stable tail dependence function in terms of the Pickands dependence function. Specifically, the Pickands dependence function $A$ can be written as a convex combination of the Pickands dependence functions $\mathpzc{A}^{(O_1)}, \dots, \mathpzc{A}^{(O_G)}$ as follows:
\begin{align}
	\mathpzc{A}(t^{(1)}, \dots, t^{(d)}) &= \frac{1}{z^{(1)}+\dots+z^{(d)}} \left[ \sum_{g=1}^G (z^{(i_{g,1})} + \dots + z^{(i_{g,d_g})}) \mathpzc{A}^{(O_g)}(\textbf{t}^{(O_g)}) \right] \nonumber \\ 
	&= \sum_{g=1}^G w^{(O_g)}(\textbf{t}) \mathpzc{A}^{(O_g)}(\textbf{t}^{(O_g)}) =: \mathpzc{A}^{(O)}(t^{(1)}, \dots, t^{(d)}), \label{eq:conv_pick}
\end{align}
with $t^{(j)} = z^{(j)} / (z^{(1)}+\dots+z^{(d)})$ for $j \in \{2,\dots, d\}$ and $t^{(1)} = 1- (t^{(2)} + \dots + t^{(d)})$, $w^{(O_g)}(\textbf{t})  = (z^{(i_{g,1})} + \dots + z^{(i_{g,d_g})}) / (z^{(1)} + \dots + z^{(d)})$ for $g \in \{2, \dots, G\}$ and $w^{(O_1)}(\textbf{t}) = 1-(w^{(O_2)}(\textbf{t})+\dots+w^{(O_G)}(\textbf{t}))$, $\textbf{t}^{(O_g)} = (t^{(i_{g,1})},\dots, t^{(i_{g,d_g})})$ where $t^{(i_{g,\ell})} = z^{(i_{g,\ell})} / (z^{(i_{g,1})} + \dots + z^{(i_{g,d_g})})$ and $(i_{g,\ell})$ designates the $\ell$th variable in the $g$th cluster for $\ell \in \{1, \dots, d_g\}$ and $g \in \{1, \dots, G\}$. As a convex combination of Pickands dependence functions, $\mathpzc{A}$ is itself a Pickands dependence function (see \cite[Page 123]{falk2010laws}).

In the context of independence between extreme random variables, it is well-known that the inequality $\mathpzc{A}(\textbf{t}) \leq 1$ holds for $\textbf{t} \in \Delta_{d-1}$, where $\mathpzc{A}$ is the Pickands dependence function and equality stands if and only if the random variables are independent. This result extends to the case of random vectors, with the former case being a special case where $d_1 = \dots = d_G = 1$.
\begin{proposition}
	\label{prop:ineq}
	Consider a random vector $\normalfont{\textbf{X}} \in \mathbb{R}^d$ with copula $C_\infty$ and Pickands dependence function $\mathpzc{A}$. Let $\mathpzc{A}^{(O)}$ be as defined in \eqref{eq:conv_pick}. For all $\normalfont{\textbf{t}} \in \Delta_{d-1}$, we have:
	\begin{equation*}
		\left(\mathpzc{A}^{(O)} - \mathpzc{A}\right) (\normalfont{\textbf{t}}) \geq 0,
	\end{equation*}
	with equality if and only if $\normalfont{\textbf{X}}^{(O_1)}, \dots, \normalfont{\textbf{X}}^{(O_G)}$ are independent.
\end{proposition}
We provide two methods for establishing this result: the first leverages the convexity and homogeneity of order one of the stable tail dependence function, while the second takes advantage of the associativity of random vectors having extreme value distribution $H$.
\begin{proof}[Proof of Proposition \ref{prop:ineq}]
	For the first method, the stable tail dependence function $L$ is subadditive as an homogeneous convex function under a cone, i.e., $$L(\textbf{x} + \textbf{y}) \leq L(\textbf{x}) + L(\textbf{y}),$$ for every $\textbf{x}, \textbf{y} \in [0, \infty)^d$. In particular, we obtain by induction on $G$ $$L\left(\sum_{g=1}^G \textbf{x}^{(g)}\right) \leq \sum_{g=1}^G L(\textbf{x}^{(g)}),$$ where $\textbf{x}^{(g)} \in [0, \infty)^d$ and $g \in \{1,\dots, G\}$. Consider now $\textbf{z}^{(O_g)} = (\textbf{0}, z^{(i_{g,1})}, \dots, z^{(i_{g,d_g})}, \textbf{0})$, we directly obtain using the equation above
\begin{equation*}
	L(\textbf{z}) = L\left(\sum_{g=1}^G \textbf{z}^{(O_g)}\right) \leq \sum_{g=1}^G L(\textbf{z}^{(O_g)}) = \sum_{g=1}^G L^{(O_g)} (z^{(i_{g,1})}, \dots, z^{(i_{g,d_g})}).
\end{equation*}
Translating the above inequality in terms of Pickands dependence function results on
\begin{align*}
	\mathpzc{A}(\textbf{t}) &\leq \sum_{g=1}^G \frac{1}{z^{(1)} + \dots + z^{(d)}} L^{(O_g)}(z^{(i_{g,1})}, \dots, z^{(i_{g,d_g})}) \\ &= \sum_{g=1}^G \frac{z^{(i_{g,1})} + \dots + z^{(i_{g,d_g}})}{z^{(1)} + \dots + z^{(d)}} \mathpzc{A}^{(O_g)}(t^{(i_{g,1})}, \dots, t^{(i_{g,d_g})}),
\end{align*}
where $t^{(i)} = z^{(i)} / (z^{(1)} + \dots + z^{(d)})$. Hence the result.

We can also prove this result by using the associativity of extreme-value distributions (see \cite[Proposition 5.1]{marshall1983} or \cite[Section 5.4.1]{resnick2008extreme}), i.e.,
\begin{equation*}
	\mathbb{E}\left[f(\textbf{X}) g(\textbf{X}) \right] \geq \mathbb{E}\left[ f(\textbf{X})\right] \mathbb{E} \left[g(\textbf{X}) \right],
\end{equation*}
for every increasing (or decreasing) functions $f,g$. By induction on $G \in \mathbb{N}_*$, 
\begin{equation}
	\label{eq:association}
	\mathbb{E}\left[\Pi_{g=1}^G f^{(g)}(\textbf{X}) \right] \geq \Pi_{g=1}^G \mathbb{E}\left[ f^{(g)}(\textbf{X}) \right].
\end{equation}	
Take $f^{(g)}(\textbf{x}) = \mathds{1}_{\{]-\infty, \textbf{x}^{(O_g)}]\}}$ for each $g \in \{1, \dots G\}$, thus Equation \eqref{eq:association} gives
\begin{equation*}
	C(H^{(1)}(x^{(1)}), \dots, H^{(d)}(x^{(d)})) \geq \Pi_{g=1}^G C^{(O_g)} \left(H^{(O_g)}\left(\textbf{x}^{(O_g)}\right) \right),
\end{equation*}
which can be restated in terms of stable tail dependence function as 
\begin{equation*}
	L(\textbf{z}) \leq \sum_{g=1}^G L^{(O_g)}(\textbf{z}^{(O_g)}).
\end{equation*}
We obtain the statement expressing this inequality with Pickands dependence function. Finally, notice that \eqref{eq:association} with $f^{(g)}(\textbf{x}) = \mathds{1}_{\{]-\infty, \textbf{x}^{(O_g)}]\}}$ for each $g \in \{1, \dots G\}$ holds as an equality if and only if $\textbf{X}^{(O_1)},\dots,\textbf{X}^{(O_G)}$ are independent random vectors. 
\end{proof}

In the following paragraph, we give another proof of the extension of the results found in \cite{Takahashi1987SomePO, takahashi1994asymptotic} made by \cite[Proposition 2.1]{FERREIRA2011586}. Before going into details, we recall some useful expression of the dependence structure of extreme closely related to the notion of regular variation.

{Let $\textbf{X}$ be a regularly varying random vector in $\mathbb{R}^d_+$ with exponent measure $\Lambda$ which is $(-\alpha)$-homogeneous, i.e. for $y > 0$ and $A$ separated from $\boldsymbol{0}$, that is there exists an open set $U$ such that $\boldsymbol{0} \in U$ and $U^c \subset A$, we have}
\begin{equation*}
{\Lambda(yA) = y^{-\alpha} \Lambda(A).}
\end{equation*}

{Using the homogeneity of the exponent measure, we may define a probability measure $\Phi$ on $\Theta = S_d \cap [\boldsymbol{0}, \boldsymbol{\infty})$ where $S_d = \{ \textbf{x} \in \mathbb{R}^d, \,||\textbf{x}|| = 1 \}$ called the spectral measure associated to the norm $||\cdot||$ and defined by}
\begin{equation*}
	\Phi(B) = \Lambda \left( \textbf{z} \in E : || \textbf{z} || > 1 , \, \textbf{z} ||\textbf{z}||^{-1} \in B \right)
\end{equation*}
for any Borel subset $B$ of $\Theta$ (for a proper introduction to these notions, see \cite[Section 5.1]{resnick2008extreme} or \cite[Section 2.2]{kulik2020heavy}). The measure $\Phi$ is called the spectral measure. It is uniquely determined by the exponent measure $\Lambda$ and the chosen norm. The homogeneity of $\Lambda$ implies :
\begin{equation*}
 \Lambda \left( \textbf{z} \in E : || \textbf{z} || > r , \, \textbf{z} ||\textbf{z}||^{-1} \in B \right) = r^{-1} \Phi(B),
\end{equation*}
for $0 < r < \infty$.

\begin{proposition}
    \label{prop:takahashi}
Let $\normalfont{\textbf{X}}$ be a regularly varying random vector in $\mathbb{R}^d_+$ with exponent measure $\Lambda$. Consider $O = \{O_1,\dots,O_g\}$ be a partition of $\{1,\dots,d\}$, then the following are equivalent:
    \begin{enumerate}[label=\textcolor{frenchblue}{(\roman*)}]
        \item  Let $\Lambda^{(O_g)}$ be the restriction of the exponent measure to $\mathbb{R}^{(O_g)}_+$, we have
        \begin{equation*}
            \Lambda = \sum_{g=1}^G \delta_0 \otimes \dots \otimes \Lambda^{(O_g)} \otimes \dots \otimes \delta_0.
        \end{equation*} \label{takashi_i}
        \item The spectral measure $\Phi$ associated to the exponent measure $\Lambda$ verifies
        \begin{equation}
            \label{eq:takahashi_phi_pi}
             \Phi = \sum_{g=1}^G \delta_0 \otimes \dots \otimes \Phi^{(O_g)} \otimes \dots \otimes \delta_0 =: \Phi_\Pi,
        \end{equation}
         where $\Phi^{(O_g)}(B) := \Phi(\Theta^{(O_g)} \cap B)$ where $B$ is a borel set of $\Theta$ and $$\Theta^{(O_g)} = \left\{ \normalfont{\textbf{w}} \in \Theta, \; w^{(j)} > 0 \textrm{ if and only if } j \in O_g \right\}$$ for $g \in \{1,\dots,G\}$. \label{takashi_ii}
        \item There exists a $\normalfont{\textbf{v}} \in (0,\infty)^d$ such that
        \begin{equation}
            \label{eq:takahashi_identity}
          \int_{\Theta} \bigvee_{j=1}^d w^{(j)} v^{(j)} \Phi(d \normalfont{\textbf{w}}) = \sum_{g=1}^G \int_{\Theta^{(O_g)}} \bigvee_{j \in O_g} w^{(j)} v^{(j)} \Phi^{(O_g)}(d \normalfont{\textbf{w}}^{(O_g)}).
        \end{equation} \label{takashi_iii}
    \end{enumerate}    
    \end{proposition}
    \begin{proof}[Proof of Proposition \ref{prop:takahashi}]
        The equivalence between \ref{takashi_i} and \ref{takashi_ii} falls down from definitions. The implication \ref{takashi_ii} $\implies$ \ref{takashi_iii} is trivial. We show now \ref{takashi_iii} $\implies$ \ref{takashi_ii}  Notice that for every Borel set $B$ of $\Theta$, we have
        \begin{equation*}
          \Phi(B) = \sum_{g=1}^G \Phi(B \cap \Theta^{(O_g)}) + \Phi\left(B \cap (\Theta \setminus \cup_{g=1}^G \Theta^{(O_g)}) \right) \geq  \sum_{g=1}^G \Phi(B \cap \Theta^{(O_g)}) = \Phi_\Pi(B).
        \end{equation*}
       The identity in Equation \eqref{eq:takahashi_identity} can be rewritten as
        \begin{equation*}
     		\int_{\Theta} \bigvee_{j=1}^d w^{(j)} v^{(j)} (\Phi - \Phi_{\Pi})(d\textbf{w}) = 0.
        \end{equation*}
       From above, we know that $(\Phi - \Phi_{\Pi})$ defined a positive measure. For every Borel set $B$ of $\Theta$, we have
        \begin{equation*}
           \int_B \bigvee_{j=1}^d w^{(j)} v^{(j)} (\Phi - \Phi_{\Pi})(d\textbf{w}) \leq  \int_{\Theta} \bigvee_{j=1}^d w^{(j)} v^{(j)} (\Phi - \Phi_{\Pi})(d\textbf{w}) = 0.
        \end{equation*}
   Since the function $\textbf{w} \mapsto \bigvee_{j=1}^d w^{(j)}v^{(j)}$ is strictly positive, continuous and defined on a compact set, we have that $ \bigvee_{j=1}^d w^{(j)}v^{(j)} \geq c$ for a certain constant $c$ strictly positive and we obtain
        \begin{equation*}
             c (\Phi - \Phi_{\Pi})(B) \leq \int_B \bigvee_{j=1}^d w^{(j)} v^{(j)} (\Phi - \Phi_{\Pi})(d\textbf{w}) = 0.
        \end{equation*}
     The following identity is obtained 
        \begin{equation*}
            \Phi(B) = \Phi_\Pi(B),
        \end{equation*}
     since $B$ is taken arbitrary from the Borelian of $\Theta$, we conclude.
    \end{proof}
 One can notice that the integrals defined in \eqref{eq:takahashi_identity} can be rewritten with the help of stable tail dependence function, that is
    \begin{equation*}
    L\left(v^{(1)}, \dots, v^{(d)}\right) = \sum_{g=1}^G L^{(O_g)}\left( \textbf{v}^{(O_g)} \right), \quad \textbf{v} \in [0,\infty)^d,
\end{equation*}
   since for every $\textbf{v} \in [0,\infty)^d$
    \begin{equation*}
        L(\textbf{v}) = \int_{\Theta} \bigvee_{j=1}^d w^{(j)}v^{(j)} \Phi(d\textbf{w}).
    \end{equation*}
    \subsection{Additional results of Section \ref{sec:estimation}}
    \label{subsec:asympto}

    To establish the strong consistency of the estimator $\hat{\nu}_{n,m}$ in \eqref{eq:est_mado}, certain conditions on the mixing coefficients must be satisfied.
	\begin{Assumption}{$\mathcal{C}$}
	    \label{ass:consistency}
	   	Let $m_n = o (n)$. The series $\sum_{n \geq 1} \beta(m_n)$ is convergent, where $\beta$ is defined in Section \ref{eq:beta_mixing}.
	\end{Assumption}
	For the sake of notational simplicity, we will write $m = m_n$, $k = k_n$. The convergence of the series of $\beta$-mixing coefficients in Condition \ref{ass:consistency} is necessary to obtain the strong consistency of $\hat{\nu}_{n,m}$, and it can be achieved through the sufficiency condition of the Glivencko-Cantelli lemma for almost sure convergence.
	\begin{proposition}
	\label{prop:strong_consistency}
		Let $(\normalfont{\textbf{Z}_t}, t \in \mathbb{Z})$ be a stationary multivariate random process. Under Conditions \ref{ass:domain} and \ref{ass:consistency}, the madogram estimator in \eqref{eq:est_mado} is strongly consistent, i.e., 
		\begin{equation*}
		\left| \hat{\nu}_{n,m} - \nu \right| \overunderset{a.s.}{n \rightarrow \infty}{\longrightarrow} 0,
		\end{equation*}
		with $\nu$ the theoretical madogram of the random vector $\normalfont{\textbf{X}}$ with copula $C_\infty$ given in \eqref{eq:theoretical_mado}.
	\end{proposition}
    Let $C_{n,m}^o$ be the empirical estimator of the copula $C_m$ based on the (unobservable) sample $(U_{m,1}^{(j)},\dots,U_{m,k}^{(j)})$ for $j \in \{1,\dots,d\}$. The proof of Proposition \ref{prop:strong_consistency} will use twice Lemma \ref{lem:gc}, which shows that $||C_{n,m}^o - C||_{\infty}$ converges almost surely to $0$. The proof of this lemma is postponed to \ref{subsec:gc_copula} of supplementary results.

\begin{proof}[Proof of Proposition \ref{prop:strong_consistency}]
	We aim to show the following convergence
	\begin{equation*}
		|\hat{\nu}_{n,m} - \nu| \overunderset{\textrm{a.s.}}{n \rightarrow \infty}{\longrightarrow} 0.
	\end{equation*}
    Following Lemma A.1 of \cite{marcon2017multivariate}, we can show that
	\begin{equation*}
		\hat{\nu}_{n,m}^o - \nu = \phi(C_{n,m}^o - C_\infty),
	\end{equation*}
	where $\hat{\nu}_{n,m}^o$ given in \eqref{eq:mado_oracle} and $\phi : \ell^{\infty}([0,1]^d) \rightarrow \ell^\infty(\Delta_{d-1})$, $f \mapsto \phi(f)$ defined by
	\begin{equation*}
		\phi(f) = \frac{1}{d} \sum_{j=1}^d \int_{[0,1]} f(1,\dots,1,\underbrace{u}_{\textrm{$j$-th component}},1,\dots,1)du - \int_{[0,1]} f(u,\dots,u)du.
	\end{equation*}
	 Using Conditions \ref{ass:domain} and \ref{ass:consistency}, by Lemma \ref{lem:gc} in  
\ref{subsec:gc_copula},  as $||C_{n,m}^o - C_\infty||_{\infty}$ converges almost surely to $0$, we obtain that
	\begin{equation}
		\label{proof:sc_1}
		\left| \hat{\nu}_{n,m}^o - \nu \right| \overunderset{\textrm{a.s.}}{n \rightarrow \infty}{\longrightarrow} 0.
	\end{equation}
	Furthermore, using the chain of inequalities and  again Lemma \ref{lem:gc} in  
\ref{subsec:gc_copula},
	\begin{align*}
		\left| \hat{\nu}_{n,m} - \hat{\nu}^o_{n,m} \right| &\leq 2 \underset{j \in \{1,\dots,d\}}{\sup} \underset{x \in \mathbb{R}}{\sup} \left| \hat{F}_{n,m}^{(j)}(x) - F_{m}^{(j)}(x) \right| \\ &\leq 2 \underset{j \in \{1,\dots,d\}}{\sup} \underset{u \in [0,1]}{\sup} \left| \frac{1}{k} \sum_{i=1}^k \mathds{1}_{ \{ U_{m,i}^{(j)} \leq u \}} - u \right|.
	\end{align*}
Then 	we obtain that
	\begin{equation}
		\label{proof:sc_2}
		\left| \hat{\nu}_{n,m} - \hat{\nu}^o_{n,m} \right| \overunderset{\textrm{a.s.}}{n \rightarrow \infty}{\longrightarrow} 0.
	\end{equation}
	Now, write
	\begin{equation*}
		\left| \hat{\nu}_{n,m} - \nu \right| \leq \left| \hat{\nu}_{n,m} - \nu_{n,m}^o \right| + \left| \hat{\nu}_{n,m}^o - \nu \right|,
	\end{equation*}
	and use Equations \eqref{proof:sc_1} and \eqref{proof:sc_2} to obtain the statement.
\end{proof}
The strong consistency of the madogram in Proposition \ref{prop:strong_consistency} could be extended to the $\alpha$-mixing case. We present here the strong consistency of our procedure when the dimension $d$ is fixed  the sample size $n$ grows at infinity. The main technicality of the proof has already been tackled in Proposition \ref{prop:strong_consistency} and we state the precise formulation of this theorem below.

\begin{theorem}
	    \label{thm:consistency_asymptotic}
	   Consider the AI-block model as defined in Definition \ref{def:AI_block_models} under Condition \ref{ass:A} and $(\normalfont{\textbf{Z}_t}, t \in \mathbb{Z})$ be a stationary multivariate random process. For a given $\mathcal{X}$ and its corresponding estimator $\hat{\mathcal{X}}$, if Conditions \ref{ass:domain}, \ref{ass:consistency} holds, then {taking $\tau = 0$}
	   \begin{equation*}
		\underset{n \rightarrow \infty}{\lim} \, \mathbb{P}\left\{ \hat{O} = \bar{O} \right\} = 1.
	\end{equation*}	
\end{theorem}

\begin{proof}[Proof of Theorem \ref{thm:consistency_asymptotic}]  
    If $a$ and $b$ are not in the same cluster according to $\bar{O}$, i.e. $a \overset{\bar{O}}{\not \sim} b$, then $\chi(a,b) = 0$. Therefore, using Proposition \ref{prop:strong_consistency} along with Conditions \ref{ass:domain} and \ref{ass:consistency}, we can conclude that almost surely
    \begin{equation*}
        \underset{n \rightarrow \infty}{\lim} \, \hat{\chi}_{n,m}(a,b) = 0 \leq \tau.
    \end{equation*}
    Now, if $a \overset{\bar{O}}{\sim} b$, then $\chi(a,b) > 0$ and again by Propositions \ref{prop:strong_consistency} and Conditions \ref{ass:domain}, \ref{ass:consistency}, we obtain
    \begin{equation*}
         \underset{n \rightarrow \infty}{\lim} \, \hat{\chi}_{n,m}(a,b) = \chi(a,b) > 0,
    \end{equation*}
    where the the strict positiveness is obtain through Condition \ref{ass:A}, hence
    \begin{equation*}
        a \overset{\bar{O}}{\sim} b \iff \underset{n \rightarrow \infty}{\lim} \, \hat{\chi}_{n,m}(a,b) > \tau.
    \end{equation*}
    Let us prove Theorem \ref{thm:consistency_asymptotic} by induction on  the algorithm step $l$. We consider the algorithm at some step $l-1$ and assume that the algorithm was consistent up to this step, i.e. $\hat{O}_j = \bar{O}_j$ for $j = 1,\dots,l-1$.
		
	If $\underset{n \rightarrow \infty}{\lim} \hat{\chi}_{n,m}(a_l,b_l) = 0$, then no $b \in S$ is in the same group of $a_l$. Since the algorithm has been consistent up to this step $l$, it means that $a_l$ is a singleton and $\hat{O}_l = \{a_l\}$.
		
	If $\underset{n \rightarrow \infty}{\lim} \hat{\chi}_{n,m}(a_l,b_l) > \tau$, then $a_l \overset{\bar{O}}{\sim} b$. The equivalence above implies that $\hat{O}_l = S \cap \bar{O}_l$.  Since the algorithm has been consistent up until this step, we know that $\hat{O}_l = \bar{O}_l$. Therefore, the algorithm remains consistent at step $l$ with probability tending to one as $n \rightarrow \infty$, and Theorem \ref{thm:consistency_asymptotic} follows by induction.
\end{proof}

\section{Further results}
\label{sec:supp_mat}
  
\subsection{A usefull Glivenko-Cantelli result for the copula with known margins in a weakly dependent setting}
\label{subsec:gc_copula}
  
In this section, we will prove an important auxiliary result: the empirical copula estimator $\hat{C}_{n,m}^{o}$ based on the weakly dependent  sample $\textbf{U}_{m,1},\dots, \textbf{U}_{m,k}$ is uniformly strongly consistent towards the extreme value copula $C$. This result is a main tool to obtain important results in the paper such as Proposition \ref{prop:strong_consistency}, Theorem \ref{thm:consistency_asymptotic}.
 For that purpose, the Berbee's coupling lemma is of prime interest (see, e.g., \cite[Chapter 5]{rio2017asymptotic}) which gives an approximation of the original process by conveniently defined independent random variables.  

\begin{lemma}
\label{lem:gc}
	Under conditions of Proposition \ref{prop:strong_consistency}, we have
	\begin{equation*}
		||C_{n,m}^o - C ||_{\infty} \overunderset{\textrm{a.s.}}{n \rightarrow \infty}{\longrightarrow} 0.
	\end{equation*}
\end{lemma}
\begin{proof}[Lemma \ref{lem:gc}]
	Using triangle inequality, one obtain the following bound
	\begin{equation}
		\label{eq:decompositon_strong_consistency}
		||C_{n,m}^o - C ||_{\infty} \leq ||C_{n,m}^o - C_m ||_{\infty} + ||C_m - C ||_{\infty}.
	\end{equation}
	As $\{C_m, n \in \mathbb{N}\}$ is an equicontinuous class of functions (for every $m$, $C_m$ is a copula hence a 1-Lipschitz function), defined on the compact set $[0,1]^d$ (by Tychonov's theorem) which converges pointwise to $C$ by Condition \ref{ass:domain}. Then the convergence is uniform over $[0,1]^d$. Thus the second term of the RHS of Equation \eqref{eq:decompositon_strong_consistency} converges to $0$ almost surely.
	
	Now, let us prove that $||C_{n,m}^o - C_m ||_{\infty}$ converges almost surely to $0$. By Berbee's coupling lemma (see \cite[Theorem 6.1]{rio2017asymptotic} or \cite[Theorem 3.1]{bucher2014extreme} for similar applications), we can construct inductively a sequence $(\bar{\textbf{Z}}_{im+1},\dots, \bar{\textbf{Z}}_{im+m})_{i \geq 0}$ such that the following three properties hold:
	\begin{enumerate}[label=(\roman*)]
		\item $(\bar{\textbf{Z}}_{im+1},\dots, \bar{\textbf{Z}}_{im+m}) \overset{d}{=} (\textbf{Z}_{im+1},\dots, \textbf{Z}_{im+m})$ for any $i \geq 0$; \label{property:law}
		\item both $(\bar{\textbf{Z}}_{2im+1},\dots, \bar{\textbf{Z}}_{2im+m})_{i\geq 0}$ and $(\bar{\textbf{Z}}_{(2i+1)m+1},\dots, \bar{\textbf{Z}}_{(2i+1)m+m})_{i\geq 0}$ sequences are independent and identically distributed; \label{property:iid}
		\item $\mathbb{P}\{(\bar{\textbf{Z}}_{im+1},\dots, \bar{\textbf{Z}}_{im+m}) \neq (\textbf{Z}_{im+1}, \dots, \textbf{Z}_{im+m}) \} \leq \beta(m)$. \label{property:mixing}
	\end{enumerate}
	Let $\bar{C}_{n,m}^o$ and $\bar{\textbf{U}}_{m,i}$ be defined analogously to $C_{n,m}^o$ and $\textbf{U}_{m,i}$ respectively but with $\textbf{Z}_1,\dots, \textbf{Z}_n$ replaced with $\bar{\textbf{Z}}_1,\dots, \bar{\textbf{Z}}_n$. Now write
	\begin{equation}
		\label{eq:coupling}
		C_{n,m}^o(\textbf{u}) = \bar{C}_{n,m}^o(\textbf{u}) + \left\{ C_{n,m}^o(\textbf{u}) - \bar{C}_{n,m}^o(\textbf{u}) \right\}.
	\end{equation}
	We will show below that the term under brackets converges uniformly to $0$ almost surely. Write $\bar{C}_{n,m}^o(\textbf{u}) = \bar{C}_{n,m}^{o,\textrm{odd}}(\textbf{u}) + \bar{C}_{n,m}^{o,\textrm{even}}(\textbf{u})$ where $\bar{C}_{n,m}^{o,\textrm{odd}}(\textbf{u})$ and $\bar{C}_{n,m}^{o,\textrm{even}}(\textbf{u})$ are defined as sums over the odd and even summands of $\bar{C}_{n,m}^o(\textbf{u})$, respectively. Since both of these sums are based on i.i.d. summands by properties \ref{property:law} and \ref{property:iid}, we have $||\bar{C}^o_{n,m} - C_m ||_{\infty} \overunderset{\textrm{a.s.}}{n \rightarrow \infty}{\longrightarrow} 0$ using Glivenko-Cantelli (see \cite[Chapter 2.5]{bestbook}).
	
	It remains to control the term under brackets on the right hand side of Equation \eqref{eq:coupling}, we have that
	\begin{align*}
		\left| C_{n,m}^o(\textbf{u}) - \bar{C}_{n,m}^o(\textbf{u}) \right| &\leq \frac{1}{k} \sum_{i=1}^k \left| \mathds{1}_{ \{\bar{\textbf{U}}_{m,i} \leq \textbf{u} \}} - \mathds{1}_{\{ \textbf{U}_{m,i} \leq \textbf{u} \}} \right|  \\ &\leq \frac{1}{k} \sum_{i=1}^k \mathds{1}_{\{ (\bar{Z}_{im+1},\dots, \bar{Z}_{im+m}) \neq (Z_{im+1}, \dots, Z_{im+m}) \}}.
	\end{align*}
	Hence, using Markov's inequality and property \ref{property:mixing}, we have
	\begin{equation*}
		\mathbb{P}\left\{ \underset{u \in [0,1]^d}{\sup} \left| \bar{C}_{n,m}^o(\textbf{u}) - C_{n,m}^o(\textbf{u}) \right| > \epsilon \right\} \leq \frac{\beta(m)}{\epsilon}.
	\end{equation*}
	Thus by Condition \ref{ass:consistency},
	\begin{equation*}
	\sum_{n \geq 1 }\mathbb{P}\left\{ \underset{u \in [0,1]^d}{\sup} \left| \bar{C}_{n,m}^o(\textbf{u}) - C_{n,m}^o(\textbf{u}) \right| > \epsilon \right\} < \infty.
	\end{equation*}
	Applying Borel-Cantelli gives the desired convergence to $0$ almost surely of the term under bracket in Equation \eqref{eq:coupling}. Gathering all results gives that the term $||C_{n,m}^o - C_m ||_{\infty}$ converges almost surely to 0. Hence the statement using Equation \eqref{eq:decompositon_strong_consistency}.
\end{proof}

\subsection{Weak convergence of an estimator  of $\mathpzc{A}^{(O)} - \mathpzc{A}$}
\label{subsec:functional_central_limit_theorem}
 
We now state conditions on the block size $m$ and the number of blocks $k$, as in \cite{bucher2014extreme}, to demonstrate the weak convergence of the empirical copula process based on the (unobservable) sample $(U_{n,m,1}^{(j)}, \dots, U_{n,m,k}^{(j)})$ for every $j \in \{1,\dots,d\}$ under mixing conditions. An additional condition will be required within the theorem to establish the weak convergence of the rank-based copula estimator under the same mixing conditions.
	
\begin{Assumption}{$\mathcal{F}$}
    \label{ass:regularity}
    There exists a positive integer sequence $\ell_n$ such that the following statement holds:
    \begin{enumerate}[label=(\roman*)]
        \item $m_n \rightarrow \infty$ and $m_n = o (n)$ \label{ass:m=o(n)}
        \item $\ell_n \rightarrow \infty$ and $\ell_n = o(m_n)$
        \item $k_n \alpha(\ell_n) = o(1)$ and $(m_n / \ell_n) \alpha(\ell_n) = o(1)$ \label{ass:fidi}
        \item $\sqrt{k_n} \beta(m_n) = o(1)$ \label{ass:tightness} \label{ass:series}
    \end{enumerate}
\end{Assumption}
We recall that both $m$ and $k$ depends on $n$. Also, for notational convenience, we will write in the following $\ell_n = \ell$. Note that Condition \ref{ass:regularity} \ref{ass:fidi} guarantees that the limit $C$ is an extreme value copula by \cite[Theorem 4.2]{hsing1989extreme}. As usual, the weak convergence of the empirical copula process stems down from the finite dimensional convergence and the asymptotic tightness of the process which then hold from Condition \ref{ass:regularity} \ref{ass:fidi} and \ref{ass:tightness} respectively. In order to apply Hadamard's differentiability to obtain the weak convergence of the empirical copula based on the sample's scaled ranks, we need a classical condition over the derivatives of the limit copula stated as follows. 

\begin{Assumption}{$\mathcal{G}$}
\label{ass:smooth}
For any $j \in \{1,\dots,d\}$, the $j$th first order partial derivative $\dot{C}^{(j)} = \partial C / \partial u^{(j)}$ exists and is continuous on $\{ \normalfont{\textbf{u}} \in [0,1]^d, u^{(j)} \in (0,1)\}$.
\end{Assumption}

The estimator of the Pickands dependence function that we present is based on the madogram concept (\cite{cooley2006variograms, marcon2017multivariate}), a notion borrowed from geostatistics in order to capture the spatial dependence structure. Our estimator is defined as 
\begin{equation*}
	\hat{\mathpzc{A}}_{n,m}(\textbf{t}) = \frac{\hat{\nu}_{n,m}(\textbf{t}) + c(\textbf{t})}{1-\hat{\nu}_{n,m}(\textbf{t})-c(\textbf{t})},
\end{equation*}
where
\begin{align*}
	&\hat{\nu}_{n,m}(\textbf{t}) = \frac{1}{k} \sum_{i=1}^k \left[\bigvee_{j=1}^d \left\{ \hat{U}_{n,m,j}^{(j)} \right\}^{1/t^{(j)}} - \frac{1}{d} \sum_{j=1}^d \left\{ \hat{U}_{n,m,i}^{(j)} \right\}^{1/t^{(j)}} \right], \quad c(\textbf{t}) = \frac{1}{d} \sum_{j=1}^d \frac{t^{(j)}}{1+t^{(j)}},
\end{align*}
and $\hat{U}_{n,m,i}^{(j)} = \hat{F}_{n,m}^{(j)}(M_{m,i}^{(j)})$ corresponds to ranks scaled by $k^{-1}$. By convention, here $u^{1/0} = 0$ for $u \in (0,1)$. Let $g \in \{1,\dots,G\}$ and define 
\begin{equation*}
    \hat{\mathpzc{A}}_{n,m}^{(O_g)}\left(\textbf{t}^{(O_g)}\right) = \hat{\mathpzc{A}}_{n,m}\left(\textbf{0},\textbf{t}^{(O_g)}, \textbf{0}\right)
\end{equation*}
the empirical Pickands dependence function associated to the $k$-th subvector of $\textbf{X}$p. We consider the empirical process of the difference between estimates of the Pickands dependence functions of subvectors $\textbf{X}^{(O_g)}, g \in \{1,\dots,G\},$ and the estimator of the Pickands dependence function of $\textbf{X}$:
\begin{equation*}
	\mathcal{E}_{nG} (\textbf{t}) = \sqrt{k}\left(\hat{\mathpzc{A}}^{(O)}_{n,m}(\textbf{t}) - \hat{\mathpzc{A}}_{n,m}(\textbf{t})\right),
\end{equation*}
where $\hat{\mathpzc{A}}_{n,m}^{(O)}(\textbf{t}) = \sum_{g=1}^G w^{(O_g)}(\textbf{t}) \hat{\mathpzc{A}}_{n,m}^{(O_g)}(\textbf{t}^{(O_g)})$. Noticing that multiplying the above process by $d$ and taking $\textbf{t} = (d^{-1},\dots,d^{-1})$ gives 

\begin{equation*}
	\sqrt{k} \widehat{SECO}(O) = \sqrt{k}\left( \sum_{g=1}^G \hat{\theta}_{n,m}^{(O_g)} - \hat{\theta}_{n,m} \right).
\end{equation*}
	Hence, the weak convergence of the above empirical process will immediately comes  down from the one of the empirical process in $\mathcal{E}_{nG}$, as stated in the theorem below.
\begin{theorem}
	\label{thm:weak_conv}
	Consider the AI-block model in Definition \ref{def:AI_block_models} with a given partition $O$, i.e., $\mathpzc{A} = \mathpzc{A}^{(O)}$ where the latter is defined in Equation  \eqref{eq:conv_pick}. Under Conditions \ref{ass:domain}, \ref{ass:regularity}, \ref{ass:smooth} and $\sqrt{k}(C_m - C) \rightsquigarrow \Gamma$, the empirical process $\mathcal{E}_{nG}$ converges weakly in $\ell^{\infty}(\Delta_{d-1})$ to a tight Gaussian process having representation
	\begin{align*}
		\mathcal{E}_G(\normalfont{\textbf{t}}) &= \left(1+\mathpzc{A}(\normalfont{\textbf{t}})\right)^2 \int_{[0,1]} (N_{C_\infty} + \Gamma)(u^{t^{(1)}}, \dots, u^{t^{(d)}})du \\ &- \sum_{g=1}^G w^{(O_g)}(\normalfont{\textbf{t}})\left(1+\mathpzc{A}^{(O_g)}(\normalfont{\textbf{t}}^{(O_g)})\right)^2 \int_{[0,1]} (N_{C_\infty} + \Gamma)(\normalfont{\textbf{1}},u^{t^{(i_{g,1})}}, \dots, u^{t^{(i_{g,d_g})}}, \normalfont{\textbf{1}})du,
	\end{align*}
	where $N_{C_\infty}$ is a continuous tight Gaussian process with representation 
	\begin{equation*}
		N_{C_\infty}(u^{(1)},\dots,u^{(d)}) = B_{C_\infty}(u^{(1)}, \dots, u^{(d)}) - \sum_{j=1}^d \dot{C}_\infty^{(j)}(u^{(1)}, \dots, u^{(d)}) B_{C_\infty}(\normalfont{\textbf{1}}, u^{(j)}, \normalfont{\textbf{1}}),
	\end{equation*}
	and $B_{C_\infty}$ is a continuous tight Gaussian process with covariance function
	\begin{equation*}
		\Cov(B_{C_\infty}(\normalfont{\textbf{u}}), B_{C_\infty}(\normalfont{\textbf{v}})) = C_\infty(\normalfont{\textbf{u} \wedge \textbf{v}}) - C_\infty(\normalfont{\textbf{u}})C_\infty(\normalfont{\textbf{v}}) = C_{\Pi}(\normalfont{\textbf{u} \wedge \textbf{v}}) - C_{\Pi}(\normalfont{\textbf{u}})C_{\Pi}(\normalfont{\textbf{v}}),
	\end{equation*}
 where $C_{\Pi}(\textbf{u}^{(O_g)}) = \Pi_{g=1}^G C_\infty^{(O_g)}(\textbf{u}^{(O_g)})$.
\end{theorem}
\begin{proof}[Theorem \ref{thm:weak_conv}]
    The proof is straightforward, notice that by the triangle diagram in Figure~\ref{fig:diagram}
    \begin{equation*}
        \mathcal{E}_{nG} = \psi \circ \phi \left(\sqrt{k}(\hat{\mathpzc{A}}_{n,m} - \mathpzc{A}) \right),
    \end{equation*}
    where $\phi$ is detailed as
    \begin{equation*}
        \begin{array}{ccccc}
            \phi & : & \ell^{\infty}(\Delta_{d-1}) & \to & \ell^{\infty}(\Delta_{d-1}) \otimes (\ell^{\infty}(\Delta_{d-1}), \dots, \ell^{\infty}(\Delta_{d-1})) \\
            & & x & \mapsto & (x, \phi_1(x), \dots, \phi_G(x)), \\
        \end{array}
    \end{equation*}
    with for every $g\in \{1,\dots,G\}$
    \begin{equation*}
        \begin{array}{ccccc}
            \phi_g & : & \ell^{\infty}(\Delta_{d-1}) & \to & \ell^{\infty}(S_{d}) \\
            & & x & \mapsto & w^{(O_g)}(t^{(1)}, \dots, t^{(G)}) x(\textbf{0}, t^{(i_{g,1})},\dots, t^{(i_{g,d_g})}, \textbf{0}), \\
        \end{array}
    \end{equation*}
    and also
    \begin{equation*}
        \begin{array}{ccccc}
            \psi & : & \ell^{\infty}(\Delta_{d-1}) \otimes (\ell^{\infty}(\Delta_{d-1}), \dots, \ell^{\infty}(\Delta_{d-1})) & \to & \ell^{\infty}(\Delta_{d-1}) \\
            & & (x, \phi_1(x), \dots, \phi_G(x)) & \mapsto & \sum_{g=1}^G \phi_g(x) - x. \\
        \end{array}
    \end{equation*}
     \begin{figure}
    \centering
        \begin{tikzpicture}[node distance=2cm, auto]
            (0,0) \node(A) {$\sqrt{k}\left(\hat{\mathpzc{A}}_{n,m} - \mathpzc{A}\right)$};
            (5,0)\node(B) [right of=A, xshift = 1cm ] {$\mathcal{E}_{nG}$};
            \node (C) [below of=B] {$\left(\sqrt{k}\left(\hat{\mathpzc{A}}_{n,m} - \mathpzc{A}\right) ; w^{(O_1)}\sqrt{k}\left(\hat{\mathpzc{A}}_{n,m}^{(O_1)} - \mathpzc{A}^{(O_1)}\right), \dots, w^{(O_G)}\sqrt{k}\left(\hat{\mathpzc{A}}_{n,m}^{(O_G)} - \mathpzc{A}^{(O_G)}\right)\right)$};	
            \draw[->](A) to node {}(B);
            \draw[->](A) to node [left] {$\phi$}(C);
            \draw[->](C) to node [below right=0.5ex] {$\psi$}(B);
        \end{tikzpicture}
        \caption{Commutative diagram of composition of function.}
        \label{fig:diagram}
    \end{figure}
    
    The function $\phi_g$ is a linear and bounded function hence continuous for every $g$, it follows that $\phi$ is continuous  since each coordinate functions  is continuous.  As a linear and bounded function, $\psi$ is also a continuous function. Noticing that, 
    \begin{equation*}
        (C_m - C_\infty)(\textbf{1},u,\textbf{1}) = 0, \quad \forall n \in \mathbb{N},  
    \end{equation*}
  where $m$ is the block length for a sample size $n$. We thus have
    \begin{equation*}
        \sqrt{k}(C_m-C_\infty)(\textbf{1},u,\textbf{1}) \underset{n \rightarrow \infty}{\longrightarrow} 0.
        \end{equation*}
        Therefore $\Gamma(\textbf{1},u,\textbf{1}) = 0$. Combining this equality with Corollary 3.6 of \cite{bucher2014extreme} and the same techniques as in the proof of Theorem 2.4 in \cite{marcon2017multivariate}, we obtain along with Conditions \ref{ass:domain}, \ref{ass:regularity}, \ref{ass:smooth}
    \begin{equation*}		
     \sqrt{k}(\hat{\mathpzc{A}}_{n,m}(\textbf{t}) - \mathpzc{A}(\textbf{t})) \rightsquigarrow -\left(1+\hat{\mathpzc{A}}_{n,m}(\textbf{t}))\right)^2 \int_{[0,1]}(N_{C_\infty} + \Gamma)(u^{t^{(1)}}, \dots, u^{t^{(d)}}) du.
    \end{equation*}
    Applying the continuous mapping theorem for the weak convergence in $\ell^{\infty}(\Delta_{d-1})$ (Theorem 1.3.6 of \cite{bestbook}) leads the result.
\end{proof}

\end{document}